\input ./style/arxiv-general.cfg
\documentclass[aap,MSNbibl,dvips]{arximspdf}
\makeatletter
   \@ifpackageloaded{graphicx}{}{\usepackage{graphicx}}
\makeatother

\doi{10.1214/14-AAP1039} 
\volume{25}
\issue{4}
\pubyear{2015}
\firstpage{1936}
\lastpage{1992}
\docsubty{FLA}

\makeatletter
\newcommand{\lleft}{\left}
\newcommand{\rright}{\right}

\newcommand{\dbinom}[2]{\pmatrix{#1 \cr #2}}
\newtheorem{theorem}{Theorem}[section]
\newtheorem{lemma}[theorem]{Lemma}
\newtheorem{proposition}[theorem]{Proposition}
\newproclaim{definition}[theorem]{Definition}
\newtheorem{corollary}[theorem]{Corollary}

\def\uro{{U}^{\!\!\!\!\raise5pt\hbox{$\scriptstyle o$}}}
\makeatother

\begin{document}
\begin{frontmatter}

\title{Critical population and error threshold on the sharp peak
landscape for the Wright--Fisher model}
\runtitle{Critical population and error threshold}

\begin{aug}
\author[A]{\fnms{Rapha\"el}~\snm{Cerf}\corref{}\ead[label=e1]{rcerf@math.u-psud.fr}\ead[label=u1,url]{http://www.foo.com}}
\runauthor{R. Cerf}
\affiliation{Universit\'e Paris Sud and IUF}
\address[A]{Math\'ematique\\
Universit\'e Paris Sud\\
B\^atiment~425\\
91405 Orsay Cedex\\
France\\
\printead{e1}} 
\end{aug}

\received{\smonth{7} \syear{2012}}
\revised{\smonth{5} \syear{2014}}

%
\begin{abstract}
We pursue the task of developing a finite population counterpart
to Eigen's model.
We consider the classical Wright--Fisher model describing the evolution
of a population of size~$m$ of chromosomes of length~$\ell$
over an alphabet of cardinality $\kappa$.
The mutation probability per locus is~$q$.
The replication rate is
$\sigma>1$ for the master sequence and $1$ for the other sequences.
We study the equilibrium distribution of the process in the regime where
\begin{eqnarray*}
\ell&\to&+\infty,\qquad m\to+\infty,\qquad q\to0,
\\
{\ell q} &\to& a\in\,]0,+\infty[, \qquad\frac{m}{\ell}\to\alpha\in
[0,+\infty].
\end{eqnarray*}
We obtain an equation
$\alpha\psi(a)=\ln\kappa$
in the parameter space $(a,\alpha)$
separating the regime where the equilibrium population is totally
random from the
regime where a quasispecies is formed. We observe the existence of a critical
population size necessary for a quasispecies to emerge, and
we recover the finite population counterpart of the error threshold.
The result is the~twin brother of the corresponding result for
the Moran model. The proof is more complex, and it relies on the
Freidlin--Wentzell theory of random perturbations of dynamical
systems.
\end{abstract}

%
\end{frontmatter}

\section{Introduction}
In 1971,
Eigen studied a population of macro\-mo\-le\-cu\-les,
evolving under replication
and mutation \cite{EI1}.
He considered the situation where one specific sequence, called the
master sequence,
replicates faster than the others.
A fundamental discovery of Eigen is the existence of an error threshold.
If the mutation rate exceeds a critical value, called the error threshold,
then, at equilibrium, the population is completely random.
If the mutation rate is below the error threshold,
then, at equilibrium, the population contains a positive fraction of
the master
sequence and a cloud of mutants which are quite close to the master sequence.
This specific distribution of individuals is called a quasispecies.
Since then,
the notions of error threshold and quasispecies have been widely used
to understand
the evolution of populations in biology. However,
biological populations are finite, and Eigen's model cannot be directly applied
in this context because it is
formulated
for an infinite population
of macromolecules.
A crucial task is therefore to reformulate and to understand
the notions of error threshold and quasispecies in biological
models describing the
evolution of a finite population.

The Wright--Fisher model is one of the most studied models in
mathematical population genetics.
In this work, we apply
to a
basic Wright--Fisher model
the ideas presented in~\cite{CE} for
the Moran model, thereby pursuing
the task of developing a finite population counterpart
to Eigen's model.
Numerous works have attacked this issue \cite{AF,Deme,GI,Cas1,NS,WE}.
Using different techniques,
Saakian, Deem and Hu \cite{SAA1},
Park, Mu\~noz and Deem \cite{PEM},
Musso \cite{MUS} and
Dixit, Srivastava and Vishnoi \cite{DSV}
considered finite population models which approximate Eigen's
model when the population size goes to infinity.
These models are
variants or generalizations of the classical Wright--Fisher model
of population genetics.
The problem is to understand how the error threshold phenomenon
present in Eigen's model in the infinite population limit
shows up in the finite population model.
We refer to the introduction of \cite{CE} for a detailed
discussion of this question and
the heuristics guiding
our strategy.
We consider here
the classical Wright--Fisher model describing the evolution
of a population of size~$m$ of chromosomes of length~$\ell$
over an alphabet of cardinality $\kappa$.
The mutation probability per locus is~$q$.
The replication rate is
$\sigma>1$ for the master sequence and $1$ for the other sequences.
We study the equilibrium distribution of the process in the regime where
\begin{eqnarray*}
\ell &\to&+\infty,\qquad m\to+\infty,\qquad q\to0,
\\
{\ell q} &\to& a\in\,]0,+\infty[, \qquad\frac{m}{\ell}\to\alpha\in
[0,+\infty].
\end{eqnarray*}
We obtain an equation
$\alpha\psi(a)=\ln\kappa$
in the parameter space $(a,\alpha)$
separating the regime where the equilibrium population is totally
random from the
regime where a quasispecies is formed. We observe the existence of a critical
population size necessary for a quasispecies to emerge, and
we recover the finite population counterpart of the error threshold.
It is a classical fact that the Moran model and the
Wright--Fisher model have similar dynamics. Indeed, the main
result here is the~twin brother of the main result of~\cite{CE},
the only difference being the equation of~the critical curve.
While we could compute exactly the critical curve for the
Moran model, here the critical curve is defined through a variational
problem depending on the parameter~$a$.
Apart from this point, the scaling and the associated exponents
are the same in both cases.
This confirms a conjecture of~\cite{CE}, and it sustains the hope
that this kind of analysis is robust.

A potential application of the result concerns genetic algorithms.
Indeed, the Wright--Fisher model is identical to the genetic
algorithm without crossover.
In her Ph.D. thesis \cite{OCH},
Ochoa investigated the role of the error
threshold phenomenon for genetic algorithms, and
she concluded that there exists
a relationship between the optimal mutation rate and
the error threshold.
%
The result proved here
provides a theoretical basis for some heuristics to control
efficiently the genetic algorithms proposed in \cite{CGA}.

On the technical side, the
Wright--Fisher model is much more difficult to handle than
the Moran model.
In the Moran model, the estimates of the selection drift relied on
a birth and death model introduced by
Nowak and Schuster \cite{NS}.
In the
Wright--Fisher model, the bounding processes are more complicated;
they involve three dependent binomial laws.
As the size of the population grows, their transition probabilities
satisfy a large deviation principle, derived with the help
of the classical Cram\'er theorem.
In the set of the populations containing the master sequence, the
process can be seen as the random perturbation of a discrete dynamical
system.
This discrete dynamical system is simply
the sequence of the iterates of a rational map
$F\dvtx [0,1]\to[0,1]$. Depending on the parameters,
this map has either one stable fixed point or two fixed points,
one stable and the other unstable.
This opens the way to the application of the
general scheme developed by Freidlin and Wentzell \cite{FW} to
study the random perturbations of dynamical systems.
Originally,
Freidlin and Wentzell studied diffusion processes arising as
Brownian perturbations of a differential equation. These
processes are continuous time Markov processes with a continuous
state space.
However, their approach is robust, and it can be applied in other
contexts. Kifer \cite{KI,KID} reworked this theory in the
discrete time case. Unfortunately, our bounding processes do not
fit the hypothesis of Kifer's model, for the following
two reasons.
In Kifer's model,
the large deviation rate function of the transition
probabilities is not allowed to be infinite, and
the large deviation principle for
the transition probabilities
is assumed to be uniform with respect to the starting
point.
Certainly the general framework considered by Kifer could be
adjusted to include our case, with the help of
some relaxed hypothesis.
Yet in our case, we have only two attractors, one unstable
and one stable, and we need only two specific estimates
from the general theory, which is concerned
with a finite number of attractors of any type.
In fact, the kind of estimates we need have
been computed
in two
other works handling closely related models.
In an unpublished work \cite{TD} (transmitted to me
by courtesy of Gregory Morrow), Darden analyzed a
Wright--Fisher model with two alleles and no mutation
with the help of
the Freidlin--Wentzell theory.
What we have to do essentially is to obtain results analogous
to Darden for the model with mutations.
Morrow and Sawyer \cite{MS} considered a more general model of
Markov chains evolving in a convex subset of ${\mathbb R}^d$ around
one stable attractor. Our bounding processes
would fit this framework, were it for the uniform assumption
on the variance of the
transition probabilities. In our
case, this condition is violated close to the unstable
attractor~$0$. We can apply their results outside a neighborhood
of $0$, but this would lead to a messy construction.
It appears that, in any case, if we try to apply the results of Kifer
or of
Morrow and Sawyer, we have to make a specific study of our process
in the vicinity of the unstable fixed point~$0$.
In the end,
it seems
that the most efficient presentation
consists
in deriving from scratch
the required estimates, following the initial ideas of Freidlin and
Wentzell.
The techniques involved in the proof are classical and go
back to the seminal work of Freidlin and Wentzell.
However, there is an important simplifying feature in our case.
Indeed, the bounding processes are monotone.
This allows us to avoid uniform large deviation estimates
and to provide substantially simpler proofs.

We describe the model in the next section, and we present
the main result in Section~\ref{mainres}. The rest of the paper
is devoted to the proofs. The global strategy is identical to the
case of the Moran model.
The lumping is performed in Section~\ref{seclum}.
In Section~\ref{secmono}, we build a coupling and
we prove the monotonicity of the occupancy process.
This allows us to define simple bounding processes in
Section~\ref{secbounds}.
Section~\ref{bide}, which analyzes
the dynamics of the bounding processes, is much more complicated
than for the Moran model.
Section~\ref{disc} presents the estimates in the neutral region.
These estimates were derived in \cite{CE} for the Moran model, and
they can
be easily adapted to the Wright--Fisher model, so most of the proofs
are omitted.
\section{The Wright--Fisher model}\label{secmodel}
Let $\mathcal A$ be a finite alphabet, and let
\index{$\mathcal A$}
\index{$\kappa$}
$\kappa=\operatorname{card}\mathcal A$ be its cardinality.
Let $\ell\geq1
$\index{$\ell$} be an integer. We consider the space
${\mathcal A}^\ell$ of sequences of length $\ell$ over the
alphabet $\mathcal A$.
Elements of this space represent the chromosome of an haploid
individual, or equivalently its genotype.
In our model, all the genes have the same set of alleles, and each
letter of the alphabet $\mathcal A$ is a possible allele.
Typical examples are
${\mathcal A}=\{ A,T,G,C \}$ to model standard DNA, or
${\mathcal A}=\{ 0,1 \}$ to deal with binary sequences.
Generic elements of
${\mathcal A}^\ell$ will be denoted by the letters $u,v,w$.\index{$u,v,w$}
A population is an $m$-tuple of elements of
${\mathcal A}^\ell$.
Generic populations will be denoted by the letters
$x,y,z\index{$x,y,z$}$.
Thus a population $x$ is a vector
\[
x = \pmatrix{ x(1)
\cr
\vdots
\cr
x(m)} %
\]
whose components are chromosomes.
For $i\in\{ 1,\ldots,m \}$, we denote by
\[
x(i,1),\ldots,x(i,\ell)
\]
the letters of the sequence $x(i)$. This way a population $x$
can be represented as an array
\[
x = \pmatrix{ x(1,1)&\cdots&x(1,\ell)
\cr
\vdots& & \vdots
\cr
x(m,1)&\cdots&x(m,
\ell)} %
\]
of size $m\times\ell$ of elements of $\mathcal A$, the
$i$th line being the $i$th chromosome.
The evolution of the population is random and it is driven
by two antagonistic forces: replication and mutation.

\medskip\textit{Replication.}
The replication favors the development of fit chromosomes.
The fitness of a chromosome is encoded in a fitness function
\[
A\dvtx {\mathcal A}^\ell\to[0,+\infty[. \index{$A$} %
\]
With the help of the fitness function $A$, we define a selection function
$F\dvtx {\mathcal A}^\ell\times({\mathcal A}^\ell)^m\to[0,1]$
by setting
\begin{eqnarray*}
&& \forall u\in{\mathcal A}^\ell,\ \forall x\in\bigl({\mathcal
A}^\ell\bigr)^m
\qquad F(u,x) = \frac{A(u)}{
A(x(1))+\cdots+A(x(m))} \sum_{1\leq i\leq m}1_{x(i)=u}.
\end{eqnarray*}
The population $x$ being fixed, the values
$F(u,x)$, $u\in{\mathcal A}^\ell$, define a probability distribution
over ${\mathcal A}^\ell$.
The value
$F(u,x)$ is the probability of choosing~$u$ when sampling from the
population~$x$.

\medskip\textit{Mutation.}
The mutation mechanism is the same
for all the loci, and mutations occur independently.
We denote by $q\in\,]0,1-1/\kappa[\index{$q$}$ the probability that
a mutation occurs
at one particular locus. If a mutation occurs, then the letter is replaced
randomly by another letter, chosen uniformly over the
$\kappa-1$ remaining letters.
Mutations are rare, and the most likely outcome for a given letter is to
stay unaltered; this is why we impose that $q\leq1-1/\kappa$.
We encode this mechanism in a mutation matrix
\[
M(u,v),\qquad u,v\in{\mathcal A}^\ell, \index{$M(\cdot,\cdot)$}
\]
where $M(u,v)$ is the probability that the chromosome $u$ is transformed
by mutation into the chromosome $v$.
The analytical formula for
$M(u,v)$ is
\[
M(u,v) = \prod_{j=1}^\ell
\biggl((1-q){1}_{u(j)=v(j)} +\frac{q}{\kappa-1} {1}_{u(j)\neq v(j)}
\biggr).
\]

\textit{Transition matrix.}
We consider the classical Wright--Fisher model. In this model,
generations do not overlap.
The mechanism to build
a new generation is divided in two steps.
In the first step,
$m$ chromosomes are sampled with
replacement
from the population.
The sampling law is given by the selection function.
In the second step,
each chromosome
mutates according to the law specified by the mutation matrix.
For $n\geq0$, we denote by $X_n$ the $n$th
generation.
The Wright--Fisher model
is
the Markov chain $(X_n)_{n\in\mathbb N}\index{$X_n$}$
on the space
$ ({\mathcal A}^\ell)^m$
whose
transition matrix is given by
%
\begin{eqnarray*}
&& \forall n\in{\mathbb N},\ \forall x,y\in\bigl({\mathcal A}^\ell
\bigr)^m
\\
&& \qquad
P (X_{n+1}=y |
X_n=x ) = \prod_{1\leq i\leq m} \biggl( \sum
_{u\in{\mathcal A}^\ell} F (u,x ) M \bigl(u,y(i) \bigr) \biggr).
\end{eqnarray*}

\section{Main results}\label{mainres}
We present the main results in this section.

\medskip\textit{Sharp peak landscape.}
We will consider only the sharp peak landscape defined as follows.
We fix a specific sequence, denoted by $w^*$, called the wild type or the
master sequence.
Let $\sigma>1\index{$\sigma$}$ be a fixed real number.
The fitness function $A$ is given by
\[
\forall u\in{\mathcal A}^\ell\qquad A(u) = \cases{ 1, &\quad if $u
\neq w^*$,
\cr
\sigma, &\quad if $u=w^*$.} %
\]

\textit{Density of the master sequence.}
We denote by $N(x)\index{$N(x)$}$
the number of copies of the master sequence~$w^*$
present in the population $x$:
\[
N(x) = \operatorname{card} \bigl\{ i\dvtx  1\leq i\leq m, x(i)=w^* \bigr\}.
\]
We are interested in the expected
density of the master sequence in the steady state
distribution of the process, that is,
\[
\operatorname{Master}(\sigma,\ell,m,q) = \lim_{n\to\infty} E \biggl(
\frac{1}{m} N(X_n) \biggr) \index{$\operatorname{Master}$},
\]
%
as well as the variance
\[
\operatorname{Variance}(\sigma,\ell,m,q) = \lim_{n\to\infty} E
\biggl( \biggl( \frac{1}{m} N(X_n) -\operatorname{Master}(
\sigma,\ell,m,q) \biggr)^2 \biggr) \index{$\operatorname{Variance}$}.
\]
%
The ergodic theorem for Markov chains ensures that the
above limits exist.
We denote by $I(p,t)$ the rate function governing
the large deviations of the binomial law of parameter~$p\in[0,1]$,
given by
\[
\forall t\in[0,1]\qquad I(p,t) = t\ln\frac{t}{p} +(1-t)\ln
\frac{1-t}{1-p}.
\]
We define,
for
$a\in\,]0,+\infty[$,
\begin{eqnarray*}
\rho^*(a) &=& \cases{ \displaystyle\frac{\sigma e^{-a}-1}{\sigma-1},
&\quad if $\sigma
e^{-a}>1$,
\cr
0, &\quad if $\sigma e^{-a}\leq1$,}
\\
\psi(a) &=& \inf_{l\in{\mathbb N}} \inf\Biggl\{ \sum
_{k=0}^{l-1} I \biggl(\frac{\sigma\rho_k}{(\sigma-1)\rho_k+1},
\gamma_k \biggr)+ \gamma_k I \biggl(e^{-a},
\frac{\rho_{k+1}}{\gamma_k} \biggr)\dvtx
\\
&&\hspace*{35pt} \rho_0= \rho^*(a), \rho_{l}=0, \rho_k,
\gamma_k\in[0,1]\mbox{ for }0\leq k<l \Biggr\}.
\end{eqnarray*}
Since $I(p,0)=-\ln(1-p)$, we have
\[
\psi(a) \leq I \biggl(\frac{\sigma\rho^*(a)}{(\sigma-1)\rho^*(a)+1},0
\biggr) = \ln\frac{(\sigma-1)\rho^*(a)+1}{1-\rho^*(a)}.
\]
Thus the function $\psi$ is finite on $]0,\ln\sigma[$, and it
vanishes on
$[\ln\sigma,+\infty[$.
We will prove in Lemma~\ref{vpro} that $\psi$ is positive on
$]0,\ln\sigma[$.
%

\begin{theorem}\label{mainth}
We suppose that
%
\[
\ell\to+\infty,\qquad m\to+\infty,\qquad q\to0,
\]
in such a way that
\[
{\ell q} \to a\in\,]0,+\infty[, \qquad\frac{m}{\ell}\to\alpha\in
[0,+\infty].
\index{$a,\alpha$} %
\]
We have the following dichotomy:
\begin{itemize}
\item If $\alpha\psi(a)<\ln\kappa$, then
$\operatorname{Master} (\sigma,\ell,m,q ) \to0$.

\item If $\alpha\psi(a)>\ln\kappa$, then
$\operatorname{Master} (\sigma,\ell,m,q ) \to
\rho^*(a)$.
\end{itemize}

In both cases, we have
$\operatorname{Variance} (\sigma,\ell,m,q ) \to0$.
\end{theorem}

%
The statement of the theorem holds also in the case where $\alpha$ is
null or infinite, but $a$ must belong to $]0,+\infty[$.
This result is very similar to the result for the Moran model.
Therefore all the comments made for the Moran model apply here
as well.
The main difference is that the function $\psi(a)$ is more complicated.
While we could obtain an explicit formula in the case
of the Moran model, here the function~$\psi(a)$ is the solution of
a complicated variational problem.
The general structure of the proof is similar to the one for
the Moran model.
We use the lumping theorem to reduce the size of the state space.
We couple the lumped processes with different initial conditions.
The coupling for the occupancy process turns out to be monotone.
We construct then a lower and an upper process.
These processes behave like the original process
in the neutral region and like a
Wright--Fisher model with two alleles whenever the master sequence
is present in the population.
The dynamics of these models is analyzed with a specific implementation
of the Freidlin--Wentzell theory.
We compute estimates of the persistence time of the
master sequence, as well as its equilibrium density.
Although the results are similar to the case of the Moran model, this
part is much more technical in the case of the
Wright--Fisher model. Indeed, in the case of the Moran model, we
needed simply to estimate some explicit formula associated
to the birth and death model introduced by
Nowak and Schuster \cite{NS}.
The approach used here to handle the
Wright--Fisher model is quite robust, and it should work for
other variants of the model.
In the final section we analyze the discovery time of the master
sequence. This part is similar to the case of the Moran model.
It is even simpler, so most proofs are omitted.

\section{Lumping}\label{seclum}
We denote by $d_H$ the Hamming distance between two chromosomes
\[
\forall u,v\in{\mathcal A}^\ell\qquad d_H(u,v) =
\operatorname{card} \bigl\{ j\dvtx  1\leq j\leq\ell, u(j)\neq v(j) \bigr\}.
\index{$d_H$} %
\]
%
We define a function
$
{H}\dvtx {\mathcal A}^\ell\to
\{ 0,\ldots,\ell\}
\index{$H$}$
by setting
\[
\forall u\in{\mathcal A}^\ell\qquad H(u) = d_H
\bigl(u,w^* \bigr).
\]
We define further a
vector function
${\mathbb H}\dvtx { ({\mathcal A}^\ell)^m}\to
\{ 0,\ldots,\ell\}^m
\index{$\mathbb H$}
$
by setting
\[
\forall x = \pmatrix{ x(1)
\cr
\vdots
\cr
x(m)} \in\bigl({\mathcal
A}^\ell\bigr)^m \qquad{\mathbb H}(x) = \pmatrix{ H
\bigl(x(1) \bigr)
\cr
\vdots
\cr
H \bigl(x(m) \bigr)}. %
\]

\textit{Mutation.}
We state some results on the mutation matrix that have been
proved in~\cite{CE}.
The mutation matrix is lumpable with respect
to the function $H$.
Let $b,c\in\{ 0,\ldots,\ell\} $, and let $u\in{\mathcal A}^\ell
$ such that $H(u)=b$.
The sum
\[
\mathop{\sum_{w\in{{\mathcal A}^\ell}}}_{H(w)=c} M(u,w)
\]
does not depend on $u$ in $H^{-1}(\{ b \})$.
It is a function of $b$ and $c$ only, which we denote by
$M_H(b,c) \index{$M_H$}$.
The coefficient
$M_H(b,c)$ is equal to
\[
\mathop{\mathop{\sum_{0\leq k\leq\ell
-b}}_{0\leq l\leq b}}_{k-l=c-b}
{ \dbinom{\ell-b} {k}} {\dbinom{b} {l}} q^k 
(1-q
)^{\ell-b-k} \biggl(\frac{q}{\kappa-1} \biggr)^l \biggl(1-
\frac{q}{\kappa-1} \biggr)^{b-l}. %
\]

\textit{Replication.}
The\index{$A_H$} fitness function $A$ of the
sharp peak landscape can be factorized through $H$. If we define
\[
\forall b\in\{ 0,\ldots,\ell\} \qquad A_H(b) = \cases{ \sigma, &
\quad if $b=0$,
\cr
1, &\quad if $b\geq1$,}
\]
then we have
\[
\forall u\in{\mathcal A}^\ell\qquad A(u) = A_H
\bigl(H(u)\bigr).
\]
%

\textit{Distance process.}
We define the distance process
$(D_n)_{n\geq0}
$\index{$D_n$}
by
\[
\forall n\geq0\qquad D_n = {\mathbb H} (X_n ).
\]
As in \cite{CE}, it can be checked
that the Markov chain
$(X_n)_{n\geq0}$ is lumpable with respect to the partition of
$ ({\mathcal A}^\ell)^m$ induced by the map $\mathbb H$,
so that the distance process
$(D_n)_{n\geq0}$
is a genuine Markov chain.
Its
transition matrix $p_H$ is given by
\begin{eqnarray*}
&& \forall d,e\in\{ 0,\ldots,\ell\}^m
\\
&& \qquad
p_H (d,e ) = \prod_{1\leq i\leq m} \biggl( \sum
_{1\leq j\leq m} \frac{A_H(j)
M_H (d(j),e(i) )
}{
A_H(d(1))+\cdots+A_H(d(m))} \biggr).
\end{eqnarray*}

\textit{Occupancy process.}
We denote by ${\mathcal P}^m_{\ell+1}\index{${\mathcal P}^m_{\ell
+1}$}$ the set of
the ordered partitions of the integer $m$ in at most
$\ell+1$ parts,
\[
{\mathcal P}^m_{\ell+1} = \bigl\{ \bigl(o(0),\ldots,o(\ell)
\bigr)\in{\mathbb N}^{\ell+1}\dvtx  o(0)+\cdots+o(\ell)=m \bigr\}.
\]
These partitions are interpreted as occupancy distributions. The
partition
$(o(0), \ldots,o(\ell))$
corresponds to a population in which
$o(l)$ chromosomes are at Hamming distance $l$ from the
master sequence,
for any $l\in\{ 0,\ldots,\ell\}$.
Let $\mathcal O \index{${\mathcal O}$}$ be the map which associates to each
population $x$
its occupancy distribution ${\mathcal O}(x)= (o(x,0),\ldots,o(x,\ell))$,
defined by
\[
\forall l\in\{ 0,\ldots,\ell\}\qquad o(x,l) = \operatorname{card} \bigl
\{ i\dvtx  1
\leq i\leq m, d_H\bigl(x(i),w^*\bigr)=l \bigr\}.
\]
%
For $d\in\{ 0,\ldots,\ell\}^m$, we set
\[
o_H(d,l) = \operatorname{card} \bigl\{ i\dvtx  1\leq i\leq
m, d(i)=l \bigr\},
\]
and we define a map
${\mathcal O}_H\dvtx \{ 0,\ldots,\ell\}^m\to{\mathcal P}^m_{\ell+1}
\index{${\mathcal O}_H$}$
by setting
\[
{\mathcal O}_H(d) = \bigl(o_H(d,0),
\ldots,o_H(d,\ell)\bigr).
\]
%
We define the occupancy process
$(O_n)_{n\geq0} \index{$\Omega_n$}$ by setting
\[
\forall n\geq0\qquad O_n = {\mathcal O}(X_n) = {
\mathcal O}_H(D_n).
\]
As in \cite{CE}, it can be checked
that the Markov chain
$(D_n)_{n\geq0}$ is lumpable with respect to the partition of
$ {\{ 0,\ldots,\ell\}^m}$
induced by the map ${\mathcal O}_H$,
so that the occupancy process
$(O_n)_{n\geq0}$
is a genuine Markov chain.
Its
transition matrix $p_O$ is given by
%
\begin{eqnarray*}
\forall o,o'\in{\mathcal P}^m_{\ell+1}
\qquad
p_{O} \bigl(o,o' \bigr) &=& \prod
_{0\leq h\leq\ell} \biggl( \frac{
\sum_{k\in\{ 0,\ldots,\ell\} }
o(k)A_H(k)
M_H (k,h)
}{
\sum_{0\leq h\leq\ell}
{o}(h)
{A_H(h)}} \biggr)^{o'(h)}.
\end{eqnarray*}

\section{Monotonicity}\label{secmono}
A crucial property for comparing the Wright--Fisher model with other processes
is monotonicity. We will realize a coupling of
the lumped Wright--Fisher processes with different
initial conditions, and we will deduce the monotonicity from the coupling
construction.
All the processes will be built on a single large probability space.
We consider a probability space $(\Omega,{\mathcal F}, P)$
containing the following collection of independent random variables,
all of them following the uniform law on the interval
$[0,1]$:
\begin{eqnarray*}
&\displaystyle U_{n}^{i,j},\qquad n\geq1,\qquad1\leq i\leq m,\qquad1
\leq j\leq\ell\index{$U_{n,l}$},&
\\
&\displaystyle S_{n}^i,\qquad n\geq1 \index{$S_n$},\qquad1\leq i\leq
m.&
\end{eqnarray*}

\subsection{Coupling of the lumped processes}\label{coulum}
We build here a coupling of the lumped processes.
We set
\[
\forall n\geq1\qquad
R_n = \pmatrix{
S_n^1,  U_{n}^{1,1},\ldots,U_{n}^{1,\ell}
\cr
\vdots\hspace*{15pt}\vdots\hspace*{15pt}\cdots\hspace*{15pt}\vdots
\cr
S_n^m, U_{n}^{m,1},\ldots,U_{n}^{m,\ell}}.
\]
The matrix $R_n$ is the random input which is used to perform the
$n$th step of the Markov chains.
We denote by ${\mathcal R}$ the set of the matrices
of size $m\times(\ell+1)$ with coefficients in $[0,1]$.
The sequence
$(R_n)_{n\geq1}$ is a sequence of independent identically distributed
random matrices with values in
${{\mathcal R}}$.

\medskip\textit{Mutation.}
We define a map 
\[
{\mathcal M}_H\dvtx \{ 0,\ldots,\ell\} \times[0,1]^\ell\to\{
0,\ldots,\ell\} \index{${\mathcal M}_H$}
\]
in order to
couple the mutation mechanism starting with different
chromosomes.
Let $b\in\{ 0,\ldots,\ell\} $, and let $u_1,\ldots,u_\ell\in
[0,1]^\ell$.
The map ${\mathcal M}_H$ is defined by setting
\[
{\mathcal M}_H(b, u_1,\ldots,u_\ell)= b-\sum
_{k=1}^b1_{u_k<q/(\kappa-1)} +\sum
_{k=b+1}^\ell1_{u_k>1-q}.
\]
The map ${\mathcal M}_H$ is built in such a way that,
if $U_1,\ldots,U_\ell$ are
random variables
with uniform law on
the interval
$[0,1]$,
all being independent, then for any
$b\in\{ 0,\ldots,\ell\} $, the law of
${\mathcal M}_H(b,
U_1,\ldots,U_\ell)$
is given by the line of the mutation matrix $M_H$ associated to $b$,
that is,
\[
\forall c\in\{ 0,\ldots,\ell\} \qquad P \bigl( {\mathcal M}_H(b,
U_1,\ldots,U_\ell)=c \bigr) = M_H(b,c).
\]

\textit{Selection for the distance process.}
We realize the replication mechanism with the help of a selection
map
\[
{\mathcal S}_H\dvtx \{ 0,\ldots,\ell\}^m\times[0,1]\to\{ 1,
\ldots,m \}. \index{${\mathcal S}_H$} %
\]
Let $d\in\{ 0,\ldots,\ell\}^m$, and let $s\in[0,1[$. We define
${\mathcal S}_H(d,s)=i$ where $i$ is the unique index in $\{ 1,\ldots,m
\}$ satisfying
\[
\frac{A_H(d(1))
+\cdots+ A_H(d(i-1))}{
A_H(d(1))
+\cdots+ A_H(d(m))} \leq s < 
\frac{A_H(d(1))
+\cdots+ A_H(d(i))}{
A_H(d(1))
+\cdots+ A_H(d(m))}.
\]
The map ${\mathcal S}_H$ is built in such a way that,
if $S$ is
a random variable
with uniform law on
the interval
$[0,1]$,
then for any
$d\in\{ 0,\ldots,\ell\}^m$,
the law of
${\mathcal S}_H(d,S)$
is given by
\[
\forall i\in\{ 1,\ldots,m \}\qquad P \bigl( {\mathcal S}_H(d, S)=i
\bigr) = \frac{
{{A_H(d(i))} }
}{
A_H(d(1))
+\cdots+ A_H(d(m))}.
\]
%
%

\textit{Coupling for the distance process.}
We build a deterministic map
%
\[
\Psi_H\dvtx  \{ 0,\ldots,\ell\}^m\times{{\mathcal R}}\to\{
0,\ldots,\ell\}^m \index{$\Psi_H$} %
\]
in order to realize the coupling between distance processes
with various initial conditions.
The coupling map $\Psi_H$ is defined by
\begin{eqnarray*}
&& \forall r\in{\mathcal R},\ \forall d\in\{ 0,\ldots,\ell\}^m
\\
&&\qquad \Psi_H(d,r) = \pmatrix{ {\mathcal M}_H \bigl(d\bigl({
\mathcal S}_H\bigl(d,r(1,1)\bigr)\bigr),r(1,2),\ldots,r(1,\ell+1)
\bigr)
\cr
\qquad\vdots\qquad
\cr
{\mathcal M}_H \bigl(d\bigl({
\mathcal S}_H\bigl(d,r(m,1)\bigr)\bigr),r(m,2),\ldots,r(m,\ell+1)
\bigr)}.
\end{eqnarray*}
%
The coupling is then built in a standard way with the help of the
i.i.d. sequence
$(R_n)_{n\geq1}$
and the map
$\Psi_H$.
Let $d\in\{ 0,\ldots,\ell\}^m$ be the starting point of the process.
We build the distance process
$(D_n)_{n\geq0}
\index{$D_n$}$
by setting
$D_0=d$ and
\[
\forall n\geq1\qquad D_n = \Psi_H (D_{n-1},
R_n ). %
\]
A routine check shows that the process
$(D_n)_{n\geq0}$
is a Markov chain starting from $d$
with the adequate transition matrix. This way we have coupled
the distance processes with various initial conditions.

\medskip\textit{Selection for the occupancy process.}
We realize the replication mechanism with the help of a selection
map
\[
{\mathcal S}_O\dvtx {\mathcal P}^m_{\ell+1}
\times[0,1]\to\{ 0,\ldots,\ell\} \index{${\mathcal S}_O$}.
\]
Let $o\in{\mathcal P}^m_{\ell+1}$, and let $s\in[0,1[$. We define
${\mathcal S}_O(o,s)=l$ where $l$ is the unique index in $\{ 0,\ldots,\ell\} $ satisfying
\[
\frac{o(0) A_H(0)
+\cdots+ o(l-1) A_H(l-1)}{
o(0) A_H(0)
+\cdots+ o(\ell) A_H(\ell)} 
\leq s< 
\frac{o(0) A_H(0)
+\cdots+ o(l) A_H(l)}{
o(0) A_H(0)
+\cdots+ o(\ell) A_H(\ell)}.
\]
The map ${\mathcal S}_O$ is built in such a way that,
if $S$ is
a random variable
with uniform law on
the interval
$[0,1]$,
then for any
$o\in{\mathcal P}^m_{\ell+1}$,
the law of
${\mathcal S}_O(o,S)$
is given by
\[
\forall l\in\{ 0,\ldots,\ell\} \qquad P \bigl( {\mathcal S}_O(o,
S)=l \bigr) = \frac{
{{o(l) A_H(l)} }
}{
o(0) A_H(0)
+\cdots+ o(\ell) A_H(\ell)}.
\]

\textit{Coupling for the occupancy process.}
We build a deterministic map
%
\[
\Psi_O\dvtx  {\mathcal P}^m_{\ell+1}\times{{\mathcal
R}}\to{\mathcal P}^m_{\ell+1} \index{$\Psi_O$} %
\]
in order to realize the coupling between occupancy processes
with various initial conditions.
The coupling map $\Psi_O$ is defined by
\begin{eqnarray*}
&& \forall r\in{\mathcal R},\ \forall o\in{\mathcal P}^m_{\ell+1}
\\
&&\qquad \Psi_O(o,r) = {\mathcal O}_H \pmatrix{ {\mathcal
M}_H \bigl({\mathcal S}_O\bigl(o,r(1,1)\bigr),r(1,2),
\ldots,r(1,\ell+1) \bigr)
\cr
\qquad\vdots\qquad
\cr
{\mathcal M}_H
\bigl({\mathcal S}_O\bigl(o,r(m,1)\bigr),r(m,2),\ldots,r(m,\ell+1)
\bigr)}.
\end{eqnarray*}
Let $o\in{\mathcal P}^m_{\ell+1}$ be the starting point of the process.
We build the occupancy process
$(O_n)_{n\geq0}
\index{$\Omega_n$}$
by setting
$O_0=o$ and
\[
\forall n\geq1\qquad O_n = \Psi_O (O_{n-1},
R_n ). %
\]
A routine check shows that the process
$(O_n)_{n\geq0}$
is a Markov chain starting from~$o$
with the adequate transition matrix. This way we have coupled
the occupancy processes with various initial conditions.

\subsection{Monotonicity of the model}
We first recall some standard definitions concerning monotonicity
and coupling
for stochastic processes.
A classical reference
is Liggett's book
\cite{LIG},
especially for applications to particle systems.
In the next two definitions,
we consider
a discrete time Markov chain
$(X_n)_{n\geq0}$
with
values in a space ${\mathcal E}$.
We suppose that the state space ${\mathcal E}$ is finite and
that it is equipped with a partial order $\leq$.
A function $f\dvtx {\mathcal E}\to{\mathbb R}$ is nondecreasing if
\[
\forall x,y\in{\mathcal E}\qquad x\leq y\quad\Rightarrow\quad f(x)\leq f(y).
\]
%
%

\begin{definition}
The Markov chain
$(X_n)_{n\geq0}$ is said to be monotone if,
for any nondecreasing function $f$, the function
\[
x\in{\mathcal E}\mapsto E \bigl(f(X_n) | X_0=x \bigr)
\]
is nondecreasing.
\end{definition}

A natural way to prove monotonicity is to construct an adequate coupling.
%

\begin{definition}
A coupling
for the Markov chain
$(X_n)_{n\geq0}$
is a family of processes
$(X_n^x)_{n\geq0}$
indexed by
$x\in{\mathcal E}$, which are all defined on the same probability
space, and
such that,
for $x\in{\mathcal E}$, the process
$(X_n^x)_{n\geq0}$ is the Markov chain
$(X_n)_{n\geq0}$
starting from $X_0=x$.
The coupling is said to be monotone if
\[
\forall x,y\in{\mathcal E}\qquad x\leq y\quad\Rightarrow\quad\forall
n\geq1
\qquad X_n^x\leq X_n^y.
\]
\end{definition}

If there exists a monotone coupling,
then the
Markov chain
is monotone.

We try next to apply these definitions to our model.
The space $\{ 0,\ldots,\ell\}^m$ is naturally endowed with a
partial order
\[
d\leq e\quad\Longleftrightarrow\quad\forall i\in\{ 1,\ldots,m \}
\qquad d(i)
\leq e(i).
\]
The map ${\mathcal M}_H$ is nondecreasing with respect to the Hamming class,
that is,
\begin{eqnarray*}
&& \forall b,c \in\{ 0,\ldots,\ell\},\ \forall u_1,
\ldots,u_\ell\in[0,1]
\\
&& \qquad
b\leq c\quad\Rightarrow\quad{\mathcal
M}_H(b,u_1,\ldots,u_\ell) \leq{\mathcal
M}_H(c,u_1,\ldots,u_\ell);
\end{eqnarray*}
see \cite{CE} for a detailed proof.
In the neutral case $\sigma=1$,
the map ${\mathcal S}_H$ does not depend on the population, in fact,
\[
\forall d \in\{ 0,\ldots,\ell\}^m,\ \forall s\in[0,1] \qquad{
\mathcal S}_H(d,s) = \lfloor ms\rfloor.
\]
As a consequence, we have
\begin{eqnarray*}
&&\forall d,e \in\{ 0,\ldots,\ell\}^m,\ \forall s\in[0,1]
\qquad
d\leq
e\quad\Rightarrow\quad d \bigl({\mathcal S}_H(d,s) \bigr) \leq e
\bigl({\mathcal S}_H(e,s) \bigr).
\end{eqnarray*}

%
\begin{lemma}\label{monphimor}
In the neutral case $\sigma=1$,
the map $\Psi_H$ is nondecreasing with respect to the distances,
that is,
\[
\forall d,e \in\{ 0,\ldots,\ell\}^m,\ \forall r\in{\mathcal R}
\qquad d\leq e \quad\Rightarrow\quad\Psi_H(d,r) \leq
\Psi_H(e,r).
\]
\end{lemma}

\begin{pf}
Let $r\in{\mathcal R}$, and let $d,e\in\{ 0,\ldots,\ell\}^m$,
$d\leq e$.
Let $i\in\{ 1,\ldots,m \}$.
Since
\[
{\mathcal S}_H\bigl(d,r(i,1)\bigr) = {\mathcal S}_H
\bigl(e,r(i,1)\bigr) = \bigl\lfloor m r(i,1)\bigr\rfloor,
\]
we have
%
\[
d\bigl({\mathcal S}_H\bigl(d,r(i,1)\bigr)\bigr) \leq e
\bigl({\mathcal S}_H\bigl(e,r(i,1)\bigr)\bigr).
\]
This inequality and
the monotonicity of the map ${\mathcal M}_H$
imply that
\begin{eqnarray*}
&& {\mathcal M}_H \bigl( d\bigl({\mathcal S}_H
\bigl(d,r(i,1)\bigr)\bigr),r(i,2),\ldots,r(i,\ell+1) \bigr)
\\
&&\qquad \leq{\mathcal M}_H \bigl( e\bigl({\mathcal S}_H
\bigl(e,r(i,1)\bigr)\bigr),r(i,2),\ldots,r(i,\ell+1) \bigr).
\end{eqnarray*}
Therefore $\Psi_H(d,r) \leq
\Psi_H(e,r)$ as requested.
\end{pf}

%
\begin{corollary}\label{corneu}
In the neutral case $\sigma=1$,
the distance process
$(D_n)_{n\geq0}$ is monotone.
\end{corollary}

Unfortunately, the map
$\Psi_H$ is not monotone for $\sigma>1$. Indeed, suppose that
%
\begin{eqnarray*}
&\displaystyle \kappa=3,\qquad\sigma=2,\qquad m=3,\qquad\ell\geq2,&
\\
&\displaystyle \tfrac{2}{3}<s_1< \tfrac{3}{4},\qquad
\tfrac{3}{4}<s_2<1,\qquad\tfrac{3}{4}<s_3<1,&
\\
&\displaystyle \forall i\in\{ 1,2,3 \},\ \forall j\in\{ 1,\ldots,\ell\} \qquad
u_{i,j}\in\biggl[ \frac{q}{3},1-q \biggr].&
\end{eqnarray*}
Recall that
\[
r = \pmatrix{ s_1, u_{1,1},\ldots,u_{1,\ell}
\cr
s_2, u_{2,1},\ldots,u_{2,\ell}
\cr
s_3,
u_{3,1},\ldots,u_{3,\ell}}. %
\]
We have then
%
\[
\Psi_H \pmatrix{ 0
\cr
2
\cr
1} = \pmatrix{ 2
\cr
1
\cr
1},\qquad
\Psi_H \pmatrix{ 1
\cr
2
\cr
1}= \pmatrix{ 1
\cr
1
\cr
1}. %
\]
This creates a serious
complication.
To get around this problem, we
lump further the distance process in order to build the occupancy
process. It turns out that the occupancy process
is monotone even in the nonneutral case.
%
We define an order $\preceq\index{$\preceq$}$ on
${\mathcal P}^m_{\ell+1}$
as follows.
Let
$o=(o(0),\ldots,o(\ell))$
and
$o'=(o'(0),\ldots,o'(\ell))$
belong to~${\mathcal P}^m_{\ell+1}$.
We say that $o$ is smaller than or equal to $o'$,
which we denote by $o\preceq o'$, if
%
\[
\forall l\leq\ell\qquad o(0)+\cdots+o(l) \leq o'(0)+
\cdots+o'(l).
\]
As shown in \cite{CE},
the map ${\mathcal S}_O$ is nonincreasing with respect to the occupancy
distribution, that is,
\[
\forall o,o' \in{\mathcal P}^m_{\ell+1},\
\forall s\in[0,1]\qquad o\preceq o'\quad\Rightarrow\quad{\mathcal
S}_O(o,s) \geq{\mathcal S}_O\bigl(o',s
\bigr). %
\]

%
\begin{lemma}\label{monophio}
The map $\Psi_O$ is nondecreasing with respect to
the occupancy distribution, that is,
\[
\forall o,o' \in{\mathcal P}^m_{\ell+1},\
\forall r\in{\mathcal R} \qquad
o\preceq o' \quad
\Rightarrow\quad\Psi_O(o,r) \preceq\Psi_O
\bigl(o',r\bigr). %
\]
\end{lemma}

\begin{pf}
Let
$r\in{\mathcal R}$, and let
$o,o' \in{\mathcal P}^m_{\ell+1}$ be such that
$o\preceq o'$.
Using the monotonicity of the map ${\mathcal S}_O$,
we have
\[
\forall i\in\{ 1,\ldots,m \}\qquad{\mathcal S}_O\bigl(o,r(i,1)
\bigr) \geq{\mathcal S}_O\bigl(o',r(i,1)\bigr).
\]
This inequality and
the monotonicity of the map ${\mathcal M}_H$
imply that
\begin{eqnarray*}
\forall i\in\{ 1,\ldots,m \}\qquad &&{\mathcal M}_H \bigl( {\mathcal
S}_O\bigl(o,r(i,1)\bigr),r(i,2),\ldots,r(i,\ell+1) \bigr)
\\
&&\qquad
\geq{\mathcal M}_H \bigl( {\mathcal S}_O
\bigl(o',r(i,1)\bigr),r(i,2),\ldots,r(i,\ell+1) \bigr).
\end{eqnarray*}
Therefore $\Psi_O(o,r) \leq
\Psi_O(o',r)$ as requested.
\end{pf}

%
\begin{corollary}\label{corocc}
The occupancy process
$(O_n)_{n\geq0}$ is monotone.
\end{corollary}

\subsection{The FKG inequality}
\label{poco}
We consider the product space
$\{ 0,\ldots,\ell\} ^m$ equipped with the natural product order
\[
d\leq e\quad\Longleftrightarrow\quad\forall i\in\{ 1,\ldots,m \}
\qquad d(i)
\leq e(i).
\]
%

\begin{definition}
A probability measure $\mu$ on $\{ 0,\ldots,\ell\} ^m$
is said to have positive correlations if
for any
functions $f,g\dvtx \{ 0,\ldots,\ell\} ^m\to{\mathbb R}$ which are
nondecreasing, we have
\[
\sum_{d\in\{ 0,\ldots,\ell\} ^m} f(d)g(d) \mu(d) \geq\biggl(\sum
_{d\in\{ 0,\ldots,\ell\} ^m} f(d) \mu(d) \biggr) \biggl(\sum
_{d\in\{ 0,\ldots,\ell\} ^m} g(d) \mu(d) \biggr).
\]
\end{definition}

The Harris inequality, or
the FKG inequality in this context, says that any
product probability measure on $\{ 0,\ldots,\ell\} ^m$
has positive correlations.
The FKG inequality is in fact true for any product
probability measure on a product of the interval
$[0,1]$; see Section~2.2 of
Grimmett's book~\cite{GRI}.
As far as correlations are concerned,
there is not much to do with the original Wright--Fisher model because its
state space is not partially ordered.
So we examine the distance process.
%

\begin{proposition}\label{lawD}
Suppose that we are in the neutral case $\sigma=1$.
If the law of $D_0$ has
positive correlations, then
for any $n\geq0$, the law of $D_n$
has positive correlations.
\end{proposition}


\begin{pf}
The Wright--Fisher model
$(X_n)_{n\geq0}$
can be seen as a probabilistic cellular
automaton. Indeed, given the population $X_n=x$ at time $n$,
the individuals $(X_{n+1}(i), 1\leq i\leq m)$ of the population
at time $n+1$ are independent. This still holds for the
distance process.
By Corollary~\ref{corneu}, the neutral
distance process
$(D_n)_{n\geq0}$ is monotone.
Monotone probabilistic cellular automata preserve the FKG inequality.
This is explained in detail by Mezi\'c \cite{MI}, and it was first
observed by Harris \cite{Har} at the very end of his article on
continuous time processes. Because the argument is very short, we
reproduce it here.
Suppose that the initial law $\mu$ of $D_0$ has
positive correlations.
Let
$f,g\dvtx \{ 0,\ldots,\ell\} ^m\to{\mathbb R}$ be two
nondecreasing
functions.
For any $d\in\{ 0,\ldots,\ell\}^m$, the conditional law of $D_1$
knowing that
$D_0=d$ is a product measure on $\{ 0,\ldots,\ell\}^m$, thus it
satisfies the
FKG inequality, whence
\begin{eqnarray*}
&& \forall d\in\{ 0,\ldots,\ell\}^m
\\
&&\qquad E \bigl(f(D_1)g(D_1) | D_0=d \bigr) \geq E
\bigl(f(D_1) | D_0=d \bigr) E \bigl(g(D_1) |
D_0=d \bigr).
\end{eqnarray*}
We integrate the inequality with respect to the initial law $\mu$:
\begin{eqnarray*}
&& \sum_{d\in\{ 0,\ldots,\ell\} ^m} E \bigl(f(D_1)g(D_1)
| D_0=d \bigr) \mu(d)
\\
&&\qquad \geq
\sum_{d\in\{ 0,\ldots,\ell\} ^m}
E \bigl(f(D_1) | D_0=d \bigr) E \bigl(g(D_1)
| D_0=d \bigr) \mu(d).
\end{eqnarray*}
Since $(D_n)_{n\geq0}$ is monotone, the maps
\begin{eqnarray*}
d\in\{ 0,\ldots,\ell\}^m&\mapsto& E \bigl(f(D_1) |
D_0=d \bigr),
\\
d\in\{ 0,\ldots,\ell\}^m&\mapsto& E
\bigl(g(D_1) | D_0=d \bigr),
\end{eqnarray*}
are nondecreasing.
By hypothesis,
the initial law $\mu$ has
positive correlations, therefore
%
\begin{eqnarray*}
\hspace*{-2pt}&& \sum_{d\in\{ 0,\ldots,\ell\} ^m} E \bigl(f(D_1) |
D_0=d \bigr) E \bigl(g(D_1) | D_0=d \bigr)
\mu(d)
\\
\hspace*{-2pt}&&\qquad  \geq
\biggl(\sum_{d\in\{ 0,\ldots,\ell\} ^m} E \bigl(f(D_1) |
D_0=d \bigr) \mu(d) \biggr) \biggl(\sum
_{d\in\{ 0,\ldots,\ell\} ^m} E \bigl(g(D_1) | D_0=d \bigr)
\mu(d) \biggr).
\end{eqnarray*}
The two above inequalities
imply
that the law of $D_1$ has positive correlations.
We conclude
by iterating the argument.
\end{pf}

\section{Stochastic bounds}\label{secbounds}
In this section, we take advantage of the monotonicity of the map $\Psi_O$
to compare the process
$(O_n)_{n\geq0}$ with simpler processes.
\subsection{Lower and upper processes}
We shall build
a lower
process
$(O^\ell_n)_{n\geq0} \index{$O^\ell_n$}$
and an upper
process
$(O^1_n)_{n\geq0} \index{$O^1_n$}$
satisfying
\[
\forall n\geq0\qquad O^\ell_n \preceq O_n
\preceq O^1_n.
\]
Loosely speaking, the upper process
evolves as follows. As long as there is no master sequence present
in the population, the process
$(O^1_n)_{n\geq0}$ evolves exactly as the initial process
$(O_n)_{n\geq0}$.
When the first master sequence appears, all the other chromosomes
are set in the Hamming class~$1$; that is, the process jumps to the
state $(1,m-1,0,\ldots,0)$. As long as the master sequence is present,
the mutations on nonmaster sequences leading to nonmaster sequences are
suppressed, and any mutation of a master sequence leads to a chromosome
in the first Hamming class.
The dynamics of the lower process is similar, except that the chromosomes
distinct from the master sequence are sent to the last Hamming
class~$\ell$
instead of the first one.
We shall next construct precisely these dynamics.
We define two maps
$\pi_\ell, \pi_1\dvtx {\mathcal P}^m_{\ell+1}\to{\mathcal P}^m_{\ell+1}
\index{$\pi_\ell,\pi_1$}$
by setting
%
\begin{eqnarray*}
\forall o\in{\mathcal P}^m_{\ell+1}\qquad
\pi_\ell(o) &=& \bigl( o(0),0,\ldots,0,m-o(0) \bigr),
\\
\pi_1(o) &=& \bigl(o(0),m-o(0),0,\ldots,0 \bigr).
\end{eqnarray*}
Obviously,
\[
\forall o\in{\mathcal P}^m_{\ell+1}\qquad
\pi_\ell(o) \preceq o \preceq\pi_1(o).
\]
We denote by ${\mathcal W}^*$ the set of the occupancy distributions
containing the master sequence,
that is,
\[
{\mathcal W}^*= \bigl\{ o\in{\mathcal P}^m_{\ell+1}\dvtx o(0)\geq1
\bigr\} \index{${\mathcal W}^*$} %
\]
and
by ${\mathcal N}$ the set of the occupancy distributions
which do not contain the master sequence,
that is,
\[
{\mathcal N} = \bigl\{ o\in{\mathcal P}^m_{\ell+1}\dvtx o(0)=0
\bigr\}. \index{${\mathcal N}$} %
\]
Let $\Psi_O$ be the coupling map defined in
Section~\ref{coulum}.
We define
a lower map $\Psi_O^{\ell}$ by setting,
for $o\in{\mathcal P}^m_{\ell+1}$ and $r\in{\mathcal R}$,
\[
\index{$\Psi_O^{\ell}$}
\Psi_O^{\ell}(o,r) = \cases{ \Psi_{O}(o,r), &
\quad if $o\in{\mathcal N}$ and $\Psi_O(o,r)\notin{\mathcal W}^*$,
\vspace*{3pt}\cr
\pi_\ell\bigl(\Psi_{O}(o,r) \bigr), &\quad if $o\in{
\mathcal N}$ and $\Psi_O(o,r)\in{\mathcal W}^*$,
\vspace*{3pt}\cr
\pi_\ell\bigl(\Psi_{O}\bigl(\pi_\ell(o),r\bigr)
\bigr), &\quad if $o\in{\mathcal W}^*$.}
\]
Similarly, we define
an upper map $\Psi_O^{1}$ by setting,
for $o\in{\mathcal P}^m_{\ell+1}$ and $r\in{\mathcal R}$,
\[
\index{$\Psi_O^{1}$}
\Psi_O^{1}(o,r) = \cases{ \Psi_{O}(o,r), &
\quad if $o\in{\mathcal N}$ and $\Psi_O(o,r)\notin{\mathcal W}^*$,
\vspace*{3pt}\cr
\pi_1 \bigl(\Psi_{O}(o,r) \bigr), &\quad if $o\in{
\mathcal N}$ and $\Psi_O(o,r)\in{\mathcal W}^*$,
\vspace*{3pt}\cr
\pi_1 \bigl(\Psi_{O}\bigl(\pi_1(o),r\bigr)
\bigr), &\quad if $o\in{\mathcal W}^*$.}
\]
A direct application of
Lemma~\ref{monophio} yields that the map
$\Psi_O^{\ell}$ is below the map
$\Psi_{O}$ and the map
$\Psi_O^{1}$ is above the map
$\Psi_{O}$
in the following sense:
%
\[
\forall r\in{\mathcal R},\ \forall o\in{\mathcal P}^m_{\ell+1}
\qquad\Psi_O^{\ell}(o,r) \preceq\Psi_{O}(o,r)
\preceq\Psi_O^{1}(o,r). %
\]
We\vspace*{1pt} define a lower process
$(O^\ell_n)_{n\geq0}$ and
an upper process
$(O^1_n)_{n\geq0}$
with the help
of the i.i.d. sequence
$(R_n)_{n\geq1}$ and
the maps $\Psi_O^{\ell}$,
$\Psi_O^{1}$ as follows.
Let $o\in{\mathcal P}^m_{\ell+1}$ be the starting point of the process.
We set
$O^\ell_0=O^1_0=o$ and
\[
\forall n\geq1\qquad O^\ell_n = \Psi_O^{\ell}
\bigl(O^\ell_{n-1}, R_n \bigr), \qquad
O^1_n = \Psi_O^{1}
\bigl(O^1_{n-1}, R_n \bigr). %
\]
%

\begin{proposition}\label{domiji}
Suppose that the three processes
$(O^\ell_n)_{n\geq0}$,
$(O_n)_{n\geq0}$,
$(O^1_n)_{n\geq0}$,
start from the same occupancy distribution $o$.
We have
\[
\forall n\geq0\qquad O^\ell_n \preceq O_n
\preceq O^1_n.
\]
\end{proposition}

The proof is similar to the proof of Proposition~8.1 in \cite{CE}.
\subsection{Dynamics of the bounding processes}\label{dynabound}
We study next the dynamics of the processes
$(O^\ell_n)_{n\geq0}$
and
$(O^1_n)_{n\geq0}$
in ${\mathcal W}^*$. The computations
are the same for both processes. Throughout the section,
we fix $\theta$ to be either $1$ or~$\ell$, and we denote by
$(O^\theta_n)_{n\geq0}$ the corresponding process.
For the process
$(O^\theta_n)_{n\geq0}$, the states
\[
{\mathcal T}^\theta= \bigl\{ o\in{\mathcal P}^m_{\ell+1}\dvtx
o(0)\geq1\mbox{ and } o(0)+o(\theta)<m \bigr\} \index
{${\mathcal T}^\theta$}
\]
are transient,
while the populations in
$ {{\mathcal N}\cup({\mathcal W}^*\setminus{\mathcal T}^\theta
)}$
form a recurrent class.
Let us look at the transition mechanism of the process restricted to
$ {{\mathcal W}^*\setminus{\mathcal T}^\theta}$.
Since
\[
{{\mathcal W}^*\setminus{\mathcal T}^\theta} = \bigl\{ o\in{\mathcal
P}^m_{\ell+1}\dvtx  o(0)\geq1\mbox{ and } o(0)+o(\theta)=m
\bigr\},
\]
we see that a state of
$ {{\mathcal W}^*\setminus{\mathcal T}^\theta}$
is completely determined by the first occupancy number,
which is equal to
the number of copies of the master sequence present in the
population.
From the previous observations,
we conclude that,
whenever
$(O^\theta_n)_{n\geq0}$ starts in
$ {{\mathcal W}^*\setminus{\mathcal T}^\theta}$,
the dynamics\vspace*{1pt} of the number of master sequences~$(O^\theta_n(0))_{n\geq0}$
is Markovian
until the time of exit from
${\mathcal W}^*\setminus{\mathcal T}^\theta$.
We denote
by $(Z^\theta_n)_{n\geq0}\index{$Z^\theta_n$}$
a Markov chain
on $\{ 0,\ldots,m \}$
with the following
transition probabilities:
for
$h\in\{ 1,\ldots,m \}$ and $k\in\{ 0,\ldots,m \}$,
%
\[
\forall n\geq0\qquad P \bigl(Z^\theta_{n+1}=k |
Z^\theta_n=h \bigr) = 
P \bigl(O^\theta_{n+1}(0)=
k | O^\theta_n(0)= h \bigr), 
\]
and for $h=0$ and
$k\in\{ 0,\ldots,m \}$,
%
\[
\forall n\geq0\qquad P \bigl(Z^\theta_{n+1}=k |
Z^\theta_n=0 \bigr) = 
\pmatrix{ {m}
\cr
{k}}
M_H(\theta,0)^k \bigl(1-M_H(\theta,0)
\bigr)^{m-k}. %
\]
Let us denote by
$p^\theta(h,k)$ the above transition probability, and let us compute
its value. We use the definition of the transition mechanism of
$(O^\theta_n)_{n\geq0}$ to get
\[
p^\theta(h,k) = \sum_{i\in\{ 0,\ldots,m \}} \sum
_{
j=0
}^i p^\theta(h,i,j,k)
\]
%
where $p^\theta(h,i,j,k)$ is given by
\begin{eqnarray*}
p^\theta(h,i,j,k) &=& P\lleft(\matrix{ \mbox{$i$ master sequences are
selected,}
\cr
\mbox{$j$ master sequences do not mutate,}
\cr
\mbox{$k-j$ nonmaster sequences}
\cr
\mbox{mutate into a master sequence}} \left|\rule{0pt}{25pt}
Z^\theta_n=h \right.\rright)
\\
&=&
\pmatrix{ {m}
\cr
{i}} \frac{
(\sigma h)^i
(m-h)^{m-i}}{ ((\sigma-1)h+m )^m} \pmatrix{ {i}
\cr
{j}}
M_H(0,0)^j \bigl(1- M_H(0,0)
\bigr)^{i-j}
\\
&&{}\times\pmatrix{ {m-i}
\cr
{k-j}} M_H(
\theta,0)^{k-j} \bigl(1- M_H(\theta,0)
\bigr)^{m-i-k+j}.
\end{eqnarray*}
The
Markov chain
$(Z^\theta_n)_{n\geq0}$ corresponds to the evolution of the
number of master sequences in a Wright--Fisher model with two types,
the master type having fitness~$\sigma$
and the other type having fitness~$1$, and with the following
mutation matrix between the two types:
\begin{eqnarray*}
 P(\mbox{the master type mutates into the nonmaster type}) &=&
1-M_H(0,0),
\\
 P(\mbox{the nonmaster type mutates into
the master type}) &=& M_H(\theta,0).
\end{eqnarray*}
We can also realize
the Markov chain
$(Z^\theta_n)_{n\geq0}$ on our common probability space.
We define two maps
$\Xi^\ell, \Xi^1\dvtx \{ 0,\ldots,m \}\to{\mathcal P}^m_{\ell+1}$ by
setting
\begin{eqnarray*}
\forall i\in\{ 0,\ldots,m \}\qquad \Xi^\ell(i) &=& (i,0,\ldots,0,m-i),
\\
\Xi^1(i) &=& (i,m-i,0,\ldots,0).
\end{eqnarray*}
%
Let $i\in\{ 0,\ldots,m \}$ be the starting point of the process.
We set
$Z_0^\theta=i$ and
\[
\forall n\geq1\qquad Z^\theta_n = \Psi_O^{\theta}
\bigl(\Xi^\theta\bigl(Z^\theta_{n-1}\bigr),
R_n \bigr) (0). %
\]
This construction yields
a
Markov chain
$(Z^\theta_n)_{n\geq0}$ starting from $i$
with the adequate transition matrix.
Moreover the maps
$\Xi^\ell$, $\Xi^1$ are nondecreasing.
By Lemma~\ref{monophio}, the map
$\Psi_O$ is also nondecreasing with respect to the occupancy distribution.
We conclude that
the above coupling is monotone, and the
Markov chain
$(Z^\theta_n)_{n\geq0}$ is monotone.

\subsection{Invariant probability measures}\label{bounds}
Our goal is to estimate the law $\nu$ of the
fraction of the master sequence in the population
at equilibrium.
The probability measure $\nu$ is the probability measure on the
interval $[0,1]$
satisfying the following identities.
For any function $f\dvtx [0,1]\to{\mathbb R}$,
\[
\int_{[0,1]}f \,d\nu= \lim_{n\to\infty}
E \biggl( f \biggl( \frac{1}{m} N(X_n) \biggr)
\biggr) = 
\sum_{x\in({\mathcal A}^\ell)^m} 
f \biggl(
\frac{1}{m} N(x) \biggr) \mu(x),
\]
where $\mu$ is the invariant probability measure of the Markov chain
$(X_n)_{n\geq0}$.
In fact, the probability measure $\nu$
is the image of $\mu$
through the map
\[
x\in\bigl({\mathcal A}^\ell\bigr)^m\mapsto
\frac{1}{m} N(x)\in[0,1].
\]
We denote by
$\mu_O^\ell\index{$\mu^\ell_O$}$,
$\mu_O \index{$\mu_O$}$,
$\mu^1_O \index{$\mu^1_O$}$
the invariant probability measures of the Markov chains
$(O^\ell_n)_{n\geq0}$,
$(O_n)_{n\geq0}$,
$(O^1_n)_{n\geq0}$.
The probability $\nu$
is also the image of $\mu_O$
through the map
\[
o\in{\mathcal P}^m_{\ell+1}\mapsto\frac{1}{m} o(0)
\in[0,1].
\]
Thus,
for any function $f\dvtx [0,1]\to{\mathbb R}$,
\[
\int_{[0,1]}f \,d\nu= \sum_{o\in{\mathcal P}^m_{\ell+1}} f
\biggl(\frac{
o(0)
}{m} \biggr) \mu_O(o) = \lim
_{n\to\infty} E \biggl( f \biggl( \frac{1}{m}
O_n(0) \biggr) \biggr). %
\]
%
We fix now a nondecreasing function
$f\dvtx [0,1]\to{\mathbb R}$ such that $f(0)=0$.
Proposition~\ref{domiji} yields the inequalities
\[
\forall n\geq0\qquad f \biggl( \frac{1}{m} O^\ell_n(0)
\biggr) \leq f \biggl( \frac{1}{m} O_n(0) \biggr) \leq f
\biggl( \frac{1}{m} O^1_n(0) \biggr). %
\]
Taking the expectation and sending $n$ to $\infty$, we get
\[
\sum_{o\in{\mathcal P}^m_{\ell+1}} f \biggl( \frac{
o(0)
}{m}
\biggr) \mu_O^\ell(o) 
\leq\int
_{[0,1]}f \,d\nu
\leq
\sum_{o\in{\mathcal P}^m_{\ell+1}} f \biggl( \frac{ o(0) }{m} \biggr)
\mu_O^1(o). %
\]
%
We seek next estimates on the above sums.
The strategy is the same for the lower and the upper sum.
Thus we fix $\theta$ to be either $1$ or~$\ell$, and
we study
the invariant probability measure $\mu^\theta_O$.
For the
Markov chain
$(O^\theta_n)_{n\geq0}$,
the states of ${\mathcal T}^\theta$ are
transient,
while the populations in
$ {{\mathcal N}\cup({\mathcal W}^*\setminus{\mathcal T}^\theta
)}$
form a recurrent class.
Let $o^\theta_{\mathrm{exit}}$
be the occupancy distribution having
$m$ chromosomes in the Hamming class $\theta$,
\[
\forall l\in\{ 0,\ldots,\ell\} \qquad o^\theta_{\mathrm{exit}}(l)
= \cases{ m, &\quad if $l=\theta$,
\cr
0, &\quad otherwise.}\index{$o^\theta_{\mathrm{exit}}$}
\]
%
The
process
$(O^\theta_n)_{n\geq0}$
always exits
${\mathcal W}^*\setminus{\mathcal T}^\theta$ at
$o^\theta_{\mathrm{exit}}$.
This allows us to
estimate the invariant measure with the help of
the following renewal result.
%

\begin{proposition}\label{renewal}
Let
$(X_n)_{n\geq0}$ be a discrete time Markov chain with
values in a finite state space ${\mathcal E}$ which is irreducible and
aperiodic.
Let $\mu$ be the invariant probability
measure of the Markov chain
$(X_n)_{n\geq0}$.
Let ${\mathcal W}^*$ be a subset of ${\mathcal E}$, and let $e$ be a
point of
${\mathcal E}
\setminus{\mathcal W}^*$.
Let $f$ be a map from ${\mathcal E}$ to ${\mathbb R}$ which vanishes on
${\mathcal E}\setminus
{\mathcal W}^*$.
Let
\[
\tau^* = \inf\bigl\{ n\geq0\dvtx  X_n\in{\mathcal W}^* 
\bigr\},\qquad\tau= \inf\bigl\{ n\geq\tau^*\dvtx  X_n=e 
\bigr\}. %
\]
We have
\[
\sum_{x\in{\mathcal E}} f(x) \mu(x) = \frac{1}{E(\tau| X_0=e)} E
\Biggl(\sum_{n=\tau^*}^{\tau} f(X_n) \bigg|
X_0=e \Biggr).
\]
\end{proposition}

This\vspace*{1pt} result is proved in detail in \cite{CE}.
We apply the renewal result of Proposition~\ref{renewal}
to the\vspace*{1pt} process
$(O^\theta_n)_{n\geq0}$ restricted to
$ {{\mathcal N}\cup({\mathcal W}^*\setminus{\mathcal T}^\theta
)}$,
the set ${\mathcal W}^*\setminus{\mathcal T}^\theta$,
the occupancy distribution $o^\theta_{\mathrm{exit}}$
and
the function $o\mapsto f(o(0)/m)$.
Setting
\[
\index{$\tau^*$} \tau^* = \inf\bigl\{ n\geq0\dvtx  O^\theta_n
\in{\mathcal W}^* 
\bigr\},\qquad\tau= \inf\bigl\{ n\geq\tau^*\dvtx
O^\theta_n=o^\theta_{\mathrm{exit}} \bigr
\}, %
\]
we have
\[
\sum_{o\in{\mathcal P}^m_{\ell+1}} f \biggl(
\frac{ o(0) }{m} \biggr) \mu_O^\theta(o) =
\frac{E (\sum_{n=\tau^*}^{\tau}
f ( (O^\theta_n(0))/{m} )
|
O^\theta_0=o^\theta_{\mathrm{exit}}
)
}{
E (\tau|
O^\theta_0=o^\theta_{\mathrm{exit}}
)}. %
\]
Yet, whenever the process
$(O^\theta_n)_{n\geq0}$ is in ${\mathcal W}^*\setminus{\mathcal
T}^\theta$,
the dynamics of the number of master sequences
$(O^\theta_n(0))_{n\geq0}$
is the same as the dynamics of the Markov chain
$(Z^\theta_n)_{n\geq0}$ defined in Section~\ref{dynabound}.
Let $\tau_0$ be the hitting time of $0$, defined by
\[
\tau_0 = \inf\bigl\{ n\geq0\dvtx  Z^\theta_n=0
\bigr\}. \index{$\tau_0$}
\]
%
The process
$(O^\theta_n)_{n\geq0}$
always exits
${\mathcal W}^*\setminus{\mathcal T}^\theta$
at $o^\theta_{\mathrm{exit}}$. Therefore $\tau$ coincides with the
exit time of
${\mathcal W}^*\setminus{\mathcal T}^\theta$ after $\tau^*$.
Let $i\in\{ 1,\ldots,m \}$.
From the previous elements, we see that,
conditionally on the event
$ { \{
O^\theta_{\tau^*}(0)=i
\}}$,
the trajectory
$ (O^\theta_n(0), {\tau^*}\leq n \leq{\tau} )$
has the same
law as
the trajectory
$ (Z^\theta_n, 0\leq n\leq\tau_0 )$
starting from
$Z^\theta_0= i$,
whence
\begin{eqnarray*}
E \bigl(\tau-{\tau^*} | O^\theta_{\tau^*}(0)=i \bigr) & =& E
\bigl({\tau_0} | Z^\theta_0= i \bigr),
\\
E
\Biggl(\sum_{n=\tau^*}^{\tau} f \biggl(
\frac{O^\theta_n(0)}{m} \biggr) \Big| O^\theta_{\tau^*}(0)=i \Biggr) & =& E
\Biggl(\sum_{n=0}^{\tau_0} f \biggl(
\frac{Z^\theta_n}{m} \biggr) \Big| Z^\theta_0= i \Biggr).
\end{eqnarray*}
Conditioning with respect to
$O^\theta_{\tau^*}(0)$
and reporting in the formula for the invariant probability measure
$\mu^\theta_O$, we get
\begin{eqnarray*}
&& \sum_{o\in{\mathcal P}^m_{\ell+1}} f \biggl( \frac{ o(0) }{m} \biggr)
\mu_O^\theta(o)
\\
&&\qquad =
\frac{
\sum_{i=1}^m
E (\sum_{n=0}^{\tau_0}
f ({Z^\theta_n}/{m} ) |
Z^\theta_0= i )
P (
O^\theta_{\tau^*}(0)=i
|
O^\theta_0=o^\theta_{\mathrm{exit}}
)
}{
E (\tau^* |
O^\theta_0=o^\theta_{\mathrm{exit}}
)+
\sum_{i=1}^m
E ({\tau_0}
|
Z^\theta_0= i )
P (
O^\theta_{\tau^*}(0)=i
|
O^\theta_0=o^\theta_{\mathrm{exit}}
)
}.
\end{eqnarray*}
We\vspace*{1pt} must next estimate these expectations.
In Section~\ref{bide},
we deal with the terms involving the
Markov chain
$(Z^\theta_n)_{n\geq0}$.
In Section~\ref{disc},
we deal with the discovery time $\tau^*$.
\section{Approximating processes}\label{bide}
This section is devoted to the study of the dynamics of the
Markov chains
$(Z^\ell_n)_{n\geq0}$ and
$(Z^1_n)_{n\geq0}$.
The estimates
are carried out exactly in the same way for both Markov chains.
As we said before,
the
Markov chain
$(Z^\theta_n)_{n\geq0}$ corresponds to the evolution of the
number of master sequences in a Wright--Fisher model with two types.
Throughout the section, we fix $\theta=1$ or $\theta=\ell$, and
we remove $\theta$ from the notation in most places, writing simply
$p,Z_n$ instead of
$p^\theta,
Z^\theta_n$.

\medskip\textit{Asymptotic regime.}
We shall derive estimates in the regime where
\[
\ell\to+\infty,\qquad m\to+\infty,\qquad q\to0, \qquad{\ell q} \to
a\in\,]0,+
\infty[.
\]
Several inequalities will be valid only when the parameters
are sufficiently close to their limits.
We will say that a property holds
asymptotically to express that it holds for
$\ell,m$ large enough, $q$ small enough and $\ell q$ close
enough to~$a$.

\subsection{Large deviations for the transition matrix}
For $p\in[0,1]$ and $t\geq0$, we define
\[
I(p,t) = \cases{ \displaystyle t\ln\frac{t}{p}+(1-t)\ln
\frac{1-t}{1-p}, &\quad$0<p<1$, $0\leq t\leq1$,
\vspace*{3pt}\cr
0, &\quad$t=p=0$ or
$t=p=1$,
\vspace*{3pt}\cr
+\infty, &\quad$\bigl(p\in\{0,1\}, t\neq p\bigr)$ or $t>1$.}
\]
The function $I(p,\cdot)$ is the rate function governing the large deviations
of the binomial distribution ${\mathcal B}(n,p)$ with parameters $n$
and $p$. We recall a basic estimate for the binomial coefficients.

%
\begin{lemma}
\label{cnk}
For any $n\geq1$,
any $k\in\{ 0,\ldots, n \}$, we have
\[
\biggl| 
\ln\frac{n!}{ k!(n-k)!} 
+ 
k\ln
\frac{k}{ n} 
+(n-k)\ln\frac{n-k}{ n} \biggr| 
\leq2\ln
n+3.
\]
\end{lemma}

\begin{pf}
The proof is standard; see, for instance, \cite{EL}.
Setting, for $n\in\mathbb N$, $\phi(n)=\ln n!-n\ln n+n$,
we have
%
\begin{eqnarray*}
\ln\frac{n!}{ k!(n-k)!} &=& \ln n!-\ln k! -\ln(n-k)!
\\
&=& n\ln n-n+\phi(n) -
\bigl(k\ln k-k+\phi(k) \bigr)
\\
&&{} - \bigl((n-k)\ln(n-k)-(n-k)+\phi(n-k)
\bigr)
\\
&=& -{k}\ln\frac{k}{n} -{(n-k)}\ln\frac{n-k}{n} +\phi(n)-\phi(k)-
\phi(n-k).
\end{eqnarray*}
Comparing the discrete sum
$\ln n!=\sum_{1\leq k\leq n}\ln k$
to the integral
$\int_1^n\ln x \,dx$,
we see that
$1\leq\phi(n)\leq\ln n +2$ for all $n\geq1$. On one hand,
%
\[
\phi(n)-\phi(k)-\phi(n-k) \leq\ln n;
\]
on the other hand,
\[
\phi(n)-\phi(k)-\phi(n-k) \geq1- (\ln k+2) -\bigl(\ln(n-k)+2\bigr)
\geq-3-2\ln
n,
\]
%
and we have the desired inequalities.
\end{pf}

We define a function $f\dvtx [0,1]\to[0,1]$ by
\[
f(r) = \frac{\sigma r}{(\sigma-1)r+1} 
\]
and a function $I_\ell\dvtx [0,1]^4\to[0,+\infty]$ by
\[
I_\ell(r,s,\beta,t) = I \bigl(f(r),s \bigr)+ s I \biggl(M_H(0,0),
\frac{\beta}{s} \biggr)+ (1-s) I \biggl(M_H(\theta,0),
\frac{t-\beta}{1-s} \biggr).
\]
The function $I_\ell$
depends on $\ell$ through the mutation probabilities
$M_H(0,0)$ and
$M_H(\theta,0)$.
Using Lemma~\ref{cnk}
and the expression of $p^\theta$, we see that
\begin{eqnarray*}
&& \forall h,i,j,k\in\{ 0,\ldots,m \}
\\[-2pt]
&&\qquad \ln p(h,i,j,k) = -mI \biggl( f
\biggl(\frac{h}{m} \biggr), \frac{i}{m} \biggr) -iI \biggl(
M_H(0,0), \frac{j}{i} \biggr)
\\[-1pt]
&&\phantom{\qquad \ln p(h,i,j,k) =}{} -(m-i)I \biggl( M_H(\theta,0), \frac{k-j}{m-i} \biggr) +
\Phi(h,i,j,k,m)
\\[-1pt]
&&\phantom{\qquad \ln p(h,i,j,k)} = -mI_\ell\biggl( \frac{h}{m},
\frac{i}{m}, \frac{j}{m}, \frac{k}{m} \biggr)+\Phi(h,i,j,k,m),
\end{eqnarray*}
where the error term $\Phi(h,i,j,k,m)$ satisfies
\[
\forall h,i,j,k\in\{ 0,\ldots,m \}\qquad\bigl|\Phi(h,i,j,k,m) \bigr| \leq6\ln m+9.
\]
%
In the asymptotic regime,
for $\theta=1$ or $\theta=\ell$,
we have
$M_H(0,0)\to e^{-a}$,
$M_H(\theta,0)\to0$,
so that,
for $r,s,\beta,t\in[0,1]^4$,
\[
I_\ell(r,s,\beta,t) \to\cases{ I(r,s,t), &\quad if $\beta= t$,
\cr
+
\infty, &\quad if $\beta\neq t$,} %
\]
where the function
$I(r,s,t)$ is given by
\[
\forall r,s,t\in[0,1]^3\qquad I(r,s,t) = I \bigl(f(r),s \bigr)+ s I
\biggl(e^{-a}, \frac{t}{s} \biggr).
\]
%

\begin{proposition}\label{p1pgd}
We define a function $V_1$ on $[0,1]\times[0,1]$ by
\[
\forall r,t\in[0,1]\qquad V_1(r,t) = \inf\bigl\{ I (r,s,t )\dvtx  s
\in[0,1] \bigr\}. %
\]
The one step transition probabilities of
$(Z_n)_{n\geq0}$
satisfy the large deviation principle governed by $V_1$:
for $r\in[0,1]$ and
any subset $U$ of $[0,1]$,
we have,
for any $n\geq0$,
\begin{eqnarray*}
&& -\inf\bigl\{ V_1(r,t)\dvtx  t\in\uro\bigr\}
\\
&&\qquad \leq\mathop{\liminf
_{\ell,m\to\infty, q\to0}}_{{\ell q} \to a } \frac{1}{m}\ln P
\bigl(Z_{n+1}\in mU | Z_n= {\lfloor rm\rfloor} \bigr),
\\
&& \mathop{\limsup_{\ell,m\to\infty, q\to0}}_{{\ell q} \to a } \frac
{1}{m}\ln P
\bigl(Z_{n+1}\in mU | Z_n= {\lfloor rm\rfloor} \bigr)
 \leq
-\inf\bigl\{ V_1(r,t)\dvtx  t\in\overline{U} \bigr\}.
\end{eqnarray*}
\end{proposition}

\begin{pf}
Let $r\in[0,1]$, and let $U$ be a subset of $[0,1]$.
For any $n\geq0$,
\begin{eqnarray*}
P \bigl(Z_{n+1}\in mU | Z_n= {\lfloor rm
\rfloor} \bigr) &=& \sum_{k\in mU\cap\{ 0,\ldots,m \}} p \bigl( {\lfloor rm
\rfloor},k \bigr)
\\
&=& \mathop{\sum_{k\in\{ 0,\ldots,m \}}}_{k\in mU }
\sum_{i=0}^{m} \sum
_{
j=0
}^i p \bigl( {\lfloor rm\rfloor},i,j,k \bigr).
\end{eqnarray*}
From the previous inequalities, we have
\begin{eqnarray*}
&& P \bigl(Z_{n+1}\in mU | Z_n= {\lfloor rm\rfloor} \bigr)
\\
&&\qquad \leq(m+1)^3\max\bigl\{ p \bigl( {\lfloor rm\rfloor},i,j,k
\bigr)\dvtx  0\leq i\leq m, 0\leq j\leq i, k\in mU \bigr\}
\\
&&\qquad \leq m^{11}
\exp\biggl( -m\min\biggl\{ I_\ell\biggl( \frac{\lfloor rm\rfloor}{
m},
\frac{i}{m}, \frac{j}{m}, \frac{k}{m} \biggr)\dvtx
0\leq j\leq i\leq m, 
{k}\in mU \biggr\}
\biggr).
\end{eqnarray*}
For each $m\geq1$, let $i_m,j_m,k_m$ be three integers in $\{
0,\ldots,m \}$
which realize the above minimum.
By compactness of $[0,1]$, up to the extraction of a subsequence, we can
suppose that, as $m$ goes to $\infty$,
$ {i_m}/{m}\to s$,
${j_m}/{m}\to\beta$,
${k_m}/{m}\to t$.
If $\beta< t$, then
\begin{eqnarray*}
&& \mathop{\limsup_{\ell,m\to\infty,
q\to0}}_{{{\ell q} \to a
}} -I_\ell
\biggl( \frac{\lfloor rm\rfloor}{
m}, \frac{i_m}{m}, \frac{j_m}{m},
\frac{k_m}{m} \biggr)
\\
&&\qquad \leq
\mathop{\limsup_{\ell,m\to\infty,
q\to0
}}_{{{\ell q} \to a
}
}
- \biggl(1- \frac{i_m}{m} \biggr) I \biggl(M_H(\theta,0),
\frac{({k_m}/{m})-({j_m}/{m})
}{1-({i_m}/{m})} \biggr) 
=
-\infty
\end{eqnarray*}
because
\[
\mathop{\limsup_{\ell,m\to\infty,
q\to0
}}_{{\ell q} \to a
} - \frac{k_m-j_m}{m}
\ln\frac{{(k_m-j_m)/{m}}
}{
(
1-({i_m}/{m})
)
M_H(\theta,0)
} = -\infty. %
\]
Thus we need only to consider the case where $\beta=t$. We have then
%
\[
\mathop{\limsup_{\ell,m\to\infty,
q\to0
}}_{{\ell q} \to a
} -I_\ell
\biggl( \frac{\lfloor rm\rfloor}{
m}, \frac{i_m}{m}, \frac{j_m}{m},
\frac{k_m}{m} \biggr) 
\leq-I \bigl(f(r),s \bigr)- s I
\biggl(e^{-a}, \frac{t}{s} \biggr).
\]
%
This implies the large deviation upper bound
%
\begin{eqnarray*}
\hspace*{-3pt}&&\mathop{\limsup_{\ell,m\to\infty,
q\to0
}}_{{{\ell q} \to a
}
} \frac{1}{m}\ln P
\bigl(Z_{n+1}\in mU | Z_n= {\lfloor rm\rfloor} \bigr)
\leq-\inf\bigl\{ 
I (r,s,t )\dvtx s\in[0,1], t\in\overline{U} \bigr\}.
\end{eqnarray*}
%
Conversely, let $s,t\in[0,1]$. We have
%
\begin{eqnarray*}
&& P \bigl(Z_{n+1}= {\lfloor tm\rfloor} | Z_n= {\lfloor rm
\rfloor} \bigr)
\\
&&\qquad \geq p \bigl( {\lfloor rm\rfloor}, {\lfloor sm\rfloor},
{\lfloor
tm\rfloor}, {\lfloor tm\rfloor} \bigr)
\\
&&\qquad \geq\frac{1}{m^7} \exp\biggl( -m
I_\ell\biggl( \frac{\lfloor rm\rfloor}{m}, \frac{\lfloor sm\rfloor}{m},
\frac{\lfloor tm\rfloor}{m}, \frac{\lfloor tm\rfloor}{m} \biggr) \biggr)
\\
&&\qquad \geq\frac{1}{m^7}
\exp\biggl( -m I \biggl( f \biggl(\frac{\lfloor rm\rfloor}{m} \biggr),
\frac{\lfloor sm\rfloor}{m}
\biggr) - {\lfloor sm\rfloor} I \biggl(M_H(0,0), \frac{\lfloor tm\rfloor}{
\lfloor sm\rfloor}
\biggr)
\\
&&\hspace*{174pt}{}- \bigl(m-{\lfloor sm\rfloor} \bigr) \ln\frac{1}{1-M_H(\theta,0)} \biggr).
\end{eqnarray*}
Taking $\ln$ and sending $m,\ell$ to $\infty$, we obtain
\[
\mathop{\liminf_{\ell,m\to\infty,
q\to0
}}_{{\ell q} \to a
} \frac{1}{m}\ln P
\bigl(Z_{n+1}= {\lfloor tm\rfloor} | Z_n= {\lfloor rm
\rfloor} \bigr) \geq- I (r,s,t ). %
\]
Suppose now that $t$ belongs to $\uro$, the interior of $U$.
For $m$ large enough, the integer
${\lfloor tm\rfloor}$ belongs to $mU$.
From the previous estimate, we have
\[
\mathop{\liminf_{\ell,m\to\infty,
q\to0
}}_{{\ell q} \to a
} \frac{1}{m}\ln P
\bigl(Z_{n+1}\in mU | Z_n= {\lfloor rm\rfloor} \bigr) \geq
- I (r,s,t ).
\]
Optimizing over $s,t$, we get
the large deviation lower bound
\begin{eqnarray*}
\hspace*{-4pt}&& \mathop{\liminf_{\ell,m\to\infty,
q\to0
}}_{{\ell q} \to a
} \frac{1}{m}\ln P
\bigl(Z_{n+1}\in mU | Z_n= {\lfloor rm\rfloor} \bigr)
\geq-\inf\bigl\{ 
I (r,s,t )\dvtx s\in[0,1], t\in{\uro} \bigr\}.
\end{eqnarray*}
This finishes the proof of
the large deviation principle.
\end{pf}

Proceeding in the same way, we can prove that the
$l$-step transition
probabilities satisfy a large deviation principle.
For $l\geq1$, we define a function $V_l$ on $[0,1]\times[0,1]$ by
\begin{eqnarray*}
&&V_l(r,t) = \inf\Biggl\{ {\sum_{k=0}^{l-1}}
I (\rho_k,\gamma_k,\rho_{k+1} )\dvtx
\rho_0=r, \rho_{l}=t,
\rho_k,
\gamma_k\in[0,1]\mbox{ for }0\leq k<l \Biggr\}.
\end{eqnarray*}
%

\begin{corollary}\label{ppgd}
For $l\geq1$, the $l$-step transition probabilities of
$(Z_n)_{n\geq0}$ satisfy the large deviation principle governed by $V_l$:
for any subset $U$ of $[0,1]$, any $r\in[0,1]$, we have,
for any $n\geq0$,
\begin{eqnarray*}
&& -\inf\bigl\{ V_l(r,t)\dvtx  t\in\uro\bigr\}
\\
&&\qquad \leq \mathop{\liminf
_{\ell,m\to\infty, q\to0}}_
{{\ell q} \to a } \frac{1}{m}\ln P
\bigl(Z_{n+l}\in mU | Z_n= {\lfloor rm\rfloor} \bigr),
\\
&& \mathop{\limsup_{\ell,m\to\infty, q\to0 }}_
{{\ell q} \to a } \frac{1}{m}\ln P
\bigl(Z_{n+l}\in mU | Z_n= {\lfloor rm\rfloor} \bigr) \leq
-\inf\bigl\{ V_l(r,t)\dvtx  t\in\overline{U} \bigr\}.
\end{eqnarray*}
\end{corollary}

Let us examine when the rate function $I(r,s,t)$ vanishes.
We see that
\[
I (r,s,t ) = 0\quad\Longleftrightarrow\quad s=f(r),\qquad e^{-a}=
\frac{t}{s}.
\]
Let us define a function $F\dvtx [0,1]\to[0,1]$ by
\[
\forall r\in[0,1]\qquad F(r) = e^{-a}f(r) = \frac{\sigma
re^{-a}}{(\sigma-1)r+1}.
\]
The Markov chain
$(Z_n)_{n\geq0}$
can be considered as a random perturbation of the dynamical system
associated to the map $F$
\[
z_0\in[0,1],\ \forall n\geq1\qquad z_n =
F(z_{n-1}).
\]
%
Let us set
\[
\rho^*(a) = \cases{ \displaystyle\frac{\sigma e^{-a}-1}{\sigma-1},
&\quad if $\sigma
e^{-a}>1$,
\vspace*{3pt}\cr
0,&\quad if $\sigma e^{-a}\leq1$.}
\]
Since $F$ is nondecreasing, the sequence $(z_n)_{n\in\mathbb N}$ is monotonous
and it converges to a fixed point of $F$.
If $\sigma e^{-a}\leq1$,
the function $F$ admits only one fixed point, $0$, and
$(z_n)_{n\in\mathbb N}$ converges to $0$.
If $\sigma e^{-a}> 1$,
the function $F$ admits two fixed points,
$0$
and $\rho^*(a)$. If $z_0>0$, then
$(z_n)_{n\in\mathbb N}$ converges to $\rho^*(a)$.

The natural strategy
to study the Markov chain
$(Z_n)_{n\geq0}$
is to use the
Freidlin--Wentzell theory \cite{FW}.
The crucial quantity to analyze the dynamics is the following
cost function~$V$.
We define, for $s,t\in[0,1]$,
\begin{eqnarray*}
V(s,t) &=& \inf_{l\geq1} V_l(s,t)
\\
&=&
\inf_{l\geq1} \inf\Biggl\{ \sum_{k=0}^{l-1}
I (\rho_k,\gamma_k,\rho_{k+1} )\dvtx
\rho_0=s, \rho_{l}=t, \rho_k,
\gamma_k\in[0,1]
\\[-3pt]
&&\hspace*{191pt}\mbox{ for }0\leq k<l \Biggr\}.
\end{eqnarray*}

%
%
\begin{lemma}\label{vpro}
Suppose that $\sigma e^{-a}>1$.
For
$s,t\in[0,1]$, we have $V(s,t)=0$ if and only if:
\begin{itemize}
\item
either $s=t=0$,

\item
or there exists $l\geq1$ such that $t=F^l(s)$,

\item
or $s\neq0, t=\rho^*(a)$.
\end{itemize}
\end{lemma}

\begin{pf}
Throughout the proof we write $\rho^*$ instead of $\rho^*(a)$.
Let $s,t\in[0,1]$ be such that $V(s,t)=0$.
Suppose first that $s=0$. Since $I(0,\gamma,\rho)=+\infty$ unless
$\gamma=\rho=0$,
any sequence
$(\rho_0,\gamma_0,\ldots,\gamma_l)$ such that
$\rho_0=s=0$
and
%
\[
\sum_{k=0}^{l-1} I (\rho_k,
\gamma_k,\rho_{k+1} ) < +\infty
\]
has to be the null sequence, so that necessarily $t=0$.
We suppose next that $s>0$.
For each $n\geq1$, let
$(\rho^n_0,\gamma^n_0,\ldots,\rho^n_{l(n)})$ be a sequence
of length $l(n)$ in $[0,1]$ such that
\[
\rho^n_0=s,  \rho^n_{l(n)}=t,
\qquad\sum_{k=0}^{l(n)-1} I \bigl(
\rho^n_k,\gamma^n_k,
\rho^n_{k+1} \bigr) \leq\frac{1}{n}.
\]
We consider two cases.
If the sequence $(l(n))_{n\geq1}$ is bounded, then
we can extract a subsequence
\[
\bigl(\rho^{\phi(n)}_0,\gamma^{\phi(n)}_0,
\ldots, \rho^{\phi(n)}_{l({\phi(n)})} \bigr)
\]
such that
$l({\phi(n)})=l$ does not depend on $n$, and
for any $k\in\{ 0,\ldots,l-1 \}$, the following limits exist:
\[
\lim_{n\to\infty} \rho^{\phi(n)}_k =
\rho_k,\qquad\lim_{n\to\infty} \gamma^{\phi(n)}_k
= \gamma_k.
\]
The map $I$ being continuous, we have then
\[
\forall k\in\{ 0,\ldots,l-1 \}\qquad I (\rho_k,\gamma_k,
\rho_{k+1} ) = 0,
\]
whence
\[
\forall k\in\{ 0,\ldots,l \} \qquad\rho_k = F^k(
\rho_0).
\]
Since in addition
$\rho_0=s$ and $\rho_l=t$, we conclude that
$t=F^{l}(s)$.
Suppose next that the sequence $(l(n))_{n\geq1}$ is not bounded.
Our goal is to show that $t=\rho^*$.
Using Cantor's diagonal procedure, we can extract a subsequence
\[
\bigl(\rho^{\phi(n)}_0,\gamma^{\phi(n)}_0,
\ldots, \rho^{\phi(n)}_{l({\phi(n)})} \bigr)
\]
such that,
for any $k\geq0$, the following limits exist:
\[
\lim_{n\to\infty} \rho^{\phi(n)}_k =
\rho_k,\qquad\lim_{n\to\infty} \gamma^{\phi(n)}_k
= \gamma_k.
\]
The map $I$ being continuous, we have then
\[
\forall k\geq0\qquad I (\rho_k,\gamma_k,
\rho_{k+1} ) = 0,
\]
whence
\[
\forall k\geq0\qquad\rho_k = F^k(\rho_0).
\]
We have
$ {I (\rho^*,f(\rho^*),\rho^* )} = 0$.
Let $\varepsilon>0$.
The map $I$ being continuous, there exists a neighborhood $U$ of
$\rho^*$
such that
\[
\forall\rho\in U\qquad V_1\bigl(\rho^*,\rho\bigr) \leq I \bigl(
\rho^*,f\bigl(\rho^*\bigr),\rho\bigr) < \varepsilon.
\]
Since $s>0$,
the sequence $(F^n(s))_{n\in\mathbb N}$ converges to $\rho^*$
and
$F^h(s)\in U$ for some $h\geq1$.
In particular,
\[
\lim_{n\to\infty} \rho^{\phi(n)}_h =
F^h(s) \in U,
\]
so that, for $n$ large enough,
$\rho^{\phi(n)}_h$ is in $U$
and
\[
V\bigl(\rho^*,t\bigr) \leq V_1 \bigl(\rho^*, \rho^{\phi(n)}_h
\bigr)+ V \bigl( \rho^{\phi(n)}_h,t \bigr) \leq\varepsilon+
\frac{1}{n}.
\]
Letting successively $n$ go to $\infty$ and $\varepsilon$
go to $0$,
we obtain that
$V(\rho^*,t)=0$.
Let $\delta\in\,]0,\rho^*/2[$, and let
$U=\,]\rho^*-\delta,\rho^*+\delta[$.
Let $\alpha$ be the infimum
\[
\alpha= \inf\bigl\{ I (\rho_0,\gamma_0,
\rho_{1} )\dvtx  \rho_0\in\overline{U}, 
\gamma_0\in[0,1], \rho_{1}\notin U 
\bigr\}.
\]
Since $I$ is continuous on
the compact set
$\overline{U}\times[0,1]\times
([0,1]\setminus U )$, then
\[
\exists\bigl(\rho^*_0,\gamma^*_0,\rho^*_{1}
\bigr)\in\overline{U}\times[0,1]\times\bigl([0,1]\setminus U \bigr
)\qquad
\alpha= I \bigl(\rho^*_0,\gamma^*_0,
\rho^*_{1} \bigr).
\]
The function $F$ is nondecreasing and continuous, therefore
\[
F (\overline{U} ) = F \bigl( \bigl[\rho^*-\delta,\rho^*+\delta\bigr]
\bigr) =
\bigl[ F\bigl(\rho^*-\delta\bigr),F\bigl(\rho^*+\delta\bigr) \bigr].
\]
Moreover we have
\[
\rho^*-\delta< F\bigl(\rho^*-\delta\bigr) \leq F\bigl(\rho^*+\delta\bigr
) <
\rho^*+\delta.
\]
Thus
$F(\overline U) \subset U$
and
necessarily
$\rho^*_{1}\neq F(\rho^*_0)$ and $\alpha>0$.
It follows that any sequence
$(\rho_0,\gamma_0,\ldots,\rho_l)$ such that
\[
\rho_0\in U,\qquad\sum_{k=0}^{l-1}
I (\rho_k,\gamma_k,\rho_{k+1} ) < \alpha
\]
is trapped in $U$. As a consequence, a point $t$
satisfying
$V(\rho^*,t)=0$ must belong to $U$. This is true for
any $\delta>0$, hence for any neighborhood of $\rho^*$,
thus $t=\rho^*$.
\end{pf}

%
\subsection{Persistence time}
We recall that
\[
\tau_0 = \inf\{ n\geq0\dvtx  Z_n=0 \}. %
\]
In this section, we will estimate the expected hitting time $\tau_0$
starting from a point of $\{ 1,\ldots,m \}$.
This
quantity
approximates
the persistence time of the master
sequence~$w^*$.
%

\begin{proposition}\label{exta}
Let $a \in\,]0,+\infty[$ and let $i\in\{ 1,\ldots,m \}$.
The expected hitting time $\tau_0$ of $0$ starting from $i$ satisfies
\[
\mathop{\lim
_{\ell,m\to\infty}}_
{q\to0,
{\ell q} \to a} \frac{1}{m}\ln E(
\tau_0 | Z_0=i) = V\bigl(\rho^*(a),0\bigr).
\]
\end{proposition}

\begin{pf}
Before proceeding to the proof, let us explain the general strategy,
which comes directly from the theory of Freidlin and Wentzell.
To obtain the upper bound on the persistence time, we show that,
starting from any point in $\{ 1,\ldots,m \}$,
the probability to reach a neighborhood of $0$ in a finite
number of steps is larger than
\[
\exp\bigl( -mV\bigl(\rho^*,0\bigr)-m\varepsilon\bigr).
\]
This way we can bound from above $\tau_0$ by
a geometric law with this parameter; see Lemma~\ref{tup}.
To obtain the lower bound on the persistence time, we first show
in Lemma~\ref{lbed} that, starting from any point, the process
has a reasonable probability of reaching any neighborhood
of $\rho^*$ before visiting~$0$.
This estimate is quite tedious because the process might
start from
$Z_0=1$, which is close to the unstable fixed point of $F$.
Since we need to control the hitting time of~$0$ starting from
any point, such an estimate seems to be indispensable, and it
cannot be done in the more general situations considered by
Kifer \cite{KI}
or
Morrow and Sawyer \cite{MS}
without adding some extra assumptions.
So
we give a lower bound on the probability of following the iterates
of a discrete approximation of $F$.
With a Poisson fluctuation, the process jumps away from $0$, then,
because
$F$
is expanding in the neighborhood of $0$, it reaches the
point $\eta m$ after $\ln m$ steps, for some $\eta>0$,
and with a finite number of additional steps, it lands in a neighborhood
of $\rho^*$.
We study then the excursions of the process outside
a neighborhood
of $0$ and $\rho^*$.
Whenever the process is outside such a neighborhood, it reenters the
neighborhood in a finite number of steps with
probability larger than $1-\exp(-cm)$ for some $c>0$
depending on the neighborhood.
Thus the process is very unlikely to
stay a long time outside a neighborhood of the two
attractors $\{ 0,\rho^* \}$.
In fact, the length of an excursion outside
a neighborhood of
$\{ 0,\rho^* \}$
is bounded by a constant, up to
a very unlikely event.
We consider the hitting time $\tau_\delta$ of the
$\delta$-neighborhood of $0$. Obviously we have
$\tau_0\geq
\tau_\delta$.
We focus on the portion of the
trajectory which starts at the last visit to a neighborhood of
$\rho^*$ before reaching a neighborhood of~$0$.
Such an excursion occurs at a given time with probability
less than
\[
\exp\bigl( -mV\bigl(\rho^*,0\bigr)+m\varepsilon\bigr),
\]
and therefore it is unlikely to occur before time
$\exp(
mV(\rho^*,0)-m\varepsilon)$.

We start now with the implementation of this scheme.
Throughout the proof we write $\rho^*$ instead of $\rho^*(a)$.
We start by proving an upper bound on the hitting time.
The next argument works in both cases
$\sigma e^{-a}\leq1$ and
$\sigma e^{-a}> 1$. In the case
$\sigma e^{-a}\leq1$, we have $\rho^*=0$ and
$V(\rho^*,0)=0$, and
the proof becomes simpler, so there is no need to consider a path
from
$\rho^*$ to~$0$.
We have
$ {I (\rho^*,f(\rho^*),\rho^* )}=0$.
Let $\varepsilon>0$.
The map $I$ being continuous, there exists $\delta>0$ such that
\[
\forall\rho\in\bigl]\rho^*-\delta,\rho^*+\delta\bigr[\qquad I \bigl(\rho,f(\rho
),\rho^*
\bigr) < \varepsilon.
\]
Moreover
the sequence $(F^n(1))_{n\in\mathbb N}$ converges to $\rho^*$,
thus
\[
\exists h\geq1\qquad F^h(1)\in\bigl]\rho^*-\delta,\rho^*+\delta\bigr[.
\]
Let $l\geq1$ and
let
$(\rho_0,\gamma_0,\ldots,\rho_{l})$ be a sequence in $[0,1]$ such that
\[
\rho_0=\rho^*, \rho_l=0, \qquad\sum
_{k=0}^{l-1} I (\rho_k,
\gamma_k,\rho_{k+1} ) \leq V\bigl(\rho^*,0\bigr)+
\varepsilon.
\]
We consider the sequence obtained by concatenating the two previous sequences
\begin{eqnarray*}
t_0&=&1,\qquad
s_0=f(1),\qquad
t_1=F(1),\qquad\ldots,\qquad t_{h-1}=F^{h-1}(1),
\\
s_{h-1}&=&f(t_{h-1}),
\\
t_h&=&F^h(1),\qquad s_h=f(t_h),\qquad t_{h+1}=\rho^*,\qquad s_{h+1}=\gamma_0,
\\
t_{h+2}&=&\rho_1,\qquad\ldots,\qquad t_{h+l}=
\rho_{l-1},\qquad s_{h+l}=\gamma_{l-1},\qquad t_{h+l+1}=
\rho_l=0.
\end{eqnarray*}
We set $j=h+l+1$.
This sequence satisfies
\[
t_0=1, t_{j}=0, \qquad\sum
_{k=0}^{j-1} I (t_k,s_k,t_{k+1}
) \leq V\bigl(\rho^*,0\bigr)+3\varepsilon.
\]
We have then
\[
P(Z_j=0 | Z_0=m) \geq\prod
_{k=0}^{j-1} p \bigl( {\lfloor mt_k
\rfloor}, {\lfloor ms_k\rfloor}, {\lfloor mt_{k+1}\rfloor},
{\lfloor mt_{k+1}\rfloor} \bigr). %
\]
Taking $\ln$, sending $m$ to $\infty$
and using the estimate on the transition probabilities
obtained in the proof of
{Proposition}~\ref{p1pgd},
we have
\begin{eqnarray*}
&&\mathop{\liminf_{\ell,m\to\infty,
q\to0}}_{{\ell q} \to a
} \frac{1}{m}\ln
P(Z_j=0 | Z_0=m) \geq-\sum
_{k=0}^{j-1} I (t_k,s_k,t_{k+1}
)
\geq-V\bigl(\rho^*,0\bigr)-3\varepsilon.
\end{eqnarray*}
Thus, asymptotically, we have
\[
P(Z_j=0 | Z_0=m) \geq\exp\bigl( -mV\bigl(\rho^*,0
\bigr)-4m\varepsilon\bigr).
\]
Using the monotonicity
of
the Markov chain
$(Z_n)_{n\geq0}$, we conclude that,
asymptotically,
\[
\forall i\in\{ 1,\ldots,m \}\qquad P(Z_j=0 | Z_0=i)
\geq\exp\bigl( -mV\bigl(\rho^*,0\bigr)-4m\varepsilon\bigr).
\]
We have thus a lower bound on the probability of reaching $0$ in $j$ steps
starting from any point in $\{ 1,\ldots,m \}$.
For any $n\geq0$, we have, using the Markov property,
\begin{eqnarray*}
&& P \bigl(\tau_0>(n+1)j | Z_0=m \bigr)
\\[-1pt]
&&\qquad = \sum
_{h=1}^m P \bigl(\tau_0>(n+1)j,
Z_{nj}=h | Z_0=m \bigr)
\\[-1pt]
&&\qquad = \sum_{h=1}^m P (
\tau_0>nj, Z_{nj}=h, Z_{nj+1}\neq0, \ldots,
Z_{(n+1)j}\neq0 | Z_0=m )
\\[-1pt]
&&\qquad = \sum_{h=1}^m P (Z_{nj+1}
\neq0, \ldots, Z_{(n+1)j}\neq0 | \tau_0>nj,
Z_{nj}=h, Z_0=m )
\\[-1pt]
&&\hspace*{47pt}{}\times P (\tau_0>nj, Z_{nj}=h | Z_0=m )
\\[-1pt]
&&\qquad = \sum_{h=1}^m P (
\tau_0>j | Z_0=h ) P (\tau_0>nj,
Z_{nj}=h | Z_0=m )
\\[-1pt]
&&\qquad \leq \bigl(1-\exp\bigl( -mV\bigl(\rho^*,0\bigr)-4m\varepsilon\bigr
) \bigr)
P (\tau_0>nj | Z_0=m ).
\end{eqnarray*}
Iterating this inequality, we obtain the following result.
%

\begin{lemma}
\label{tup} For any $\varepsilon>0$, there exists $j\geq1$ such that
\[
\forall n\geq0\qquad P (\tau_0>nj | Z_0=m ) \leq
\bigl(1-\exp\bigl( -mV\bigl(\rho^*,0\bigr)-4m\varepsilon\bigr)
\bigr)^n.
\]
\end{lemma}

It follows that
\begin{eqnarray*}
E(\tau_0 | Z_0=m) &=& 
\sum_{n\geq0} \sum_{k=nj+1}^{(n+1)j}
P (\tau_0\geq k | Z_0=m )
\\
&\leq&\sum
_{n\geq0} j 
P (\tau_0> nj |
Z_0=m ) \leq j\exp\bigl( mV\bigl(\rho^*,0\bigr)+4m\varepsilon
\bigr),
\end{eqnarray*}
whence
\[
\mathop{\limsup_{\ell,m\to\infty,
q\to0
}}_{{{\ell q} \to a
}
} \frac{1}{m}\ln
E(\tau_0 | Z_0=m) \leq V\bigl(\rho^*,0\bigr)+4
\varepsilon. %
\]
Letting $\varepsilon$ go to $0$ yields the desired upper bound.

We compute next a lower bound on the hitting time.
If $\sigma e^{-a}\leq1$, then $\rho^*=0$,
$V(\rho^*,0)=0$,
and obviously,
\[
\mathop{\liminf_{\ell,m\to\infty,
q\to0
}}_{
{{\ell q} \to a
}
} \frac{1}{m}\ln
E(\tau_0 | Z_0=m) \geq V\bigl(\rho^*,0\bigr) = 0.
\]
Thus we need only to consider
the case $\sigma e^{-a}>1$.
We start by estimating from below the probability of going from $1$ to
a neighborhood of $\rho^*$ without visiting~$0$. Before proceeding
with the mathematical details, let us explain the strategy to get
this lower bound.
When $Z_0=1$,
the binomial law
involved in the replication mechanism
can be approximated by a Poisson law
of parameter~$\sigma$, and the process
$(Z_n)_{n\geq0}$ can jump
to any fixed $h\in{\mathbb N}$ with a probability larger than a
positive quantity
independent of~$m$.
Using a simple estimate on the central term of the binomial law,
we have that
\[
P \bigl(Z_{n+1}=G(h) | Z_n=h \bigr) \geq
\frac{1}{(m+1)^{2}},
\]
where
$G$ is a map from $\{ 0,\ldots,m \}$ to $\{ 0,\ldots,m \}$ such that
\[
\frac{1}{m} G(h) \geq F \biggl(\frac{h}{m} \biggr) -
\frac{1}{m}.
\]
We study then the iterates of the function $F(x)-1/m$.
This function, which is a small perturbation of $F$, has two fixed
points, one unstable close to $0$, of order $1/m$, and one
stable close to $\rho^*$.
We take $h$ large enough so that $h/m$ is larger than the unstable fixed
point. Then the repulsive dynamics of
$F(x)-1/m$ will bring the point $h/m$ close to a value $\eta>0$
(independent of $m$) in a number of iterates of order $\ln m$.
Once the process
$(Z_n)_{n\geq0}$ is at $\lfloor\eta m\rfloor$, a finite
number of iterates leads into the neighborhood of $\rho^*$.
The lower bound is obtained by combining the three steps
\[
P(1\to h)P(h\to\eta m) P\bigl(\eta m\to\bigl(\rho^*-\delta\bigr) m\bigr
) \geq c
\biggl(\frac{1}{(m+1)^{2}} \biggr)^{{c \ln m}+{c}},
\]
where $c$ is a constant independent of $m$. This is the
idea of the proof of the next lemma.

%
\begin{lemma}
\label{lbed}
For any $\delta>0$, there exists $c>0$, depending on $\delta$, such that,
asymptotically,
%
\[
P \bigl(Z_1>0,\ldots,Z_{\lfloor c\ln m\rfloor-1}>0,
Z_{\lfloor c\ln m\rfloor} >m\bigl(\rho^*-\delta\bigr) | Z_0=1 \bigr)
\geq
\frac{1}{m^{c\ln m}}.
\]
\end{lemma}

\begin{pf}
The binomial law ${\mathcal B}(n,p)$ of parameters $n\geq0$ and
$p<1$ is maximal at $\lfloor(n+1)p\rfloor$, therefore
%
\[
\pmatrix{ {n}
\cr
\bigl\lfloor(n+1)p\bigr\rfloor} p^{
\lfloor(n+1)p\rfloor}
(1-p)^{
n-\lfloor(n+1)p\rfloor} \geq\frac{1}{n+1}.
\]
See, for instance, Chapter VI in Feller's book \cite{FE}.
We shall use this inequality to bound from below the transition
probabilities of the Markov chain
$(Z_n)_{n\geq0}$. Let us define a map $G\dvtx \{ 0,\ldots,m \}\to\{
0,\ldots,m \}$ by
\begin{eqnarray*}
\forall h\in\{ 0,\ldots,m-1 \}\qquad G(h) &=& \biggl\lfloor\biggl( \biggl
\lfloor(m+1)f \biggl(\frac{h}{m} \biggr) \biggr\rfloor+1
\biggr)M_H(0,0) \biggr\rfloor,
\\
G(m) &=& \bigl\lfloor(m+1)M_H(0,0)\bigr\rfloor.
\end{eqnarray*}
The map $G$ depends on the parameters $m,\ell$ and $q$.
Applying the previous lower bound to the binomial laws involved in the
transition step of
$(Z_n)_{n\geq0}$, we obtain
\begin{eqnarray*}
\forall n,h\geq0\qquad && P \bigl(Z_{n+1}\geq G(h) | Z_n= h
\bigr)
\\
&&\qquad \geq
\sum_{k\geq G(h)} p \biggl(h, \biggl\lfloor(m+1)f \biggl(
\frac{h}{m} \biggr) \biggr\rfloor, G(h),k \biggr)
\\
&&\qquad
\geq\frac{1}{(m+1)^{2}}.
\end{eqnarray*}
It follows that
for $n,h\geq0$,
%
\begin{eqnarray*}
&& P \bigl( Z_{1}\geq G^{1}(h),\ldots, Z_{n}
\geq G^{n}(h) | Z_0=h \bigr)
\\
&&\qquad= \sum_{l\geq
G^{n-1}(h)
} P \bigl( Z_{1}\geq
G^{1}(h),\ldots, Z_{n-1}=l, Z_{n}\geq
G^{n}(h) | Z_0=h \bigr)
\\
&&\qquad= \sum_{l\geq
G^{n-1}(h)
} P \bigl( Z_{n}\geq
G^{n}(h) | Z_{n-1}=l \bigr)
\\
&&\hspace*{71pt}{} \times P \bigl( Z_{1}\geq G^{1}(h),\ldots,
Z_{n-1}=l | Z_0=h \bigr)
\\
&&\qquad\geq P \bigl( Z_{n}\geq G^{n}(h) |
Z_{n-1}= G^{n-1}(h) \bigr)
\\
&&\quad\qquad{} \times P \bigl( Z_{1}\geq G^{1}(h),\ldots,
Z_{n-1}\geq G^{n-1}(h) | Z_0=h \bigr)
\\
&&\qquad\geq\frac{1}{(m+1)^{2}} P \bigl( Z_{1}\geq
G^{1}(h),\ldots, Z_{n-1}\geq G^{n-1}(h) |
Z_0=h \bigr).
\end{eqnarray*}
Iterating this inequality, we obtain, for
$n,h\geq0$,
\begin{eqnarray*}
&&P \bigl( Z_{1}\geq G^{1}(h),\ldots, Z_{n}\geq
G^{n}(h) | Z_0=h \bigr) \geq\frac{1}{(m+1)^{2n}}.
\end{eqnarray*}
The map $G$ is nondecreasing.
Moreover, for $h\in\{ 0,\ldots,m \}$,
\[
G(h) \geq\biggl\lfloor(m+1)f \biggl(\frac{h}{m}
\biggr)M_H(0,0) \biggr\rfloor\geq mf \biggl(\frac{h}{m}
\biggr)M_H(0,0)-1.
\]
Let us define a map $H\dvtx [0,1]\to[0,1]$ by
\[
\forall x\in[0,1]\qquad H(x) = e^aM_H(0,0)F(x)-
\frac{1}{m}.
\]
We can rewrite the previous inequality as
\[
\forall h\in\{ 0,\ldots,m \}\qquad G(h) \geq m H \biggl(\frac{h}{m}
\biggr).
\]
Iterating this inequality, we get, thanks to the fact that both
$G$ and $H$ are nondecreasing,
\[
\forall n\geq0,\ \forall h\in\{ 0,\ldots,m \}\qquad G^{n}(h) \geq
m H^n \biggl(\frac{h}{m} \biggr).
\]
The map $H$, which is a small perturbation of the map $F$, has two fixed
points $\rho'<\rho''$, whose expansion is given by
\begin{eqnarray*}
\rho' &=& \frac{1}{m(\sigma M_H(0,0)-1)}+o \biggl(\frac{1}{m} \biggr),
\\
\rho'' &=& \frac{\sigma M_H(0,0)-1}{\sigma-1} -\frac{\sigma
M_H(0,0)}{m(\sigma M_H(0,0)-1)} +o
\biggl(\frac{1}{m} \biggr).
\end{eqnarray*}
Notice that $M_H(0,0)$ converges to $e^{-a}$, so
$\rho'$ is close to $0$ and $\rho''$ is close to $\rho^*$.
Let $\eta>0$. If $x\leq\eta$, we have
$F(x)\geq\alpha x$, where
\[
\alpha= \frac{\sigma e^{-a}}{(\sigma-1)\eta+1}.
\]
For $\eta$ sufficiently small, we have $\alpha>1$ and the map $F$
restricted to
$[0,\eta]$
is expanding.
Let $\beta=(1+\alpha)/2$. Asymptotically, we have
$\alpha e^a M_H(0,0)\geq\beta$.
Let us study the iterates of $x$ through the
map $H$. We set
\[
N = \inf\bigl\{ n\geq0\dvtx  H^n(x)>\eta\bigr\}.
\]
For $1\leq n<N$, we have then
\[
H^{n}(x) = H \bigl(H^{n-1}(x) \bigr) \geq\beta
H^{n-1}(x) 
- \frac{1}{m}, %
\]
which we rewrite as
\[
\frac{1}{\beta^{n}} H^{n}(x) \geq\frac{1}{\beta^{n-1}}
H^{n-1}(x) -\frac{1}{m\beta^{n-1}}. %
\]
Summing from $n=1$ to $N-1$, we get
\[
H^{N-1}(x) \geq{\beta^{N-1}} \Biggl( x-\frac{1}{m}\sum
_{n=0}^{N-2}\frac{1}{\beta^{n}} \Biggr) \geq{
\beta^{N-1}} \biggl( x-\frac{\beta}{m(\beta-1)} \biggr). %
\]
Let $h$ be an integer such that
\[
h \geq2\frac{\beta}{\beta-1}.
\]
Notice that this condition does not depend on $m$. We suppose that
$m>h$. We take $x=h/m$, and we denote by $N(h)$ the
associated integer.
From the previous inequalities,
we have then
\[
\eta\geq H^{N(h)-1} \biggl( \frac{h}{m} \biggr) \geq{
\beta^{N(h)-1}} \frac{h}{2m}. %
\]
Thus $N(h)$ satisfies
\[
N(h) \leq1+\frac{1}{\ln\beta}\ln\frac{2m\eta}{h},
\]
%
and we have, asymptotically,
\begin{eqnarray*}
&& P (Z_1>0,\ldots,Z_{N(h)-1}>0, Z_{N(h)} >m\eta|
Z_0=h )
\\
&&\qquad \geq P \biggl( Z_{1}\geq m H^{1}
\biggl(\frac{h}{m} \biggr),\ldots, Z_{{N(h)}}\geq m H^{{N(h)}}
\biggl(\frac{h}{m} \biggr) \Big| Z_0=h \biggr)
\\
&&\qquad \geq P \bigl(
Z_{1}\geq G^{1}(h),\ldots, Z_{{N(h)}}\geq
G^{{N(h)}}(h) | Z_0=h \bigr)
\\
&&\qquad \geq\frac{1}{(m+1)^{2{N(h)}}}.
\end{eqnarray*}
We control next the probability to go from $1$ to $h$.
We have
\[
P(Z_1\geq h | Z_0=1) \geq\pmatrix{ {m}
\cr
{h}}
\frac{ \sigma^h
(m-1)^{m-h}}{ (\sigma-1+m )^m} M_H(0,0)^h.
\]
In this regime, where $h$ is fixed and $m$ is large, the binomial law
involved in the replication mechanism can be approximated by a Poisson law
of parameter
$\sigma$, whence, asymptotically,
\[
P(Z_1\geq h | Z_0=1) \geq
\frac{1}{2} \exp(-\sigma)\frac{\sigma^h}{h!}\exp(-ah).
\]
We control finally the probability to go from $\eta m$ to
the neighborhood of~$\rho^*$. We do this by following the iterates
of $F$ starting from $\eta$, and by controlling the error term
with respect to the iterates of $H$.
%

\begin{lemma}
\label{taylor}
We suppose that $\sigma e^{-a}>1$.
For $n\geq0$, $x\in[0,1]$,
we have
%
\[
H^{n}(x) \geq\bigl(e^aM_H(0,0)F
\bigr)^{n}(x) - \frac{1}{m} \frac{
(\sigma M_H(0,0))^{n+1}}{\sigma M_H(0,0)-1}.
\]
\end{lemma}

\begin{pf}
We have
\[
\forall x\in[0,1]\qquad\bigl|F'(x) \bigr| \leq\sigma e^{-a},
\]
and, for any $n\geq0$,
\[
H^{n+1}(x) = H \bigl(H^n(x) \bigr) = e^aM_H(0,0)
F \bigl(H^n(x) \bigr) -\frac{1}{m}. %
\]
We shall prove the following inequality by induction on $n$:
\[
H^{n}(x) \geq\bigl( e^aM_H(0,0) F
\bigr)^{n}(x) - \frac{1}{m} \sum_{k=1}^{n}
\bigl(\sigma M_H(0,0)\bigr)^{k}.
\]
The inequality is true for $n=0,1$.
Suppose that the inequality holds for some $n\geq0$.
Since $F$ is nondecreasing, we deduce from the inequality
on $F'$
and the mean value theorem
that
\begin{eqnarray*}
H^{n+1}(x) &\geq& e^aM_H(0,0) F \Biggl(
\bigl( e^aM_H(0,0) F \bigr)^{n}(x) -
\frac{1}{m} \sum_{k=1}^{n} \bigl(
\sigma M_H(0,0)\bigr)^{k} \Biggr) -\frac{1}{m}
\\
& \geq& \bigl( e^aM_H(0,0) F \bigr)^{n+1}(x)
- \sigma M_H(0,0) \frac{1}{m} \sum
_{k=1}^{n} \bigl(\sigma M_H(0,0)
\bigr)^{k} -\frac{
1
}{m}
\\
&\geq& \bigl( e^aM_H(0,0) F \bigr)^{n+1}(x) -
\frac{1}{m} \sum_{k=1}^{n+1} \bigl(
\sigma M_H(0,0)\bigr)^{k},
\end{eqnarray*}
and the inequality is proved at rank $n+1$.
Summing the geometric series, we obtain the inequality stated
in the lemma.
\end{pf}

Let $\delta>0$.
The sequence $(F^n(\eta))_{n\in\mathbb N}$ converges to $\rho^*$,
thus
$F^t(\eta)
>\rho^*-\delta$ for some $t\geq1$.
For $m$ large enough, we have also
\[
\bigl(e^aM_H(0,0) F \bigr)^{t}(\eta) -
\frac{1}{m} \frac{(\sigma M_H(0,0))^{t+1}}{
\sigma M_H(0,0)-1} > \rho^*-\delta,
\]
and Lemma~\ref{taylor} implies that
$H^{t}(\eta)
>\rho^*-\delta$.
Let $i$ be an integer strictly larger than~$\eta m$.
We have
\begin{eqnarray*}
&& P \bigl(Z_1>0,\ldots,Z_{t-1}>0, Z_t >m\bigl(
\rho^*-\delta\bigr) | Z_0=i \bigr)
\\
&&\qquad \geq P \bigl( Z_{1}
\geq m H^{1}(\eta),\ldots, Z_{t}\geq m H^{t} (
\eta) | Z_0=i \bigr)
\\
&&\qquad \geq P \biggl( Z_{1}\geq m
H^{1} \biggl(\frac{i}{m} \biggr),\ldots, Z_{t}\geq
m H^{t} \biggl(\frac{i}{m} \biggr) \Big| Z_0=i \biggr)
\\
&&\qquad \geq P \bigl( Z_{1}\geq G^{1}(i),\ldots,
Z_{t}\geq G^{t}(i) | Z_0=i \bigr)
\\
&&\qquad  \geq
\frac{1}{(m+1)^{2t}}.
\end{eqnarray*}
To conclude, we use the monotonicity of
$(Z_n)_{n\geq0}$, and we combine the three previous estimates.
The values
$h,t$ do not depend on $m$, and
there exists a positive constant $c$ depending on $\eta,h$
such that, asymptotically,
\begin{eqnarray*}
N(h)+t+1 &<& c\ln m,
\\
\biggl(\frac{1}{2} \exp(-\sigma)\frac{\sigma^h}{h!}\exp(-ah)
\biggr)^{c\ln m} \frac{1}{(m+1)^{2N(h)+2t}} &\geq&\frac{1}{m^{c\ln m}}.
\end{eqnarray*}
Let us set
$s=
\lfloor c\ln m\rfloor
- (N(h)+t )$. Recall that $t$ depends on $\eta,\delta$
and $h$ depends on $\eta$.
We have
\begin{eqnarray*}
\hspace*{-4pt}&& P\bigl(Z_1>0,\ldots, Z_{
\lfloor c\ln m\rfloor-1
}>0, Z_{\lfloor c\ln m\rfloor
} >m\bigl(\rho^*-\delta\bigr) | Z_0=1\bigr)
\\
\hspace*{-4pt}&&\qquad \geq
\sum_{j\geq h} \sum_{i>m\eta}
P\bigl(Z_1\geq h,\ldots,Z_{s-1}\geq h,Z_s=j,
Z_{s+1}>0,\ldots,Z_{s+N(h)-1}>0,
\\[-1pt]
\hspace*{-4pt}&&\hspace*{85pt} Z_{s+N(h)}=i, Z_{s+N(h)+1}>0,\ldots, Z_{s+N(h)+t-1}>0,
\\
\hspace*{-4pt}&&\hspace*{217pt} Z_{s+N(h)+t} >m\bigl(\rho^*-\delta\bigr) | Z_{0}=1 \bigr)
\\
\hspace*{-4pt}&&\qquad \geq\sum_{j\geq h} \sum_{i>m\eta}
P(Z_1\geq h,\ldots,Z_{s-1}\geq h,Z_s=j |
Z_0=1)
\\[-1pt]
\hspace*{-4pt}&&\hspace*{69pt} {}\times P (Z_{s+1}>0,\ldots,Z_{s+N(h)-1}>0, Z_{s+N(h)}=i
| Z_s=j )
\\
\hspace*{-4pt}&&\hspace*{69pt} {}\times  P \bigl(Z_{s+N(h)+1}>0,\ldots, Z_{s+N(h)+t-1}>0,
\\
\hspace*{-4pt}&&\hspace*{100pt} Z_{s+N(h)+t} >m\bigl(\rho^*-\delta\bigr) |
Z_{s+N(h)}=i \bigr)
\\
\hspace*{-4pt}&&\qquad \geq P(Z_1\geq h,\ldots,Z_s\geq h | Z_0=1)
\\
\hspace*{-4pt}&&\quad\qquad{}\times
\sum_{i>m\eta} P (Z_1>0,
\ldots,Z_{N(h)-1}>0,
Z_{N(h)}=i
| Z_0=h )
\\
\hspace*{-4pt}&&\hspace*{65pt} {}\times P \bigl(Z_{1}>0,
\ldots,Z_{t-1}>0, Z_{t} >m\bigl(\rho^*-\delta\bigr) |
Z_{0}=i \bigr)
\\
\hspace*{-4pt}&&\qquad \geq\bigl(P(Z_1\geq h | Z_0=1) \bigr)^s
\\
\hspace*{-4pt}&&\quad\qquad{} \times
P (Z_1>0,\ldots,Z_{N(h)-1}>0, Z_{N(h)} >m\eta|
Z_0=h ) \frac{1}{(m+1)^{2t}}
\\
\hspace*{-4pt}&&\qquad \geq\biggl(\frac{1}{2} \exp(-\sigma)\frac{\sigma^h}{h!}\exp(-ah)
\biggr)^s \frac{1}{(m+1)^{2N(h)+2t}} \geq\frac{1}{m^{c\ln m}}.
\end{eqnarray*}
This is the required lower bound.
\end{pf}

Whenever the starting point is far away from~$0$, the estimate
of Lemma~\ref{lbed} can be considerably enhanced,
as shown in the next lemma.
%

\begin{lemma}
\label{lbd}
We suppose that $\sigma e^{-a}>1$.
For any $\delta>0$, there exist $h\geq1$ and $c>0$, depending on
$\delta$,
such that, asymptotically,
\[
P \bigl(Z_1>0,\ldots,Z_{h-1}>0,
Z_h >m\bigl(\rho^*-\delta\bigr) | Z_0= {\lfloor m\delta
\rfloor} \bigr) \geq1-\exp(-cm).
\]
\end{lemma}

\begin{pf}
Let $\delta>0$.
The sequence $(F^n(\delta))_{n\in\mathbb N}$ converges to $\rho^*$.
Thus
there exists $h\geq1$ such that
$F^h(\delta)>\rho^*-\delta$.
By continuity of the map $F$, there exist
$\rho_0,\rho_1,\ldots,\rho_h>0$ such that
$\rho_0=\delta$,
$\rho_h>\rho^*-\delta$ and
\[
\forall k\in\{ 1,\ldots,h \}\qquad F(\rho_{k-1}) > \rho_k.
\]
Now,
\begin{eqnarray*}
&& P \bigl(Z_1>0,\ldots,Z_{h-1}>0, Z_h >m\bigl(
\rho^*-\delta\bigr) | Z_0= {\lfloor m\delta\rfloor}
\bigr)
\\
&&\qquad \geq
P
\bigl( \forall k\in\{ 1,\ldots,h \},  Z_k\geq m\rho_{k}
| Z_0= {\lfloor m\delta\rfloor} \bigr).
\end{eqnarray*}
Passing to the complementary event, we have
\begin{eqnarray*}
&& P \bigl(\exists k\in\{ 1,\ldots,h-1 \}, Z_k=0 \mbox{ or }
Z_h \leq m\bigl(\rho^*-\delta\bigr) | Z_0= {
\lfloor m\delta\rfloor} \bigr)
\\
&&\qquad \leq P \bigl( \exists k\in\{ 1,\ldots,h \}, Z_k< m
\rho_{k} 
| Z_0= {\lfloor m\delta\rfloor}
\bigr)
\\
&&\qquad \leq\sum_{1\leq k\leq h} P \bigl( Z_{1} \geq m
\rho_{1},\ldots, Z_{k-1} \geq m\rho_{k-1},
Z_k< m\rho_{k} | Z_0= {\lfloor m\delta
\rfloor} \bigr)
\\
&&\qquad \leq\sum_{1\leq k\leq h} \sum_{i\geq
m\rho_{k-1}
}P \bigl( Z_{k-1} =i, Z_k< m\rho_{k} |
Z_0= {\lfloor m\delta\rfloor} \bigr)
\\
&&\qquad \leq\sum_{1\leq k\leq h} \sum_{i\geq
m\rho_{k-1}}
P ( Z_k< m\rho_{k} | Z_{k-1} =i ) P \bigl(
Z_{k-1} =i | Z_0= {\lfloor m\delta\rfloor} \bigr)
\\
&&\qquad \leq\sum_{1\leq k\leq h} P \bigl( Z_1< m
\rho_{k} | Z_{0} = {\lfloor m\rho_{k-1} \rfloor}
\bigr).
\end{eqnarray*}
The large deviation principle for the transition probabilities of
the Markov chain
$(Z_n)_{n\geq0}$
stated in
Proposition~\ref{p1pgd}
implies that
for $k\in\{ 1,\ldots,h \}$,
\begin{eqnarray*}
&& \mathop{\limsup_{\ell,m\to\infty,
q\to0}}_
{{\ell q} \to a
} \frac{1}{m}\ln P
\bigl( Z_1< m\rho_{k} | 
Z_{0} = {
\lfloor m\rho_{k-1} \rfloor} \bigr)
\\
&&\qquad \leq-\inf\bigl\{ 
I
(\rho_{k-1},s,t )\dvtx s\in[0,1], t\leq\rho_{k} \bigr\} < 0.
\end{eqnarray*}
Since $h$ is fixed, we conclude that
\[
\mathop{\limsup_{\ell,m\to\infty, q\to0 }}_{{{\ell q} \to a } } \frac
{1}{m}\ln P
\left( \matrix{\exists k\in\{ 1,\ldots,h-1 \},
Z_k=0
\vspace*{3pt}\cr
\mbox{or } Z_h \leq m\bigl(\rho^*-\delta
\bigr)} \Big| Z_0= {\lfloor m\delta\rfloor} \right) < 0, %
\]
and this yields the desired estimate.
\end{pf}

With the estimate of Lemma~\ref{lbd}, we show that the process
is very unlikely to stay a long time in
$[m\delta,
m(\rho^*-\delta)]$.
%

\begin{corollary}
\label{cbd}
We suppose that $\sigma e^{-a}>1$.
Let $\delta>0$. There exist $h\geq1$ and $c>0$ such that,
asymptotically,
%
\begin{eqnarray*}
&&\forall k\in\bigl[m\delta, m\bigl(\rho^*-\delta\bigr)\bigr],\ \forall n\geq0
\\
&&\qquad P \bigl( m\delta\leq Z_t\leq m\bigl(\rho^*-\delta\bigr)
\mbox{ for }0\leq t\leq n | Z_0=k \bigr) \leq\exp\biggl(-cm \biggl
\lfloor\frac{n}{h} \biggr\rfloor\biggr).
\end{eqnarray*}
\end{corollary}

\begin{pf}
Let
$k\in[m\delta,
m(\rho^*-\delta)]$.
Let $\delta>0$, and
let $h\geq1$ and $c>0$ be associated
to $\delta$ as in
Lemma~\ref{lbd}.
We divide the interval $\{ 0,\ldots,n \}$ into subintervals of length
$h$, and
we use repeatedly the estimate of Lemma~\ref{lbd}.
Let $i\geq0$. We write
\begin{eqnarray*}
&& P \bigl( m\delta\leq Z_t\leq m\bigl(\rho^*-\delta\bigr)
\mbox{ for }0\leq t\leq(i+1)h | Z_0=k \bigr)
\\
&&\qquad =
\sum_{\delta m\leq j
\leq(\rho^*-\delta)m
} P \bigl( m\delta\leq Z_t
\leq m\bigl(\rho^*-\delta\bigr)
\\
&&\hspace*{104pt} \mbox{ for }0\leq t\leq(i+1)h,
Z_{ih}=j | Z_0=k \bigr)
\\
&&\qquad = \sum_{\delta m\leq j
\leq(\rho^*-\delta)m
} P \bigl( m\delta\leq Z_t
\leq m\bigl(\rho^*-\delta\bigr) \mbox{ for }0\leq t\leq ih,
Z_{ih}=j | Z_0=k \bigr)
\\
&&\hspace*{92pt}{}\times P \bigl( m\delta\leq Z_t\leq m\bigl(\rho^*-\delta\bigr)
\mbox{ for }ih\leq t\leq(i+1)h | Z_{ih}=j \bigr)
\\
&&\qquad \leq\sum_{\delta m\leq j
\leq(\rho^*-\delta)m
} P \bigl( m\delta\leq
Z_t\leq m\bigl(\rho^*-\delta\bigr) \mbox{ for }0\leq t\leq ih,
Z_{ih}=j | Z_0=k \bigr)
\\
&&\hspace*{92pt}{} \times P \bigl(Z_h \leq m\bigl(\rho^*-\delta\bigr) |
Z_0= {\lfloor m\delta\rfloor} \bigr)
\\
&&\qquad \leq P \bigl( m\delta\leq Z_t\leq m\bigl(\rho^*-\delta\bigr)
\mbox{ for }0\leq t\leq ih | Z_0=k \bigr) \exp(-cm).
\end{eqnarray*}
Iterating this inequality, we obtain
\[
\forall i\geq0\qquad P \bigl( m\delta\leq Z_t\leq m\bigl(\rho^*-
\delta\bigr) \mbox{ for }0\leq t\leq ih | Z_0=k \bigr) \leq
\exp(-cmi). %
\]
The claim of the corollary follows by applying this inequality with
$i$ equal to the integer part of $n/h$.
\end{pf}

We have computed the relevant estimates to reach the neighborhood
of $\rho^*$. Our next goal is to study the hitting
time $\tau_0$ starting from a neighborhood of $\rho^*$.
Since we need only a lower bound, we shall study the hitting time
of a neighborhood of~$0$.
For $\delta>0$, we define
\[
\tau_\delta= \inf\{ n\geq0\dvtx  Z_{n}<m\delta\}.
\]
Let $i > (\rho^*-\delta)m$.
We shall estimate the expectation of $\tau_\delta$
starting from $i$.
The strategy consists of looking
at the portion of the
trajectory starting at the last visit to the neighborhood of
$\rho^*$ before reaching the neighborhood of~$0$.
Accordingly,
we define
\[
S = \max\bigl\{ n\leq\tau_\delta\dvtx  Z_{n}> \bigl(\rho^*-
\delta\bigr)m \bigr\}.
\]
Notice that $S$ is not a Markov time.
We write, for $n,k\geq1$,
\begin{eqnarray*}
&& P (\tau_\delta\leq n | Z_0=i )
\\
&&\qquad = \sum
_{1\leq s<t\leq n} P ( \tau_\delta=t, S=s | Z_{0} =i
)
\\
&&\qquad = \mathop{\sum_{1\leq s<t\leq n}}_{
{s<t\leq s+k }} P (
\tau_\delta=t, S=s | Z_{0} =i )+ \mathop{\sum
_{1\leq s< n}}_{
{s+k<t\leq n }} P ( \tau_\delta=t, S=s |
Z_{0} =i ).
\end{eqnarray*}
Let $h\geq1$ and $c>0$ be associated
to $\delta$ as in
Corollary~\ref{cbd}.
For
${1\leq s< n}$ and $t>s+k$,
\begin{eqnarray*}
&& P ( \tau_\delta=t, S=s | Z_{0} =i )
\\
&&\qquad = \sum
_{m\delta\leq j
\leq(\rho^*-\delta)m
} P ( \tau_\delta=t, S=s, Z_{s+1}=j |
Z_{0} =i )
\\
&&\qquad \leq\sum_{m\delta\leq j
\leq(\rho^*-\delta)m
} P \left(
\matrix{ \delta m\leq Z_r\leq\bigl(\rho^*-\delta\bigr)m
\vspace*{3pt}\cr
\mbox{for }s+1\leq r\leq t-1} \Big| Z_{s+1}=j \right)
\\
&&\qquad
\leq m \exp
\biggl(-cm \biggl\lfloor\frac{t-s-2}{h} \biggr\rfloor\biggr),
\end{eqnarray*}
whence
\[
\mathop{\sum_{1\leq s< n}}_{s+k<t\leq n } P (
\tau_\delta=t, S=s | Z_{0} =i ) \leq n \sum
_{t\geq k} m\exp\biggl(-cm \biggl\lfloor\frac{t-1}{h}
\biggr\rfloor\biggr). %
\]
For
${1\leq s<t\leq n}$
and $t\leq s+k$,
\begin{eqnarray*}
&& P ( \tau_\delta=t, S=s | Z_{0} =i )
\\
&&\qquad \leq\sum
_{j> (\rho^*-\delta)m
} P ( \tau_\delta=t, S=s, Z_s=j |
Z_{0} =i )
\\
&&\qquad \leq\sum_{j
> (\rho^*-\delta)m
}
P
( Z_t<\delta m | Z_{s} =j )
\\
&&\qquad
\leq m P \bigl(
Z_{t-s}<\delta m | Z_{0} = {\bigl\lfloor\bigl(\rho^*-
\delta\bigr)m \bigr\rfloor} \bigr),
\end{eqnarray*}
whence
\begin{eqnarray*}
&&\mathop{\sum_{1\leq s< n}}_{
{s<t\leq s+k }} P (
\tau_\delta=t, S=s | Z_{0} =i ) \leq
n\sum_{1\leq t\leq k} m P \bigl( Z_{t}<\delta m |
Z_{0} = {\bigl\lfloor\bigl(\rho^*-\delta\bigr)m \bigr\rfloor} \bigr).
\end{eqnarray*}
Putting together the previous inequalities, we obtain
\begin{eqnarray*}
P (\tau_\delta\leq n | Z_0=i ) &\leq& n \sum
_{t\geq k} m\exp\biggl(-cm \biggl\lfloor\frac{t-1}{h}
\biggr\rfloor\biggr)
\\
&&{} + n\sum_{1\leq t\leq k} m P \bigl( Z_{t}<\delta m
| Z_{0} = {\bigl\lfloor\bigl(\rho^*-\delta\bigr)m \bigr\rfloor}
\bigr).
\end{eqnarray*}
We choose $k$ large enough so that
\[
\mathop{\limsup_{\ell,m\to\infty, q\to0 }}_
{{\ell q} \to a } \frac{1}{m}\ln
\biggl( \sum_{t\geq k} m\exp\biggl(-cm \biggl\lfloor
\frac{t-1}{h} \biggr\rfloor\biggr) \biggr) < -V\bigl(\rho^*-\delta,\delta
\bigr),
\]
and we use the large deviation principle stated in
Corollary~\ref{ppgd}
to estimate the second sum,
\begin{eqnarray*}
&& \mathop{\limsup_{\ell,m\to\infty, q\to0 }}_
{{\ell q} \to a } \frac{1}{m}\ln
\biggl( \sum_{1\leq t\leq k} m P \bigl( Z_{t}<\delta
m | Z_{0} = {\bigl\lfloor\bigl(\rho^*-\delta\bigr)m \bigr\rfloor}
\bigr) \biggr)
\\
&&\qquad \leq- \min_{1\leq t\leq k} V_t\bigl(\rho^*-\delta,\delta
\bigr) \leq-V\bigl(\rho^*-\delta,\delta\bigr).
\end{eqnarray*}
Applying the previous inequalities with
$n=\exp(
mV(\rho^*-\delta,\delta)-m\delta)$, we conclude that
\[
\mathop{\lim_{\ell,m\to\infty, q\to0 }}_
{{\ell q} \to a } P \bigl(
\tau_\delta\leq\exp\bigl(mV\bigl(\rho^*-\delta,\delta\bigr)-m\delta
\bigr) | Z_0=i \bigr) = 0
\]
and therefore
\[
\mathop{\liminf_{\ell,m\to\infty, q\to0}}_
{{\ell q} \to a
} \frac{1}{m}\ln
E(\tau_\delta| Z_0=i) \geq V\bigl(\rho^*-\delta,\delta
\bigr)-\delta.
\]
To derive a lower bound on the expectation of $\tau_0$
starting from~$1$,
we combine
the previous estimates as follows.
By Lemma~\ref{lbed},
asymptotically,
\[
P \bigl(Z_1>0,\ldots,Z_{
{\lfloor c\ln m\rfloor}
-1}>0, Z_{\lfloor c\ln m\rfloor} >m\bigl(
\rho^*-\delta\bigr) | Z_0=1 \bigr) \geq\frac{1}{m^{c\ln m}}.
\]
Thus,
letting $i={\lfloor
(\rho^*-\delta)m
\rfloor}+1$,
for any $n\geq
{\lfloor c\ln m\rfloor}$,
\begin{eqnarray*}
&& P(\tau_0>n | Z_0=1)
\\
&&\qquad \geq
\sum
_{j\geq i} P ( Z_1>0,\ldots,Z_{
{\lfloor c\ln m\rfloor}
-1}>0,
Z_{\lfloor c\ln m\rfloor} 
=j, \tau_0>n | Z_0=1 )
\\
&&\qquad
\geq
\sum_{j\geq i} P (
Z_1>0,\ldots,Z_{
{\lfloor c\ln m\rfloor}
-1}>0, Z_{\lfloor c\ln m\rfloor} 
=j |
Z_0=1 )
\\
&&\hspace*{45pt}{} \times P (Z_{
{\lfloor c\ln m\rfloor}
+1}>0,\ldots, Z_n>0 | Z_{\lfloor c\ln m\rfloor} =j
)
\\
&&\qquad
\geq
P \bigl( Z_1>0,\ldots,Z_{
{\lfloor c\ln m\rfloor}
-1}>0,
Z_{\lfloor c\ln m\rfloor} >m\bigl(\rho^*-\delta\bigr) | Z_0=1 \bigr)
\\
&&\quad\qquad {} \times P (Z_{
{\lfloor c\ln m\rfloor}
+1}>0,\ldots, Z_n>0 | Z_{\lfloor c\ln m\rfloor} =i
) 
\\
&&\qquad \geq\frac{1}{m^{c\ln m}} P \bigl(
\tau_0>n- {\lfloor c\ln m\rfloor} | Z_0=i \bigr).
\end{eqnarray*}
Summing from $n= {\lfloor c\ln m\rfloor}$ to $+\infty$, we get
\[
E(\tau_0 | Z_0=1) \geq\frac{1}{m^{c\ln m}} E(
\tau_0 | Z_0=i).
\]
The very definition of $\tau_\delta$ implies that $\tau_0\geq\tau
_\delta$, whence
\[
E(\tau_0 | Z_0=i) \geq E(\tau_\delta|
Z_0=i).
\]
From the lower bound on $\tau_\delta$ and the previous inequalities,
we deduce that
\[
\mathop{\liminf_{\ell,m\to\infty,
q\to0
}}_{{\ell q} \to a
} \frac{1}{m}\ln
E(\tau_0 | Z_0=1) \geq V\bigl(\rho^*-\delta,\delta
\bigr)-\delta.
\]
The conclusion follows by letting $\delta$ go to $0$.
\end{pf}

\subsection{Concentration near \texorpdfstring{$\rho^*$}{$rho^*$}}
In this section, we estimate
the numerator of the last formula
of Section~\ref{bounds}.
As usual, we drop the superscript $\theta$ from the notation when it
is not necessary,
and we put it back when we need to emphasize the differences between
the cases
$\theta=\ell$ and $\theta=1$.
Let
$f\dvtx [0,1]\to{\mathbb R}$
be a nondecreasing continuous function
such that $f(0)=0$.
Our goal here is to estimate the expected value of the sum
\[
\sum_{n=0}^{\tau_0} f \biggl(
\frac{Z_n}{m} \biggr). %
\]
%
The Markov chain
$(Z_n)_{n\geq0}$ is a perturbation of the dynamical system associated
to the
map $F$, and therefore it spends most of its time in the neighborhood
of the
stable fixed point $\rho^*$. On very large time intervals, the process
visits points far away from $\rho^*$, and then it returns quickly to
$\rho^*$.
From this picture, we conclude that
the fraction of the time spent away from $\rho^*$ is
negligible.
We will show that~the above sum is, on average, comparable to
$f(\rho^*)\tau_0$.
%

\begin{proposition}\label{exfa}
We suppose that $\sigma e^{-a}>1$.
We have,
uniformly with respect to $i\in\{ 1,\ldots,m \}$,
%
\[
\mathop{\lim
_{\ell,m\to\infty}}_{q\to0,
{\ell q} \to a} \frac{E (
\sum_{n=0}^{\tau_0}
f ({Z_n}/{m} ) | Z_0=i )}{
E (\tau_0 | Z_0=i )} = f\bigl(\rho^*
\bigr). %
\]
\end{proposition}

\begin{pf}
Before proceeding to the proof, let us explain the general strategy,
which comes directly from the theory of Freidlin and Wentzell.
For $\delta>0$, we denote by $U(\delta)$
the $\delta$-neighborhood of $\rho^*$,
\[
U(\delta) = \bigl]\rho^*-\delta,\rho^*+\delta\bigr[.
\]
%
We choose $\delta$ small enough,
so that when the process is in
$U(2\delta)$, the value $f(Z_n/m)$ is approximated by $f(\rho^*)$.
When the process is outside of $\{ 0 \}\cup U(2\delta)$,
it reenters $U(\delta)$
in
${\lfloor c\ln m\rfloor}$ steps with probability
at least
${m^{-c\ln m}}$, for some $c>0$; see Lemma~\ref{ulbed}.
Therefore the average length of an excursion is bounded by~${m^{c\ln m}}$. At a given time, the probability to start an
excursion from $U(\delta)$ reaching the outside of
$U(2\delta)$ is less than $\exp(-cm)$, for some $c>0$.
With this estimate we can control the number of these excursions
(see Lemma~\ref{Kex}), and we show that, typically, their total
length until the time $\tau_0$ is negligible compared
to $\tau_0$.

We start now the detailed proof.
Let $\varepsilon>0$. Since $f$ is continuous, there exists $\delta>0$
such that
%
\[
\forall\rho\in U(2\delta)\qquad\bigl|f(\rho)-f\bigl(\rho^*\bigr) \bigr| <
\varepsilon.
\]
We define then a sequence of stopping times
in order to track the excursions of
$(Z_n)_{n\geq0}$ outside $U(\delta)$.
More precisely, we set
$T_0=0$ and
\begin{eqnarray*}
T_1^* &=& \inf\biggl\{ n\geq0\dvtx  \frac{Z_n}{m} \in U(\delta)\biggr\},\qquad
T_1 = \inf\biggl\{ n\geq T_1^*\dvtx \frac{Z_n}{m} \notin U(2\delta) \biggr\},
\\
&\displaystyle \vdots&
\\
T_k^* &=& \inf\biggl\{ n\geq T_{k-1}\dvtx  \frac{Z_n}{m}\in U(\delta) \biggr\},\qquad
T_k = \inf\biggl\{ n\geq T_k^*\dvtx  \frac{Z_n}{m} \notin U(2\delta) \biggr\},
\\
&\displaystyle \vdots&
\end{eqnarray*}
Next, we decompose the sum over the intervals
$[T_{k-1},T_k^*[$,
$[T_{k}^*,T_k[$, $k\geq1$.
Denoting by $s\wedge t$ the minimum $\min(s,t)$, we have
\begin{eqnarray*}
&& \sum_{n=0}^{\tau_0} f \biggl(
\frac{Z_n}{m} \biggr)-f\bigl(\rho^*\bigr) \tau_0
\\
&&\qquad =
\sum
_{k\geq1} \sum_{n=T_{k-1}\wedge\tau_0}^{T_k^*\wedge\tau_0-1}
\biggl(f \biggl(\frac{Z_n}{m} \biggr) -f\bigl(\rho^*\bigr) \biggr) +
\sum
_{k\geq1} \sum_{n=T_{k}^*\wedge\tau_0}^{T_k\wedge\tau_0-1}
\biggl(f \biggl(\frac{Z_n}{m} \biggr) -f\bigl(\rho^*\bigr) \biggr).
\end{eqnarray*}
We bound next the absolute value as follows:
%
\[
\Biggl|\sum_{n=0}^{\tau_0} f \biggl(
\frac{Z_n}{m} \biggr)-f\bigl(\rho^*\bigr) \tau_0 \Biggr| \leq2f(1)
\sum_{k\geq1} \bigl(T_{k}^*\wedge
\tau_0-{T_{k-1}\wedge\tau_0} \bigr) +
\varepsilon\tau_0. %
\]
%
It remains to deal with the sum.
We define, for $n\geq0$,
\[
K(n) = \max\{ k\geq1\dvtx T_{k-1}< n \},
\]
and the sum becomes
\[
\sum_{k\geq1} \bigl(T_{k}^*\wedge
\tau_0-{T_{k-1}\wedge\tau_0} \bigr) = \sum
_{k=1}^{K(\tau_0)} \bigl(T_{k}^*\wedge
\tau_0-{T_{k-1}} \bigr). %
\]
Let $\eta>0$. We set
\[
t_m^\eta= \exp\bigl(m\bigl(V\bigl(\rho^*,0\bigr)+\eta
\bigr) \bigr).
\]
We decompose the sum as follows:
\[
\sum_{k=1}^{K(\tau_0)}
\bigl(T_{k}^*\wedge\tau_0-{T_{k-1}} \bigr) \leq
1_{\tau_0>t_m^\eta} \tau_0 + 1_{\tau_0\leq t_m^\eta} \sum
_{k= 1}^{K(\tau_0)} \bigl(T_{k}^*\wedge
\tau_0-{T_{k-1}} \bigr).
\]
We suppose that the process starts
from $i\in\{ 1,\ldots,m \}$.
The estimates are carried out exactly in the
same way for any value of $i$, therefore, to alleviate the notation,
we remove the starting point from the notation.
Throughout the proof the expectation~$E$ and the
probability~$P$ are meant with respect to the initial
condition $Z_0=i$.
Taking expectation in the previous inequalities, we get
%
\begin{eqnarray*}
&& \Biggl| E \Biggl(\sum_{n=0}^{\tau_0} f \biggl(
\frac{Z_n}{m} \biggr) \Biggr)-f\bigl(\rho^*\bigr) E (\tau_0 ) \Biggr|
\\
&&\qquad \leq
\varepsilon E (\tau_0 )+ 2f(1) E ( 1_{\tau_0>
t_m^\eta}
\tau_0 ) + 2f(1) E \Biggl( 1_{\tau_0\leq t_m^\eta} \sum
_{k= 1}^{K(\tau_0)} \bigl(T_{k}^*\wedge
\tau_0-{T_{k-1}} \bigr) \Biggr).
\end{eqnarray*}
Next,
we take care of the second term.

%
\begin{lemma}
\label{lex}
For any $N,j\geq1$,
\[
E ( \tau_01_{\tau_0>Nj} ) \leq NjP ( {\tau_0>Nj}
)+\sum_{n\geq N}j P ( {\tau_0>nj} ).
\]
\end{lemma}

\begin{pf}We compute
\begin{eqnarray*}
E ( \tau_01_{\tau_0>Nj} ) &=& \sum_{k>Nj}
k P ( {\tau_0=k} ) = \sum_{k>Nj} \sum
_{n\geq0} 1_{n<k} P ( {\tau_0=k})
\\
&=& \sum_{n\geq0} \mathop{\sum
_{k>Nj}}_{k>n} P ( {\tau_0=k} ) = \sum
_{n\geq0} P \bigl( {\tau_0>\max(Nj,n) }
\bigr)
\\
&\leq& Nj P ( {\tau_0>Nj } )+ \sum
_{n\geq Nj} P ( {\tau_0>n } ).
\end{eqnarray*}
Next,
\[
\sum_{n\geq Nj} P ( {\tau_0>n } ) = \sum
_{n\geq N} \sum_{k=0}^{j-1}
P ( {\tau_0>nj+k } ) \leq\sum_{n\geq N} jP
( {\tau_0>nj } ),
\]
and we have the desired inequality.
\end{pf}

We apply
Lemma~\ref{tup} with $\varepsilon=\eta/8$:
there exists $j\geq1$ such that
\[
\forall n\geq0\qquad P (\tau_0>nj | Z_0=m ) \leq
\bigl(1-\exp\bigl( -mV\bigl(\rho^*,0\bigr)-m\eta/2 \bigr) \bigr)^n.
\]
We apply Lemma~\ref{lex} with this $j$ and
\[
N = \bigl\lfloor t_m^\eta/j\bigr\rfloor= \biggl\lfloor
\frac{1}{j} \exp\bigl( mV\bigl(\rho^*,0\bigr)+m\eta\bigr) \biggr
\rfloor,
\]
and we use the previous inequality
\begin{eqnarray*}
&& E ( \tau_01_{\tau_0>t_m^\eta} )
\\
&&\qquad \leq E ( \tau_01_{\tau_0>Nj}
) \leq NjP ( {\tau_0>Nj} )+\sum_{n\geq N}j
P ( {\tau_0>nj} )
\\
&&\qquad \leq \bigl(Nj+j \exp\bigl( mV\bigl(\rho^*,0
\bigr)+m\eta/2 \bigr) \bigr) \bigl(1- \exp\bigl( -mV\bigl(\rho^*,0\bigr
)-m\eta/2
\bigr) \bigr)^N
\\
&&\qquad \leq (1+j) \exp\bigl( mV\bigl(\rho^*,0\bigr)+m
\eta\bigr) 
\exp\bigl( -N \exp\bigl( -mV\bigl(\rho^*,0\bigr)-m\eta/2
\bigr) \bigr).
\end{eqnarray*}
Thanks to the choice of $N$, this last quantity goes to $0$ as $m$
goes to $\infty$.
Thus
\[
\lim_{m\to\infty} E ( 1_{\tau_0>
t_m^\eta} \tau_0 ) = 0.
\]
We deal now with the last sum in the inequality before
Lemma~\ref{lex}.
We give first an upper bound on $K$.
%

\begin{lemma}
\label{Kex}
We suppose that $\sigma e^{-a}>1$.
There exists $c>0$, depending on~$\delta$,
such that,
asymptotically,
\[
\forall k,n\geq0\qquad P \bigl(K(n)> k \bigr) \leq
\frac{n^{k}}{k!} 
\exp(-cmk).
\]
\end{lemma}

\begin{pf}
For $k\geq0$, we define
\[
S_k = \sup\biggl\{ T_{k}^*\leq n< T_{k}\dvtx
\frac{Z_n}{m} \in U(\delta) 
\biggr\}.
\]
For $k,n\geq0$, we have
\[
P \bigl(K(n)> k \bigr) = P (T_k<n ) = \sum
_{
t^*\leq s<t< n
} P \bigl(T_k^*=t^*,S_k=s,T_k=t
\bigr). %
\]
Let $h\geq1$ and $c>0$ be associated
to $\delta$ as in
Corollary~\ref{cbd}.
We can suppose that $h\geq2$.
For given values of $t^*$ and $s$,
we split the sum over $t$ in two parts,
\[
\sum_{
t\dvtx s<t< n
} P \bigl(T_k^*=t^*,S_k=s,T_k=t
\bigr) = \sum_{
t\dvtx t>s+h+1
}\cdots\qquad+ \sum
_{
t\dvtx s<t\leq s+h+1
}\cdots.%
\]
We study next the first sum, when $t>s+h+1$. We condition
on the state at time~$s+1$
\begin{eqnarray*}
\sum_{t\dvtx t>s+h+1
}\cdots &=& \mathop{\sum
_{t\dvtx t>s+h+1}}_
{j\in mU(2\delta)\setminus mU(\delta)} P \bigl(T_k^*=t^*,S_k=s,Z_{s+1}=j,T_k=t
\bigr)
\\
&\leq&\mathop{\sum_{t\dvtx t>s+h+1}}_
{j\in mU(2\delta)\setminus mU(\delta)} P
\pmatrix{ T_k^*=t^*,Z_{s+1}=j, Z_t\notin mU(2
\delta)
\vspace*{3pt}\cr
Z_{s+1},\ldots,Z_{t-1}\in mU(2\delta)\setminus mU(
\delta)}
\\
&=&  \mathop{\sum_{t\dvtx t>s+h+1
}}_
{j\in mU(2\delta)\setminus mU(\delta)}
P \bigl( Z_{s+1},\ldots,Z_{t-1}\in mU(2\delta)\setminus mU(
\delta),
\\[-3pt]
&&\hspace*{136pt} Z_t\notin mU(2\delta) | Z_{s+1}=j \bigr)
\\
&&\hspace*{64pt}{}\times P \bigl( T_k^*=t^*,Z_{s+1}=j \bigr).
\end{eqnarray*}
For
${0\leq s< n}$ and $t>s+h+1$,
\begin{eqnarray*}
&& P \bigl( Z_{s+1},\ldots,Z_{t-1}\in mU(2\delta)\setminus mU(
\delta), Z_t\notin mU(2\delta) | Z_{s+1}=j \bigr)
\\
&&\qquad \leq
P \left( \matrix{ m\delta\leq Z_r\leq m\bigl(\rho^*-\delta\bigr)
\vspace*{3pt}\cr
\mbox{for }s+1\leq r\leq t-1} \Big| Z_{s+1}=j \right)
\\
&&\quad\qquad{} + P \left( \matrix{Z_r\geq m\bigl(\rho^*+\delta\bigr)
\vspace*{3pt}\cr
\mbox{for } s+1\leq r \leq t-1} \Big| Z_{s+1}=j \right)
\\
&&\qquad \leq\exp\biggl(-cm \biggl\lfloor
\frac{t-s-2}{h} \biggr\rfloor\biggr).
\end{eqnarray*}
In fact, in Corollary~\ref{cbd}, we gave an upper bound on the first
probability. Yet the second probability
can be handled in exactly the same
way.
Thus
\[
\sum_{
t\dvtx t>s+h+1
}\cdots\leq
\biggl( \sum_{
t\geq h+1
} \exp\biggl(-cm \biggl\lfloor
\frac{t-1}{h} \biggr\rfloor\biggr) \biggr) P \bigl( T_k^*=t^*
\bigr). %
\]
Let us focus on
the second sum.
We condition
on the state at time $s$
\begin{eqnarray*}
\sum_{
t\dvtx s<t\leq s+h+1
}\cdots &=& \mathop{\sum
_{t\dvtx s<t\leq s+h+1}}_{{j\in mU(\delta)}
} P \bigl(T_k^*=t^*,S_k=s,Z_{s}=j,T_k=t
\bigr)
\\
&\leq&\mathop{\sum_{t\dvtx s<t\leq s+h+1}}_{{j\in mU(\delta)}
} P
\bigl( Z_t\notin mU(2\delta) | Z_{s}=j \bigr) P
\bigl(T_k^*=t^*,Z_{s}=j \bigr)
\\
&\leq&\mathop{\sum
_{t\dvtx 1\leq t\leq h+1}}_{j\in mU(\delta)} P \bigl( Z_t
\notin mU(2\delta) | Z_{0}=j \bigr) P \bigl(T_k^*=t^*,Z_{s}=j
\bigr).
\end{eqnarray*}
%
For any
$j\in mU(\delta)$,
using the monotonicity of
$(Z_n)_{n\geq0}$,
\begin{eqnarray*}
&& P \bigl( Z_t\notin mU(2\delta) | Z_0=j \bigr)
\\
&&\qquad \leq P
\bigl( Z_t\leq m\bigl(\rho^*-2\delta\bigr) | Z_0=j
\bigr)
+ P \bigl( Z_t\geq m\bigl(\rho^*+2\delta\bigr) | Z_0=j
\bigr)
\\
&&\qquad \leq P \bigl( Z_t\leq m\bigl(\rho^*-2\delta\bigr) |
Z_0= \bigl\lfloor\bigl(\rho^*-\delta\bigr)m \bigr\rfloor\bigr)
\\
&&\quad\qquad{} + P \bigl( Z_t\geq m\bigl(\rho^*+2\delta\bigr) | Z_0=
\bigl\lfloor\bigl(\rho^*+\delta\bigr)m \bigr\rfloor\bigr).
\end{eqnarray*}
We use the large deviation principle stated in Corollary~\ref{ppgd}
to estimate the last two terms.
For any $t\in\{ 1,\ldots,h+1 \}$,
\begin{eqnarray*}
&& \mathop{\limsup_{\ell,m\to\infty, q\to0 }}_
{{\ell q} \to a } \frac{1}{m}\ln P
\bigl( Z_t\leq m\bigl(\rho^*-2\delta\bigr) | Z_0= \bigl
\lfloor\bigl(\rho^*-\delta\bigr)m \bigr\rfloor\bigr)
\\
&&\qquad \leq-\inf\bigl\{ V_{t}\bigl(\rho^*-\delta,\rho\bigr)\dvtx \rho\leq
\rho^*-2\delta\bigr\},
\\
&& \mathop{\limsup_{\ell,m\to\infty, q\to0}}_{{\ell q} \to a } \frac
{1}{m}\ln P
\bigl( Z_t\geq m\bigl(\rho^*+2\delta\bigr) | Z_0= \bigl
\lfloor\bigl(\rho^*+\delta\bigr)m \bigr\rfloor\bigr)
\\
&&\qquad \leq-\inf\bigl\{ V_{t}\bigl(\rho^*+\delta,\rho\bigr)\dvtx \rho\geq
\rho^*+2\delta\bigr\}.
\end{eqnarray*}
By compactness, the infima are realized.
Because of the constraints on~$\rho$,
the point $\rho$ realizing the infimum
\[
\inf\bigl\{ V_{t}\bigl(\rho^*-\delta,\rho\bigr)\dvtx \rho\leq\rho^*-2
\delta\bigr\} %
\]
is not an iterate of $\rho^*-\delta$ through $F$. Hence
by Lemma~\ref{vpro},
the above infimum is positive.
We argue in the same way for the second infimum,
and
we conclude that there exists
$c'>0$, depending
on $\delta$, such that,
asymptotically,
\[
\forall{j\in mU(\delta)}\qquad\sum_{t\dvtx 1\leq t\leq h+1} P \bigl(
Z_t\notin mU(2\delta) | Z_{0}=j \bigr) \leq\exp
\bigl(-c'm\bigr), %
\]
whence
\[
\sum_{
t\dvtx s<t\leq s+h+1
}\cdots\leq
\exp\bigl(-c'm\bigr) 
P \bigl(T_k^*=t^*
\bigr). %
\]
Let $c''>0$ be such that, asymptotically,
\[
\sum_{
t\geq h+1
} \exp\biggl(-cm \biggl\lfloor
\frac{t-1}{h} \biggr\rfloor\biggr) +\exp\bigl(-c'm\bigr) \leq
\exp\bigl(-c''m\bigr).
\]
Reporting in the initial equality, we obtain that,
asymptotically,
for any $n,k\geq0$,
\begin{eqnarray*}
P (T_k<n ) 
&\leq&\sum_{
t^*\leq s< n
}
\exp\bigl(-c''m\bigr) P \bigl(T_k^*=t^*
\bigr)
\\
&\leq&\sum_{
s< n
} \exp\bigl(-c''m
\bigr) P \bigl(T_k^*\leq s \bigr)
\\
&\leq& \sum
_{
s< n
} \exp\bigl(-c''m\bigr) P
(T_{k-1}< s ).
\end{eqnarray*}
Iterating this inequality, we obtain
%
\[
P (T_k<n ) \leq\sum_{0\leq n_0<\cdots<n_{k-1}<n} \exp
\bigl(-c''mk\bigr) \leq\frac{n^k}{k!} \exp
\bigl(-c''mk\bigr) %
\]
%
as required.
\end{pf}

\renewcommand{\theequation}{$\bigcirc$}
We estimate now the last sum
in the inequality before
Lemma~\ref{lex}.
By the Cauchy--Schwarz inequality, we have
%
\begin{eqnarray}\label{bigcirc}
&& E \Biggl( 1_{\tau_0\leq t_m^\eta} \sum_{k= 1}^{K(\tau_0)}
\bigl(T_{k}^*\wedge\tau_0-{T_{k-1}} \bigr)
\Biggr) \nonumber
\\
&&\qquad = \sum_{k\geq1} E \bigl( 1_{\tau_0\leq t_m^\eta}
1_{k\leq K(\tau_0)} \bigl(T_{k}^*\wedge\tau_0-{T_{k-1}}
\bigr) \bigr)
\nonumber\\[-8pt]\\[-8pt]
&&\qquad \leq\sum_{k\geq1} P \bigl( \tau_0\leq
t_m^\eta, {K(\tau_0)\geq k}
\bigr)^{1/2} \bigl(E \bigl( 1_{k\leq K(\tau_0)} \bigl(T_{k}^*
\wedge\tau_0-{T_{k-1}} \bigr)^2 
\bigr) \bigr)^{1/2}\nonumber
\\
&&\qquad \leq\sum_{k\geq1} P \bigl( {K\bigl(
t_m^\eta\bigr)\geq k} \bigr)^{1/2} \bigl(E \bigl(
1_{k\leq K(\tau_0)} \bigl(T_{k}^*\wedge\tau_0-{T_{k-1}}
\bigr)^2 \bigr) \bigr)^{1/2}.\nonumber
\end{eqnarray}
If $1\leq k\leq K(\tau_0)$,
then $T_{k-1}<\tau_0$ and
$Z_{T_{k-1}}>0$, so that,
using the Markov property,
\begin{eqnarray*}
&& E \bigl( 1_{k\leq K(\tau_0)} \bigl(T_{k}^*\wedge\tau_0-{T_{k-1}}
\bigr)^2 \bigr)
\\
&&\qquad =
\sum_{1\leq j\leq m} E
\bigl( 1_{k\leq K(\tau_0)} \bigl(T_{k}^*\wedge\tau_0-{T_{k-1}}
\bigr)^2 | Z_{T_{k-1}}=j \bigr) P ( Z_{T_{k-1}}=j )
\\
&&\qquad \leq\sum_{1\leq j\leq m} E \bigl( 
\bigl(T_{1}^*\wedge\tau_0\bigr)^2 |
Z_{0}=j \bigr) P ( Z_{T_{k-1}}=j ).
\end{eqnarray*}
%
We will next bound the time
$T_{1}^*\wedge\tau_0$,
starting from
$j\in\{ 1,\ldots,m \}$.
%

\begin{lemma}
\label{ulbed}
We suppose that $\sigma e^{-a}>1$.
For any $\delta>0$, there exists $c>0$, depending on $\delta$, such that,
asymptotically,
for $j\in\{ 1,\ldots,m \}$,
\[
P \bigl( m\bigl(\rho^*-\delta\bigr)<
Z_{\lfloor c\ln m\rfloor} <m\bigl(\rho^*+\delta\bigr) | Z_0=j \bigr)
\geq
\frac{1}{m^{c\ln m}}.
\]
\end{lemma}

\begin{pf}
Using Lemma~\ref{lbed}, there exists $c>0$ such that,
asymptotically,
%
\[
P \bigl( Z_{\lfloor c\ln m\rfloor} \leq m\bigl(\rho^*-\delta\bigr) | Z_0=1
\bigr) \leq1-\frac{1}{m^{c\ln m}}.
\]
Proceeding as in Lemma~\ref{lbd}, we obtain
that there exist $h,c'>0$ such that, asymptotically,
\[
P \bigl( Z_{h} \geq m\bigl(\rho^*+\delta\bigr) | Z_0=m
\bigr) \leq\exp\bigl(-c'm\bigr).
\]
We have then
\begin{eqnarray*}
&& P \bigl( Z_{\lfloor c\ln m\rfloor} \geq m\bigl(\rho^*+\delta\bigr) | Z_0=m
\bigr)
\\
&&\qquad =
\sum_{j\in\{ 1,\ldots,m \}} P \bigl( Z_{\lfloor c\ln m\rfloor}
\geq m\bigl(\rho^*+\delta\bigr), Z_{\lfloor c\ln m\rfloor-h}=j | Z_0=m
\bigr)
\\
&&\qquad = \sum_{j\in\{ 1,\ldots,m \}} P \bigl( 
Z_h
\geq m\bigl(\rho^*+\delta\bigr) | Z_0=j \bigr) P (
Z_{\lfloor c\ln m\rfloor-h}=j | Z_0=m )
\\
&&\qquad \leq\exp\bigl(-c'm
\bigr).
\end{eqnarray*}
Using the monotonicity of
$(Z_n)_{n\geq0}$, we have
\begin{eqnarray*}
&& P \bigl( Z_{\lfloor c\ln m\rfloor} \notin\bigl]m\bigl(\rho^*-\delta\bigr),
m\bigl(\rho^*+
\delta\bigr)\bigr[\, | Z_0=j \bigr)
\\
&&\qquad \leq P \bigl( Z_{\lfloor c\ln m\rfloor} \leq m\bigl(\rho^*-\delta\bigr) |
Z_0=j \bigr) + P \bigl( Z_{\lfloor c\ln m\rfloor} \geq m\bigl(\rho
^*+\delta
\bigr) | Z_0=j \bigr)
\\
&&\qquad \leq P \bigl( Z_{\lfloor c\ln m\rfloor} \leq m\bigl(\rho^*-\delta\bigr) |
Z_0=1 \bigr) + P \bigl( Z_{\lfloor c\ln m\rfloor} \geq m\bigl(\rho
^*+\delta\bigr) | Z_0=m \bigr)
\\
&&\qquad \leq1-\frac{1}{m^{c\ln m}}+ \exp\bigl(-c'm\bigr).
\end{eqnarray*}
This estimate is uniform with respect to $j\in\{ 1,\ldots,m \}$.
\end{pf}

%
\begin{corollary}
\label{uclbed}
We suppose that $\sigma e^{-a}>1$.
For any $\delta>0$, there exists $c>0$, depending on $\delta$, such that,
asymptotically,
for $j\in\{ 1,\ldots,m \}$,
\[
\forall n\geq0\qquad P \bigl( T_{1}^*\wedge\tau_0 \geq n
{\lfloor c\ln m\rfloor} | Z_{0}=j \bigr) \leq\biggl(1-
\frac{1}{m^{c\ln m}} \biggr)^n. %
\]
\end{corollary}

\begin{pf}
We proceed as in Corollary~\ref{cbd} to obtain this inequality.
We divide the interval
$\{ 0,\ldots,n
{\lfloor c\ln m\rfloor}
\}$ into subintervals of length
${\lfloor c\ln m\rfloor}$, and
we use repeatedly the estimate of Lemma~\ref{ulbed}.
\end{pf}

By {Corollary}~\ref{uclbed}, we have, asymptotically, for any $j\in
\{ 1,\ldots,m \}$,
\begin{eqnarray*}
E \bigl( \bigl(T_{1}^*\wedge\tau_0\bigr)^2 |
Z_{0}=j \bigr) &=& \sum_{k\geq1} P
\bigl(T_{1}^*\wedge\tau_0\geq\sqrt{k} |
Z_{0}=j \bigr)
\\
&\leq&\sum_{k\geq1} \biggl( 1-\frac{1}{m^{c\ln m}}
\biggr)^{ \lfloor {\sqrt{k}}/{
{\lfloor c\ln m\rfloor}}
\rfloor}.
\end{eqnarray*}
Let us set
\[
\alpha= 1-\frac{1}{m^{c\ln m}},\qquad t= {\lfloor c\ln m\rfloor}.
\]
We have
\begin{eqnarray*}
&&\sum_{k\geq1} \alpha^{\lfloor\sqrt{k}/t\rfloor} \leq\sum
_{k\geq1} \alpha^{\sqrt{k}/t-1} \leq\int_{0}^\infty
\alpha^{\sqrt{x}/t-1} \,dx = \frac{t^2}{\alpha(\ln\alpha)^2} \int
_{0}^\infty
e^{-\sqrt{y}} \,dy,
\end{eqnarray*}
therefore,
asymptotically,
for any $j\in\{ 1,\ldots,m \}$,
\[
E \bigl( 
\bigl( T_{1}^*\wedge\tau_0
\bigr)^2 | Z_{0}=j \bigr) \leq m^{3c\ln m}.
\]
Reporting in the inequality before Lemma~\ref{ulbed}, we have
\[
\forall k\geq1\qquad E \bigl( 1_{k\leq K(\tau_0)} \bigl(T_{k}^*\wedge
\tau_0-{T_{k-1}} \bigr)^2 \bigr) \leq
m^{3c\ln m}.
\]
Plugging this estimate in (\ref{bigcirc})
and using Lemma~\ref{Kex}, we obtain
\begin{eqnarray*}
&& E \Biggl( 1_{\tau_0\leq t_m^\eta} \sum_{k= 1}^{K(\tau_0)}
\bigl(T_{k}^*\wedge\tau_0-{T_{k-1}} \bigr)
\Biggr)
\\
&&\qquad
\leq\sum_{k\geq0} \bigl(m^{3c\ln m}
\bigr)^{{1}/{2}} P \bigl({K\bigl(t_m^\eta\bigr)> k}
\bigr)^{1/2}
\\
&&\qquad \leq m^{3c\ln m} \biggl( t_m^\eta
\exp(-cm/3) +1+\sum_{
k\geq
t_m^\eta\exp(-cm/3)
} \biggl( \frac{(t_m^\eta)^{k}}{k!}
\exp(-cmk) \biggr)^{1/2} \biggr)
\\
&&\qquad
\leq m^{3c\ln m} \biggl(
t_m^\eta\exp(-cm/3) +1+ \sum_{
k\geq0
}
\exp\biggl( \frac{k}{2} -cm\frac{k}{3} \biggr) \biggr). %
\end{eqnarray*}
To get the last inequality, we have used that
$k!\geq({k}/{e})^k$,
whence, for
$k\geq
t_m^\eta\exp(-cm/3)$,
\[
\frac{(t_m^\eta)^{k}}{k!} \leq\biggl(\frac{et_m^\eta}{k} \biggr)^{k}
\leq
\exp(k+cmk/3).
\]
We choose $\eta$ such that $\eta<c/3$. The above inequality implies
that
\begin{eqnarray*}
&& \mathop{\limsup_{\ell,m\to\infty, q\to0
}}_{{\ell q} \to a } \frac{1}{m}\ln E
\Biggl( 1_{\tau_0\leq t_m^\eta} \sum_{k= 1}^{K(\tau_0)}
\bigl(T_{k}^*\wedge\tau_0-{T_{k-1}} \bigr)
\Biggr)
\\
&&\qquad \leq\max\biggl( V\bigl(\rho^*,0\bigr)+\eta-\frac{c}{3},0 \biggr) < V
\bigl(\rho^*,0\bigr).
\end{eqnarray*}
All these estimates, together with
Proposition~\ref{exta}, imply that, asymptotically,
uniformly with respect to $i\in\{ 1,\ldots,m \}$,
\[
\Biggl| E \Biggl(\sum_{n=0}^{\tau_0} f \biggl(
\frac{Z_n}{m} \biggr) \Big| Z_0=i \Biggr)-f\bigl(\rho^*\bigr) E (
\tau_0 | Z_0=i ) \Biggr| \leq3\varepsilon E (
\tau_0 | Z_0=i ). %
\]
This achieves the proof of Proposition~\ref{exfa}.
\end{pf}

\section{The neutral phase}\label{disc}
We denote by ${\mathcal N}$
the set
of the populations which do not contain the master sequence $w^*$, that is,
\[
{\mathcal N}= \bigl({\mathcal A}^\ell\setminus\bigl\{ w^* \bigr\}
\bigr)^m.
\]
Since we deal with the sharp peak landscape, the transition mechanism
of the process restricted to the set ${\mathcal N}$ is neutral.
We consider a Wright--Fisher process
$(X_n)_{n\geq0}$
starting from a population of ${\mathcal N}$.
We wish to evaluate the first time when a master sequence appears in the
population,
\[
\tau_* = \inf\{ n\geq0\dvtx  X_n\notin{\mathcal N} \}.
\]
We call the time $\tau_*$ the discovery time.
Until the time $\tau_*$, the process evolves in~${\mathcal N}$, and the dynamics of the Wright--Fisher model in
${\mathcal N}$ does
not depend on $\sigma$.
In particular, the law
of the discovery time $\tau_*$ is the same for the Wright--Fisher
model with $\sigma>1$
and the neutral Wright--Fisher model with $\sigma=1$. Therefore,
we compute the estimates for the latter model.

\medskip\textit{Neutral hypothesis.} Throughout this section, we suppose that
$\sigma=1$.

\subsection{Mutation dynamics}\label{mutdyn}
We consider a Markov chain
$(Y_n)_{n\geq0}
\index{$Y_n$}$
with state space $\{ 0,\ldots,\ell\} $ and having for
transition matrix the lumped mutation matrix $M_H$.
In this section, we recall some properties and estimates concerning
the Markov chain
$(Y_n)_{n\geq0}$.
We refer
to the corresponding section of~\cite{CE} for the detailed proofs.
The Markov chain
$(Y_n)_{n\geq0}$ is monotone.
We
denote by ${\mathcal B}\index{${\mathcal B}$}$
the binomial law
${\mathcal B}(\ell, 1-1/\kappa)$
with parameters
$\ell$ and $1-1/\kappa$, that is,
%
\[
\forall b\in\{ 0,\ldots,\ell\} \qquad{\mathcal B}(b) = \dbinom{\ell} {b}
\biggl(1-\frac{1}{\kappa} \biggr)^b \biggl(\frac{1}{\kappa}
\biggr)^{\ell-b}.
\]
The matrix $M_H$ is reversible with respect to
the binomial law
${\mathcal B}$.
This binomial law is the invariant probability measure of the
Markov chain
$(Y_n)_{n\geq0}$.
When $\ell$ grows, the law ${\mathcal B}$ concentrates exponentially fast
in a neighborhood of its mean
$\ell_\kappa= \ell
(1-1/\kappa)$.
We restate next without proofs several inequalities and
estimates proved in~\cite{CE}.
%

\begin{lemma}\label{exco}
For $b\leq\ell/2$, we have
\[
\frac{1}{\kappa^\ell} \biggl(\frac{\ell}{2b} \biggr)^b \leq{
\mathcal B}(b) \leq\frac{\ell^b}{\kappa^{\ell-b}}.
\]
\end{lemma}

\begin{proposition}\label{firstdesc}
We suppose that
$\ell\to+\infty, q\to0,
{\ell q} \to a\in\,]0,+\infty[$.
Asymptotically,
we have
\[
\forall n\geq\sqrt{\ell}\qquad P (Y_{n}\geq\ln\ell|
Y_0=0 ) \geq1-\exp\bigl(-\tfrac{1}{2}(\ln
\ell)^2 \bigr).
\]
\end{proposition}

%
\begin{proposition}\label{secdesc}
We suppose that
$\ell\to+\infty, q\to0,
{\ell q} \to a\in\,]0,+\infty[$.
Let $\varepsilon\in\,]0,1[$.
There exists $c(\varepsilon)>0$ such that,
asymptotically,
we have
\[
\forall n\geq\frac{4\ell}{a\varepsilon}\qquad P \bigl(Y_{n}\geq
\ell_\kappa(1-\varepsilon) | Y_0=0 \bigr) \geq1- \exp
\bigl(-c(\varepsilon)\ell\bigr).
\]
\end{proposition}

We define
\[
\tau_0 = \inf\{ n\geq0\dvtx  Y_{n}=0 \}.
\]
%

\begin{proposition}\label{visitz}
For any $a\in\,]0,+\infty[$,
\[
\mathop{\limsup_{\ell\to\infty, q\to0}}_
{{\ell q} \to a} \frac{1}{\ell}\ln
E(\tau_0 | Y_0=\ell) \leq\ln\kappa. %
\]
\end{proposition}

\subsection{Ancestral lines}
Let us define an ancestral line.
For $i\in\{ 1,\ldots,m \}$ and $n\geq1$, we denote by
${\mathcal I}(i,n,n-1)
\index{${\mathcal I}(i,n,n-1)$}$
the index of the ancestor at time $n-1$ of the $i$th chromosome
at time~$n$.
More precisely, if the
$i$th chromosome of the population at time $n$ has been obtained
by replicating the $j$th chromosome of the population at time $n-1$,
then
${\mathcal I}(i,n,n-1) =j$.
For $s\leq n$,
the index
${\mathcal I}(i,n,s)$ of the ancestor at time $s$ of the $i$th chromosome
at time~$n$ is then defined recursively with the help of the following
formula:
\[
{\mathcal I}(i,n,s) = {\mathcal I}\bigl( {\mathcal
I}(i,n,n-1),n-1,s\bigr).
\]
We define also
${\mathcal I}(i,n,n)=i$.
The ancestor at time $s$ of the $i$th chromosome
at time~$n$ is the chromosome
\[
\operatorname{ancestor}(i,n,s) = X_s\bigl({\mathcal I}(i,n,s)\bigr).
\index{$\operatorname{ancestor}(i,n,s)$} %
\]
The ancestral line
of the $i$th chromosome
at time~$n$ is the sequence of its ancestors until time $0$,
\[
\bigl(\operatorname{ancestor}(i,n,s), 0\leq s\leq n\bigr) =
\bigl(X_s\bigl({\mathcal I}(i,n,s)\bigr), 0\leq s\leq n\bigr).
\]
%

\begin{proposition}\label{neal}
Let $b\in\{ 0,\ldots,m \}$, and
let
$(X_n)_{n\geq0}$ be the neutral Wright--Fisher process
starting from $(b)^m$.
Let $i\in\{ 1,\ldots,m \}$. For any $n\geq0$, the law of the
ancestral line
$(\operatorname{ancestor}(i,n,s), 0\leq s\leq n)$
of the $i$th chromosome of $X_n$
is equal to the law of
$(Y_{0},\ldots,Y_{n})$ starting from $b$.
\end{proposition}

The proof is standard. One can proceed by induction as in \cite{CE}.
In fact,
the ancestral lines
of the individuals at time $n$ are given by a coalescent process.
Along an ancestral line, a chromosome moves according to the mutation
dynamics given by the matrix $M_H$.
\subsection{Discovery time}
The dynamics of the processes
$(O_n^\ell)_{n\geq0}$,
$(O_n^1)_{n\geq0}$
in ${\mathcal N}$ are the same as the original process
$(O_n)_{n\geq0}$.
Therefore we can use the original process to compute the
discovery time
%
\[
\tau^{*} = \inf\bigl\{ n\geq0\dvtx  O_n\in{\mathcal W}^*
\bigr\}. \index{$\tau^*$} %
\]
%
The law of the discovery time $\tau^*$ is the same for the
distance process and the occupancy process.
With a slight abuse of notation, we let
\[
\tau^{*} = \inf\bigl\{ n\geq0\dvtx  D_n\in{\mathcal W}^*
\bigr\}.
\]
We will carry out the estimates of $\tau^*$ for the distance process
$(D_n)_{n\geq0}$.
Notice that the case $\alpha=+\infty$ is not covered by the result
of next proposition. This case will be handled separately, with the
help of the intermediate inequality of Corollary~\ref{majext}.

\medskip\textit{Notation.}
For $b\in\{ 0,\ldots,\ell\} $, we denote by
$(b)^m\index{$(b)^m$}$
the vector column whose components are all equal to $b$.
%

\begin{proposition}\label{bdd}
Let $a\in\,]0,+\infty[$ and
$\alpha\in[0,+\infty[$.
For any $d\in{\mathcal N}$,
%
\[
\mathop{\lim_{\ell,m\to\infty,
q\to0}}_{{\ell q} \to a,
({m}/{\ell})\to\alpha
} \frac{1}{\ell}\ln E
\bigl(\tau^* | D_0= d \bigr) = \ln\kappa.
\]
\end{proposition}

\begin{pf}
By Corollary~\ref{corneu},
the neutral distance process
$(D_n)_{n\geq0}$ is monotone. Therefore,
for any $d\in{\mathcal N}$, we have
\begin{eqnarray*}
&&E \bigl(\tau^* | D_0= (1)^m \bigr) \leq E \bigl(\tau^*
| D_0= d \bigr) \leq E \bigl(\tau^* | D_0= (
\ell)^m \bigr).
\end{eqnarray*}
To bound the discovery time $\tau^*$ from above,
we consider the time needed for a single chromosome to discover
the master sequence $w^*$,
and we remark that, if the master sequence has not been discovered
until time~$n$
in the distance process,
then certainly the ancestral line of any chromosome present at time~$n$
does not contain the master sequence.
By
Proposition~\ref{neal}, the ancestral line of any chromosome
present at time~$n$
has the same law as
$Y_{0},\ldots, Y_{n} $.
Therefore, we conclude that
\[
\forall n\geq0\qquad P \bigl(\tau^*>n | D_0= (\ell)^m
\bigr) \leq P(\tau_0>n | Y_0=\ell),
\]
where $\tau_0$ is the hitting time of $0$ for the process
$(Y_n)_{n\geq0}$.
Summing this inequality over $n\geq0$,
we obtain the following upper bound.
%

\begin{corollary}\label{majext}
For any $d\in{\mathcal N}$, any $m\geq1$, we have
\[
E \bigl(\tau^* | D_0= d \bigr) \leq E(\tau_0 |
Y_0=\ell). %
\]
\end{corollary}

With the help of
Proposition~\ref{visitz}, we obtain the desired upper bound.
To bound the discovery time $\tau^*$ from below,
we use the same strategy as in \cite{CE}.
There are two main differences in the case of the Wright--Fisher model.
First the time scale is multiplied by $m$, because $m$ mutations
can occur at each generation.
Second, the neutral distance process
$(D_n)_{n\geq0}$ has positive correlations. This makes the proof
substantially simpler than in the case of the Moran model, where
a technical exponential estimate had to be used instead of a
correlation inequality.
We give here only the main steps of the proof.
The details are similar to \cite{CE} in that they involve repeated
intermediate conditionings, use of the Markov property and
monotonicity.

We suppose that
the distance process starts from
$(1)^m$,
and we will estimate the probability
of a specific scenario leading to a discovery time close
to~${\kappa^{\ell}}$.
%
Let
${\mathcal E}$ be the event
\[
{\mathcal E} = \bigl\{ \forall n\leq\ell^{3/4},\ \forall i\in\{ 1,
\ldots,m \},
U_n^{i,1}>q/(\kappa-1) \bigr\}.
\]
If the event ${\mathcal E}$ occurs, then, until time
$\ell^{3/4}$, none of the mutation events in the process
$(D_n)_{n\geq0}$ can create a master sequence.
Let $\varepsilon>0$.
Conditioning on the population at time $m\ell^{3/4}$, we obtain
\begin{eqnarray*}
&& P \bigl(\tau^*> {\kappa^{\ell(1-\varepsilon)}} | D_{0}= (1)^m
\bigr)
\\
&&\qquad \geq P \bigl(\tau^*> {\kappa^{\ell(1-\varepsilon)}},{\mathcal E} |
D_{0}= (1)^m \bigr)
\\
&&\qquad \geq P \bigl( \tau^*> {\kappa^{\ell(1-\varepsilon)}} | D_{0}= (\ln
\ell)^m \bigr) P \bigl( D_{\ell^{3/4}}\geq(\ln\ell)^m, {\mathcal E}
 | D_{0}= (1)^m \bigr).
\end{eqnarray*}
We first study the last term in the above inequality.
The status of the process at time
$\ell^{3/4}$ is a function of the random matrices
\[
R_n = \bigl(S_n^i,U_n^{i,1},
\ldots,U_n^{i,\ell} \bigr)_{1\leq i\leq m},\qquad1\leq n\leq
\ell^{3/4}.
\]
We make an intermediate conditioning with respect to the variables
$S_n^i$,
%
\begin{eqnarray*}
&& P \bigl( D_{\ell^{3/4}}\geq(\ln\ell)^m, {\mathcal E} |
D_{0}= (1)^m \bigr)
\\
&&\qquad =
E \bigl( P \bigl(
D_{\ell^{3/4}}\geq(\ln\ell)^m, {\mathcal
E} | S_n^i, 1\leq i\leq m, 1\leq n\leq
\ell^{3/4} \bigr) | D_{0}= (1)^m \bigr).
\end{eqnarray*}
The variables
$S_n^i, 1\leq i\leq m,
1\leq n\leq\ell^{3/4}$
being fixed, all the indices of the chromosomes
selected for replication are fixed,
and since
the mutation map ${\mathcal M}_H(\cdot,u_1,\ldots,u_\ell)$ is nondecreasing
with respect to $u_1,\ldots,u_\ell$,
the state of the process
at time
$\ell^{3/4}$ is a nondecreasing function of the variables
\[
U_n^{i,1},\ldots,U_n^{i,\ell}, \qquad1
\leq i\leq m, 1\leq n\leq\ell^{3/4}.
\]
Thus the events ${\mathcal E}$ and
$D_{\ell^{3/4}}\geq
(\ln\ell)^m$
are both nondecreasing with respect to these variables.
By the FKG inequality for a product measure,
%
\begin{eqnarray*}
&& P \bigl( 
D_{\ell^{3/4}}\geq(\ln\ell)^m,
{\mathcal E} | S_n^i, 1\leq i\leq m, 1\leq n\leq
\ell^{3/4} \bigr)
\\
&&\qquad \geq
P \bigl( D_{\ell^{3/4}}\geq(\ln
\ell)^m 
| S_n^i, 1\leq i\leq m,
1\leq n\leq\ell^{3/4} \bigr) P({\mathcal E}). 
\end{eqnarray*}
We have used the fact that
${\mathcal E}$ does not depend on the variables
$S_n^i$.
Reporting in the conditioning, we obtain
\[
P \bigl( D_{\ell^{3/4}}\geq(\ln\ell)^m, {\mathcal E} \bigr)
\geq
P
\bigl( D_{\ell^{3/4}}\geq(\ln\ell)^m \bigr) P ( {\mathcal E} ).
\]
%
By Proposition~\ref{lawD},
the distance process starting from~$(1)^m$
has positive correlations,
therefore
%
\begin{eqnarray*}
&&P \bigl( D_{\ell^{3/4}}\geq(\ln\ell)^m \bigr)
\geq\prod_{1\leq i\leq m} P \bigl( 
D_{\ell^{3/4}}(i)
\geq\ln\ell\bigr) = P ( Y_{\ell^{3/4}}\geq\ln\ell)^m.
\end{eqnarray*}
Using the estimate of
Proposition~\ref{firstdesc}, we get
\[
P \bigl( D_{\ell^{3/4}}\geq(\ln\ell)^m, {\mathcal E} \bigr)
\geq\biggl( 1-\exp\biggl(-\frac{1}{2}(\ln\ell)^2 \biggr)
\biggr)^m \biggl(1-\frac{q}{\kappa-1} \biggr)^{
m\ell^{3/4}}.
\]
We study next
\[
P \bigl( \tau^*> {\kappa^{\ell(1-\varepsilon)}} 
| D_{0}= (\ln
\ell)^m \bigr).
\]
%
The following inequality can be proved exactly as Lemma~10.15 of
\cite{CE}.
%

\begin{lemma}\label{premhit}
For $b\in\{ 1,\ldots,\ell\} $, we have
\[
\forall n\geq0\qquad P \bigl( \tau^*\leq n | D_{0}=(b)^m
\bigr) 
\leq nm \frac{{\mathcal B}(0)}{{\mathcal B}(b)}.
\]
\end{lemma}

\renewcommand{\theequation}{$\heartsuit$}
Let $\varepsilon'>0$.
Conditioning on the population at time $\ell^2$, we obtain
%
\begin{eqnarray}\label{heartsuit}
&& P \bigl(\tau^*> {\kappa^{\ell(1-\varepsilon)}} | D_{0}=(\ln
\ell)^m \bigr)\nonumber
\\
&&\qquad  \geq P \bigl( \tau^*> {\kappa^{\ell(1-\varepsilon)}} |
D_{0}= \bigl(\ell_\kappa\bigl(1-\varepsilon'
\bigr)\bigr)^m \bigr)
\\
&&\quad\qquad{} \times P \bigl( \tau^*>\ell^2, D_{\ell^2}\geq\bigl(
\ell_\kappa\bigl(1-\varepsilon'\bigr)\bigr)^m
| D_{0}= (\ln\ell)^m \bigr). \nonumber
\end{eqnarray}
We first take care of the last probability.
We write
\renewcommand{\theequation}{$\natural$}
\begin{eqnarray}\label{natural}
&& P \bigl( \tau^*>\ell^2, D_{\ell^2}\geq\bigl(
\ell_\kappa\bigl(1-\varepsilon'\bigr)\bigr)^m
| D_{0}= (\ln\ell)^m \bigr)
\nonumber\\[-8pt]\\[-8pt]
&&\qquad \geq
P \bigl( D_{\ell^2}\geq\bigl(\ell_\kappa\bigl(1-
\varepsilon'\bigr)\bigr)^m | D_{0}= (\ln
\ell)^m \bigr)- P \bigl( \tau^*\leq\ell^2 |
D_{0}= (\ln\ell)^m \bigr).\nonumber
\end{eqnarray}
To control the last term,
we use the inequality of Lemma~\ref{premhit} with $n=\ell^2$
and $b=\ln\ell$, and
Lemma~\ref{exco},
%
\renewcommand{\theequation}{$\flat$}
\begin{equation}\label{flat}
P \bigl( \tau^*\leq\ell^2 | D_{0}=(\ln l)^m
\bigr) \leq\ell^2 m \frac{{\mathcal B}(0)}{{\mathcal B}(\ln\ell)} \leq
\ell^2 m
\biggl(\frac{2\ln\ell}{\ell} \biggr)^{\ln\ell}.
\end{equation}
For the other term, we use
the monotonicity of the process $(D_n)_{n\geq0}$,
the fact that it has positive correlations (by Proposition~\ref{lawD}),
and
Proposition~\ref{secdesc}
to get
\renewcommand{\theequation}{$\sharp$}
\begin{eqnarray}\label{sharp}
&& P \bigl( D_{\ell^2}\geq\bigl(\ell_\kappa\bigl(1-
\varepsilon'\bigr)\bigr)^m 
|D_{0}= (\ln\ell)^m \bigr)\nonumber
\\
&&\qquad \geq\prod_{1\leq i\leq m} P \bigl( D_{\ell^2}(i)\geq
\ell_\kappa\bigl(1-\varepsilon'\bigr) | D_{0}=
(0)^m \bigr)
\\
&&\qquad
= P \bigl( Y_{\ell^{2}}\geq\ell_\kappa
\bigl(1-\varepsilon'\bigr) | Y_0=0 \bigr)^m
\geq\bigl(1- \exp\bigl(-c\bigl(\varepsilon'\bigr)\ell\bigr)
\bigr)^m.\nonumber
\end{eqnarray}
Plugging the inequalities~(\ref{flat}) and (\ref{sharp})
into the inequality~(\ref{natural}),
we obtain
\renewcommand{\theequation}{$\clubsuit$}
\begin{eqnarray}\label{clubsuit}
&&P \bigl( \tau^*>\ell^2, D_{\ell^2}\geq\bigl(
\ell_\kappa\bigl(1-\varepsilon'\bigr)\bigr)^m
| D_{0}= (\ln\ell)^m \bigr)
\nonumber\\[-8pt]\\[-8pt]
&&\qquad \geq
\bigl( 1- \exp\bigl(-c\bigl(\varepsilon'\bigr)\ell\bigr)
\bigr)^m - \ell^2 m \biggl(\frac{2\ln\ell}{\ell}
\biggr)^{\ln\ell}.\nonumber
\end{eqnarray}
Using Lemma~\ref{premhit} with
$n={\kappa^{\ell(1-\varepsilon)}}$
and
$b=\ell_\kappa(1-\varepsilon')$, and a standard large deviation estimates,
we see that,
for $\varepsilon'$ small enough,
there exists $c(\varepsilon)>0$ such that,
for $\ell$ large enough,
\renewcommand{\theequation}{$\spadesuit$}
\begin{equation}\label{spadesuit}
P \bigl( \tau^*\leq{\kappa^{\ell(1-\varepsilon)}} | D_{0}= \bigl(
\ell_\kappa\bigl(1-\varepsilon'\bigr)\bigr)^m
\bigr) \leq\frac{
{\kappa^{\ell(1-\varepsilon)}}m
{\mathcal B}(0)}{{\mathcal B}(
\ell_\kappa(1-\varepsilon')
)} \leq m e^{-c(\varepsilon)\ell}.
\end{equation}
%
Plugging the estimates~(\ref{clubsuit}) and~(\ref{spadesuit})
into the inequality~(\ref{heartsuit}),
we conclude that, for~$\ell$ large enough,
\begin{eqnarray*}
&& P \bigl(\tau^*> {\kappa^{\ell(1-\varepsilon)}} | D_0=(1)^m
\bigr)
\\
&&\qquad \geq\biggl( 1-\exp\biggl(-\frac{1}{2}(\ln\ell)^2
\biggr) \biggr)^m \biggl(1-\frac{q}{\kappa-1} \biggr)^{
m\ell^{3/4}}
\\
&&\quad\qquad {}\times
\bigl(1-m\exp\bigl(-c(\varepsilon)\ell\bigr) \bigr) \biggl( \bigl( 1-\exp\bigl(-c\bigl(\varepsilon'\bigr)\ell\bigr) \bigr)^m
- \ell^2 m \biggl(\frac{2\ln\ell}{\ell} \biggr)^{\ln\ell} \biggr).
\end{eqnarray*}
Moreover, by Markov's inequality,
\[
E \bigl(\tau^* | D_0=(1)^m \bigr) \geq{
\kappa^{\ell(1-\varepsilon)}} P \bigl(\tau^*\geq{\kappa^{\ell
(1-\varepsilon)}} |
D_0=(1)^m \bigr).
\]
It follows that
\[
\mathop{\liminf_{\ell,m\to\infty,
q\to0
}}_{{\ell q} \to a,
({m}/{\ell})\to\alpha
} \frac{1}{\ell}\ln
E \bigl(\tau^* | D_0=(1)^m \bigr) \geq(1-\varepsilon)\ln
\kappa.
\]
Letting $\varepsilon$ go to $0$ yields the desired lower bound.
\end{pf}

%
\section{Synthesis}\label{secsyn}
As in Theorem~\ref{mainth},
we suppose that
$\ell\to+\infty$, $m\to+\infty$, $q\to0$,
in such a way that
${\ell q} \to a\in\,]0,+\infty[$,
${m}/{\ell}\to\alpha\in[0,+\infty]$.
We put now together the estimates of
Sections~\ref{bide} and~\ref{disc} in order to
evaluate the formula for the invariant measure
obtained at the end of
Section~\ref{bounds}.
Using the monotonicity of
$(Z_n^\theta)_{n\geq0}$, we have
\begin{eqnarray*}
E \bigl({\tau_0} | Z^\theta_0= 1 \bigr) &\leq&
\sum_{i=1}^m E \bigl({\tau_0}
| Z^\theta_0= i \bigr) P \bigl( O^\theta_{\tau^*}(0)=i
| O^\theta_0=o^\theta_{\mathrm{exit}} \bigr)
\\
&\leq&
E \bigl({\tau_0} | Z^\theta_0= m \bigr).
\end{eqnarray*}
These inequalities and
Proposition~\ref{exta} imply that
\[
\mathop{\lim_{\ell,m\to\infty}}_{q\to0,
{\ell q} \to a} \frac{1}{m}\ln\sum
_{i=1}^m E \bigl({\tau_0} |
Z^\theta_0= i \bigr) P \bigl( O^\theta_{\tau^*}(0)=i
| O^\theta_0=o^\theta_{\mathrm{exit}} \bigr) = V
\bigl(\rho^*(a),0\bigr).
\]
By Proposition~\ref{bdd}, for $\alpha\in[0,+\infty[$,
\[
\mathop{\lim_{\ell,m\to\infty,
q\to0
}}_{{\ell q} \to a,
({m}/{\ell})\to\alpha
} \frac{1}{\ell}\ln E
\bigl(\tau^* | O^\theta_0=o^\theta_{\mathrm{exit}}
\bigr) = \ln\kappa.
\]
For the case $\alpha=+\infty$,
by Corollary~\ref{majext} and
Proposition~\ref{visitz},
\[
\mathop{\limsup_{\ell,m\to\infty,
q\to0
}}_{{\ell q} \to a,
({m}/{\ell})\to\infty
} \frac{1}{\ell}\ln
E \bigl(\tau^* | O^\theta_0=o^\theta_{\mathrm{exit}}
\bigr) 
\leq\ln\kappa.
\]
These estimates allow us to evaluate the
ratio between the discovery time and the persistence time.
We define a function
$\psi\dvtx ]0,+\infty[\,\to
[0,+\infty[$ by setting
%
\[
\forall a \in\,]0,+\infty[\qquad\psi(a) = V\bigl(\rho^*(a),0\bigr). %
\]

For $\alpha\in[0,+\infty[$ or $\alpha=+\infty$,
we have
%
%
\begin{eqnarray*}
&& \mathop{\lim_{\ell,m\to\infty,
q\to0
}}_{{\ell q} \to a,
({m}/{\ell})\to\alpha
} \frac{{\sum_{i=1}^m
E ({\tau_0}
|
Z^\theta_0= i )
P (
O^\theta_{\tau^*}(0)=i
|
O^\theta_0=o^\theta_{\mathrm{exit}}
)
}
}{
E (\tau^* |
O^\theta_0=o^\theta_{\mathrm{exit}}
)
}
\\
&&\qquad = \cases{ 0, &\quad if $\alpha\psi(a)<\ln\kappa$,
\cr
+\infty, &\quad if $\alpha
\psi(a)>\ln\kappa$.}
\end{eqnarray*}
By Proposition~\ref{exfa}, we have
\[
\mathop{\lim_{\ell,m\to\infty}}_{q\to0,
{\ell q} \to a} \frac{\sum_{i=1}^m
E (\sum_{n=0}^{\tau_0}
f ({Z^\theta_n}/{m} ) |
Z^\theta_0= i )
P (
O^\theta_{\tau^*}(0)=i
|
O^\theta_0=o^\theta_{\mathrm{exit}}
)
}{
\sum_{i=1}^m
E ({\tau_0}
|
Z^\theta_0= i )
P (
O^\theta_{\tau^*}(0)=i
|
O^\theta_0=o^\theta_{\mathrm{exit}}
)
} = f
\bigl(\rho^*\bigr). %
\]
Putting together the bounds on $\nu$ given in Section~\ref{bounds} and
the previous considerations, we conclude that
\[
\mathop{\lim_{\ell,m\to\infty,
q\to0
}}_{{\ell q} \to a,
({m}/{\ell})\to\alpha
} \int_{[0,1]}
f\, d\nu= \cases{ 0, &\quad if $\alpha\psi(a)<\ln\kappa$,
\cr
f \bigl(\rho^*(a)
\bigr), &\quad if $\alpha\psi(a)>\ln\kappa$.} %
\]
This is valid for any
continuous nondecreasing function
$f\dvtx [0,1]\to{\mathbb R}$
such that $f(0)=0$.

\section*{Acknowledgment}
I thank an anonymous referee for his careful reading and his remarks,
which helped to improve the presentation.




%

\printaddresses

\begin{thebibliography}{25}

\bibitem{AF}
%
\begin{barticle}[author]
\bauthor{\bsnm{Alves},~\bfnm{Domingos}\binits{D.}} \AND
\bauthor{\bsnm{Fontanari},~\bfnm{Jose~Fernando}\binits{J.~F.}}
(\byear{1998}).
\btitle{Error threshold in finite populations}.
\bjournal{Phys. Rev. E}
\bvolume{57}
\bpages{7008--7013}.
\end{barticle}
%
\bptok{imsref}%
\endbibitem

\bibitem{CGA}
%
\begin{bmisc}[author]
\bauthor{\bsnm{Cerf},~\bfnm{Rapha{\"e}l}\binits{R.}}
(\byear{2010}).
\bhowpublished{Critical control of a genetic algorithm.
Preprint. Available at \arxivurl{arXiv:1005.3390}.}
\end{bmisc}
%
\bptok{imsref}%
\endbibitem

\bibitem{CE}
\begin{bbook}[author]
\bauthor{\bsnm{Cerf},~\bfnm{Rapha{\"e}l}\binits{R.}}
(\byear{2015}).
\btitle{Critical Population and Error Threshold on the Sharp Peak Landscape for a {M}oran Model}.
\bseries{Mem. Amer. Math. Soc.}
\bvolume{233}.
\bpublisher{Amer. Math. Soc.},
\blocation{Providence, RI}.
\end{bbook}
\bptok{imsref}%
\endbibitem

\bibitem{TD}
%
\begin{bmisc}[author]
\bauthor{\bsnm{Darden},~\bfnm{Thomas}\binits{T.}}
(\byear{1983}).
\bhowpublished{Asymptotics of fixation under heterosis: A large
deviation approach.
Unpublished manuscript}.
\end{bmisc}
%
\bptok{imsref}%
\endbibitem

\bibitem{Deme}
%
\begin{barticle}[mr]
\bauthor{\bsnm{Demetrius},~\bfnm{Lloyd}\binits{L.}},
\bauthor{\bsnm{Schuster},~\bfnm{Peter}\binits{P.}} \AND
\bauthor{\bsnm{Sigmund},~\bfnm{Karl}\binits{K.}}
(\byear{1985}).
\btitle{Polynucleotide evolution and branching processes}.
\bjournal{Bull. Math. Biol.}
\bvolume{47}
\bpages{239--262}.
\bid{doi={10.1016/S0092-8240(85)90051-5}, issn={0092-8240}, mr={0803564}}
\end{barticle}
%
\bptok{imsref}%
\endbibitem

\bibitem{DSV}
%
\begin{barticle}[mr]
\bauthor{\bsnm{Dixit},~\bfnm{Narendra~M.}\binits{N.~M.}},
\bauthor{\bsnm{Srivastava},~\bfnm{Piyush}\binits{P.}} \AND
\bauthor{\bsnm{Vishnoi},~\bfnm{Nisheeth~K.}\binits{N.~K.}}
(\byear{2012}).
\btitle{A finite population model of molecular evolution: Theory and
computation}.
\bjournal{J. Comput. Biol.}
\bvolume{19}
\bpages{1176--1202}.
\bid{doi={10.1089/cmb.2012.0064}, issn={1066-5277}, mr={2990752}}
\end{barticle}
%
\bptok{imsref}%
\endbibitem

\bibitem{EI1}
%
\begin{barticle}[author]
\bauthor{\bsnm{Eigen},~\bfnm{Manfred}\binits{M.}}
(\byear{1971}).
\btitle{Self-organization of matter and the evolution of biological
macromolecules}.
\bjournal{Naturwissenschaften}
\bvolume{58}
\bpages{465--523}.
\end{barticle}
%
\bptok{imsref}%
\endbibitem

\bibitem{EL}
%
\begin{bbook}[mr]
\bauthor{\bsnm{Ellis},~\bfnm{Richard~S.}\binits{R.~S.}}
(\byear{2006}).
\btitle{Entropy, Large Deviations, and Statistical Mechanics}.
\bpublisher{Springer},
\blocation{Berlin}.
\bnote{Reprint of the 1985 original}.
\bid{mr={2189669}}
\end{bbook}
%
\bptok{imsref}%
\endbibitem

\bibitem{FE}
%
\begin{bbook}[mr]
\bauthor{\bsnm{Feller},~\bfnm{William}\binits{W.}}
(\byear{1968}).
\btitle{An Introduction to Probability Theory and Its Applications.
{V}ol. {I}},
\bedition{3rd} ed.
\bpublisher{Wiley},
\blocation{New York}.
\bid{mr={0228020}}
\end{bbook}
%
\bptok{imsref}%
\endbibitem

\bibitem{FW}
%
\begin{bbook}[mr]
\bauthor{\bsnm{Freidlin},~\bfnm{M.~I.}\binits{M.~I.}} \AND
\bauthor{\bsnm{Wentzell},~\bfnm{A.~D.}\binits{A.~D.}}
(\byear{1998}).
\btitle{Random Perturbations of Dynamical Systems},
\bedition{2nd} ed.
\bseries{Grundlehren der Mathematischen Wissenschaften [Fundamental
Principles of Mathematical Sciences]}
\bvolume{260}.
\bpublisher{Springer},
\blocation{New York}.
\bnote{Translated from the 1979 Russian original by Joseph Sz{\"u}cs}.
\bid{doi={10.1007/978-1-4612-0611-8}, mr={1652127}}
\end{bbook}
%
\bptok{imsref}%
\endbibitem

\bibitem{GI}
%
\begin{barticle}[mr]
\bauthor{\bsnm{Gillespie},~\bfnm{Daniel~T.}\binits{D.~T.}}
(\byear{1976}).
\btitle{A general method for numerically simulating the stochastic
time evolution of coupled chemical reactions}.
\bjournal{J. Comput. Phys.}
\bvolume{22}
\bpages{403--434}.
\bid{issn={0021-9991}, mr={0503370}}
\end{barticle}
%
\bptok{imsref}%
\endbibitem

\bibitem{GRI}
%
\begin{bbook}[mr]
\bauthor{\bsnm{Grimmett},~\bfnm{Geoffrey}\binits{G.}}
(\byear{1999}).
\btitle{Percolation},
\bedition{2nd} ed.
\bseries{Grundlehren der Mathematischen Wissenschaften [Fundamental
Principles of Mathematical Sciences]}
\bvolume{321}.
\bpublisher{Springer},
\blocation{Berlin}.
\bid{doi={10.1007/978-3-662-03981-6}, mr={1707339}}
\end{bbook}
%
\bptok{imsref}%
\endbibitem

\bibitem{Har}
%
\begin{barticle}[mr]
\bauthor{\bsnm{Harris},~\bfnm{T.~E.}\binits{T.~E.}}
(\byear{1977}).
\btitle{A correlation inequality for {M}arkov processes in partially
ordered state spaces}.
\bjournal{Ann. Probab.}
\bvolume{5}
\bpages{451--454}.
\bid{mr={0433650}}
\end{barticle}
%
\bptok{imsref}%
\endbibitem

\bibitem{KI}
%
\begin{bbook}[mr]
\bauthor{\bsnm{Kifer},~\bfnm{Yuri}\binits{Y.}}
(\byear{1988}).
\btitle{Random Perturbations of Dynamical Systems}.
\bseries{Progress in Probability and Statistics}
\bvolume{16}.
\bpublisher{Birkh\"auser},
\blocation{Boston, MA}.
\bid{doi={10.1007/978-1-4615-8181-9}, mr={1015933}}
\end{bbook}
%
\bptok{imsref}%
\endbibitem

\bibitem{KID}
%
\begin{barticle}[mr]
\bauthor{\bsnm{Kifer},~\bfnm{Yuri}\binits{Y.}}
(\byear{1990}).
\btitle{A discrete-time version of the {W}entzell--{F}reidlin theory}.
\bjournal{Ann. Probab.}
\bvolume{18}
\bpages{1676--1692}.
\bid{issn={0091-1798}, mr={1071818}}
\end{barticle}
%
\bptok{imsref}%
\endbibitem

\bibitem{LIG}
%
\begin{bbook}[mr]
\bauthor{\bsnm{Liggett},~\bfnm{Thomas~M.}\binits{T.~M.}}
(\byear{2005}).
\btitle{Interacting Particle Systems}.
\bpublisher{Springer},
\blocation{Berlin}.
\bnote{Reprint of the 1985 original}.
\bid{mr={2108619}}
\end{bbook}
%
\bptok{imsref}%
\endbibitem

\bibitem{Cas1}
%
\begin{barticle}[author]
\bauthor{\bsnm{McCaskill},~\bfnm{John}\binits{J.}}
(\byear{1984}).
\btitle{A stochastic theory of macromolecular evolution}.
\bjournal{Biol. Cybernet.}
\bvolume{50}
\bpages{63--73}.
\end{barticle}
%
\bptok{imsref}%
\endbibitem

\bibitem{MI}
%
\begin{barticle}[mr]
\bauthor{\bsnm{Mezi{\'c}},~\bfnm{Igor}\binits{I.}}
(\byear{1997}).
\btitle{FKG inequalities in cellular automata and coupled map lattices}.
\bjournal{Phys. D}
\bvolume{103}
\bpages{491--504}.
\bid{doi={10.1016/S0167-2789(96)00281-3}, issn={0167-2789}, mr={1464259}}
\end{barticle}
%
\bptok{imsref}%
\endbibitem

\bibitem{MS}
%
\begin{barticle}[mr]
\bauthor{\bsnm{Morrow},~\bfnm{Gregory~J.}\binits{G.~J.}} \AND
\bauthor{\bsnm{Sawyer},~\bfnm{Stanley}\binits{S.}}
(\byear{1989}).
\btitle{Large deviation results for a class of {M}arkov chains arising
from population genetics}.
\bjournal{Ann. Probab.}
\bvolume{17}
\bpages{1124--1146}.
\bid{issn={0091-1798}, mr={1009448}}
\end{barticle}
%
\bptok{imsref}%
\endbibitem

\bibitem{MUS}
%
\begin{barticle}[mr]
\bauthor{\bsnm{Musso},~\bfnm{Fabio}\binits{F.}}
(\byear{2011}).
\btitle{A stochastic version of the {E}igen model}.
\bjournal{Bull. Math. Biol.}
\bvolume{73}
\bpages{151--180}.
\bid{doi={10.1007/s11538-010-9525-4}, issn={0092-8240}, mr={2770281}}
\end{barticle}
%
\bptok{imsref}%
\endbibitem

\bibitem{NS}
%
\begin{barticle}[author]
\bauthor{\bsnm{Nowak},~\bfnm{Martin~A.}\binits{M.~A.}} \AND
\bauthor{\bsnm{Schuster},~\bfnm{Peter}\binits{P.}}
(\byear{1989}).
\btitle{Error thresholds of replication in finite populations.
{M}utation frequencies and the onset of {M}uller's ratchet.}
\bjournal{J. Theoret. Biol.}
\bvolume{137}
\bpages{375--395}.
\end{barticle}
%
\bptok{imsref}%
\endbibitem

\bibitem{OCH}
%
\begin{bmisc}[author]
\bauthor{\bsnm{Ochoa},~\bfnm{Gabriela}\binits{G.}}
(\byear{2001}).
\bhowpublished{Error thresholds and optimal mutation rates in genetic
algorithms.
Ph.D. thesis, Univ. Sussex, Brighton.}
\end{bmisc}
%
\bptok{imsref}%
\endbibitem

\bibitem{PEM}
%
\begin{barticle}[author]
\bauthor{\bsnm{Park},~\bfnm{Jeong-Man}\binits{J.-M.}},
\bauthor{\bsnm{Mu{\~n}oz},~\bfnm{Enrique}\binits{E.}} \AND
\bauthor{\bsnm{Deem},~\bfnm{Michael~W.}\binits{M.~W.}}
(\byear{2010}).
\btitle{Quasispecies theory for finite populations}.
\bjournal{Phys. Rev. E}
\bvolume{81}
\bpages{011902}.
\end{barticle}
%
\bptok{imsref}%
\endbibitem

\bibitem{SAA1}
%
\begin{barticle}[author]
\bauthor{\bsnm{Saakian},~\bfnm{David~B.}\binits{D.~B.}},
\bauthor{\bsnm{Deem},~\bfnm{Michael~W.}\binits{M.~W.}} \AND
\bauthor{\bsnm{Hu},~\bfnm{Chin-Kun}\binits{C.-K.}}
(\byear{2012}).
\btitle{Finite population size effects in quasispecies models with
single-peak fitness landscape}.
\bjournal{Europhys. Lett.}
\bvolume{98}
\bpages{18001}.
\end{barticle}
%
\bptok{imsref}%
\endbibitem

\bibitem{WE}
%
\begin{bmisc}[author]
\bauthor{\bsnm{Weinberger},~\bfnm{Edward~D.}\binits{E.~D.}}
(\byear{1987}).
\bhowpublished{A~stochastic generalization of Eigen's theory of
natural selection.
Ph.D. dissertation, The Courant Institute of Mathematical Sciences,
New York Univ., New York.}
\end{bmisc}
%
\bptok{imsref}%
\endbibitem

\end{thebibliography}
\end{document}